\documentclass{amsart}
\usepackage{tikz}
\usepackage{xcolor}
\usepackage{amssymb,latexsym,amsmath,extarrows}
\usepackage{graphicx,mathrsfs,comment}
\usepackage{hyperref,url}

\usepackage{amstext}
\usepackage{bbm}

\numberwithin{equation}{section}

\newtheorem{theorem}{Theorem}[section]
\newtheorem{lemma}[theorem]{Lemma}

\newtheorem{proposition}[theorem]{Proposition}
\newtheorem{remark}[theorem]{Remark}

\newtheorem{definition}[theorem]{Definition}

\newtheorem{algorithm}[theorem]{Algorithm}

\newcommand{\al}{\alpha}
\newcommand{\be}{\beta}

\newcommand{\Ga}{\Gamma}
\newcommand{\de}{\delta}

\newcommand{\De}{\Delta}
\newcommand{\e}{\varepsilon}

\newcommand{\ka}{\kappa}
\newcommand{\la}{\lambda}

\newcommand{\si}{\sigma}
\newcommand{\Si}{\Sigma}
\newcommand{\vp}{\varphi}
\newcommand{\om}{\omega}
\newcommand{\Om}{\Omega}

\newcommand{\cq}{\mathcal Q}
\newcommand{\cp}{\mathcal P}

\newcommand{\co}{\mathcal O}

\newcommand{\cb}{\mathcal B}

\newcommand{\cl}{\mathcal L}

\newcommand{\wt}{\widetilde}
\newcommand{\wh}{\widehat}

\newcommand{\ZR}{\mathbb{R}}
\newcommand{\ZT}{\mathbb{T}}
\newcommand{\ZZ}{\mathbb{Z}}

\newcommand{\ZB}{\mathbb{B}}

\newcommand{\ZN}{\mathbb{N}}

\newcommand{\ZS}{\mathbb{S}}

\newcommand{\Id}{{\textit{1}}}

\newcommand{\Tau}{\mathcal{T}}

\newcommand{\cB}{{\mathcal B}}
\newcommand{\cT}{{\mathcal T}}
\newcommand{\cQ}{{\mathcal Q}}

\newcommand{\bq}{{\bf q}}

\newcommand{\BR}{{\rm{Br}}}
\newcommand{\R}{\mathbb{R}}

\newcommand{\Br}{{\rm{Br}}}

\newcommand{\T}{\mathbb{T}  }

\newcommand{\1}{\mathbbm{1}}

\newcommand{\cO}{\mathcal{O}}

\newcommand{\bSq}{{{\bf Sq}}}
\newcommand{\Sq}{{\rm{Sq}}}
\newcommand{\bil}{{\rm{Bil}}}
\newcommand{\vf}{\boldsymbol f}
\newcommand{\vg}{\boldsymbol g}
\newcommand{\vh}{\boldsymbol h}
\newcommand{\vn}{\vec{n}}
\newcommand{\rap}{{\rm RapDec}}
\newcommand{\dist}{{\rm dist}}
\newcommand{\dil}{{\rm Dil}}
\newcommand{\poly}{\rm{Poly}}
\def\l{\ell}

\begin{document}

\title{New bounds for Stein's square function in $\ZR^3$.}

\author{Shengwen Gan}
\address{Shengwen Gan\\ 
Deparment of Mathematics, Massachusetts Institute of Technology, Cambridge MA,  02139, USA}
\email{shengwen@mit.edu}

\author{Yifan Jing}
\address{
Yifan Jing\\
Department of Mathematics, University of Illinois at Urbana-Champaign, Urbana IL, 61801, USA}
\email{yifanjing17@gmail.com}

\author{Shukun Wu}

\address{
Shukun Wu\\
Department of Mathematics\\
University of Illinois at Urbana-Champaign\\
Urbana, IL, 61801, USA}

\email{shukunw2@illinois.edu}
\date{}
\begin{abstract}
We prove new $L^p(\ZR^3)$ bounds on Stein's square function for $p\geq3.25$. As an application, it improves the maximal Bochner-Riesz conjecture to the same range of $p$. The main method we used is polynomial partitioning.

\end{abstract}

\maketitle


\section{Introduction}

In this paper, the ambient space is always $\ZR^3$. The purpose of this paper is to prove some $L^p$ estimate for Stein's square function. To begin with, let us recall the Bochner-Riesz mean of order $\la$:
\begin{equation}
    T_t^\lambda f(x):=\int_{\mathbb{R}^3} \Big(1-\frac{|\xi|^2}{t^2}\Big)^\lambda_+\wh{f}(\xi)e^{ix\cdot\xi}d\xi\,,
\end{equation}
where $f$ is a Schwartz function. Stein \cite{Stein-BR} (See also \cite{Stein-Weiss} Chapter 7) introduced a square function $G^\la f$ defined as
\begin{equation}
    G^\la f(x):=\Big(\int_0^\infty\Big|\frac{\partial}{\partial t}T^\la_t f(x)\Big|^2tdt\Big)^{1/2}.
\end{equation}

Our main result is the following.
\begin{theorem}[Stein's square function]
\label{sqfcn-thm-intro}
Assume $p\geq3.25$. Then 
\begin{equation}
\label{sqfcn-esti-intro}
    \|G^\la f\|_p\leq C_p\|f\|_p,\hspace{1cm}\la>3\Big(\frac{1}{2}-\frac{1}{p}\Big).
\end{equation}
\end{theorem}
\noindent An immediate application of Theorem \ref{sqfcn-thm-intro} is improvements on the maximal Bochner-Riesz problems (See \cite{Carbery-MBR} for instance). Recall that the maximal Bochner-Riesz operator $T^\la_\ast$ is defined as
\begin{equation}
\label{MBR-operator}
    T^{ \lambda}_\ast f(x)=\sup_{t>0}|T_t^\la f(x)|.
\end{equation}
\begin{theorem}[Maximal Bochner-Riesz]
\label{MBR}
Assume $p\geq3.25$. Then
\begin{equation}
\label{MBR-esti-intro}
    \|T^\la_\ast f\|_p\leq C_{p}\|f\|_p, \hspace{1cm}\la>1-3/p.
\end{equation}
\end{theorem}

The two dimensional square function problem and the maximal Bochner-Riesz problem were solved by Carbery \cite{Carbery-MBR}. In $\ZR^3$, the previous best results for Theorem \ref{sqfcn-thm-intro} and Theorem \ref{MBR} are both $p\geq 10/3$, which were obtained respectively by Lee, Rogers, Seeger in \cite{Lee-Rogers-Seeger} and Lee in \cite{Lee-BR}. We remark that the range of $\la$ in both theorem is sharp. See \cite{Lee-Rogers-Seeger} for some discussions.

Both the square function $G^\la f$ and the maximal function $T^\la_\ast f$ are closely related to the almost everywhere convergence of the Bochner-Riesz mean. Unlike the Bochner-Riesz operator, the case $p<2$ is largely different from the case $p>2$ for the maximal Bochner-Riesz problem. Here we are only interested in the case $p>2$. For recent results about the maximal operator $T^\la_\ast$ when $p<2$, see \cite{Li-Wu-MBR}.

\medskip

The main method we use to prove Theorem \ref{sqfcn-thm-intro} is the polynomial partitioning method,  which was introduced by Guth and Katz \cite{Guth-Katz}, and later applied by Guth \cite{Guth-R3} to Fourier restriction estimate. The original article \cite{Guth-R3} itself is also a great overview of this method. Roughly speaking, polynomial partitioning allows one to make a uniform partition for the ambient function, and keep some geometric properties for straight lines at the same time.

In the landmark paper \cite{Guth-R3}, Guth used polynomial partitioning to improve the restriction conjecture in $\ZR^3$ to $p>3.25$. This result was later improved to $p>13/3$ by Wang \cite{Wang-restriction-R3}. Recently in \cite{Wu} and \cite{GOWWZ}, the authors were able to improve the three-dimensional Bochner-Riesz conjecture to the endpoint $p\geq3.25$, though the methods used in \cite{Wu} and \cite{GOWWZ} are quite different. In \cite{GOWWZ}, one used the Carleson-Sj\"{o}lin reduction to convert the Bochner-Riesz problem to an analogue of the restriction problem. In this way, the techniques developed in \cite{Guth-R3} and \cite{Guth-II} are still applicable.

A key ingredient in \cite{Wu} and in this paper is a backward algorithm. The algorithm allows us to bound a certain broad function using lattice-cube square function (See \eqref{broad-esti-intro}). This is different from the mixed norm $\|\cdot\|_2^\al\|\cdot\|_\infty^{1-\al}$ that was widely used in the restriction problem, for instance, in \cite{Guth-R3} and \cite{Wang-restriction-R3}. In there, the loss when using polynomial partitioning iteratively is partly picked up via the norm $\|\cdot\|_\infty$. Here in this paper, we use the backward algorithm to try to pick up the loss by looking at the norm $\|\cdot\|_2$.

\medskip

The proof of our main estimate \eqref{sqfcn-esti-intro} is quite technical and involved, so we give a quick sketch of proof to help the readers to digest. The argument here is more intuitive and less rigorous, as we want to avoid messy notations.

We first reduce the square function $G^\la f$ to a curved, congruent, Littlewood-Paley-type of square function. Let $R\geq1$ be a large integer and let $\zeta:[-1,1]\to \ZR$ be a smooth function. For each $j=1,2,\ldots,R$, set $t_j:=1+j/R$ and define a Fourier multiplier $\wh\zeta_j:\ZR^3\to\ZR$ as $\wh\zeta_j(\xi):=\zeta(R|\xi|-Rt_j)$. Then each $\wh\zeta_j$ is supported in a thin neighbourhood of the sphere of radius $t_j$, which we denote by $\Ga_j$. We introduce a spherical square function
\begin{equation}
\label{sqfcn-intro}
    {\bf Sq}f:=\Big(\sum_{j=1}^R|\zeta_j\ast f|^2\Big)^{1/2}.
\end{equation}
After several standard reductions, to get \eqref{sqfcn-esti-intro}, one only needs to prove for any $\e>0$ and $p\geq3.25$,
\begin{equation}
\label{sqfcn-main}
    \|{\bf Sq}f\|_{L^p(B_R)}\leq C_\e R^{\frac{p-3}{p}+\e}\|f\|_p.
\end{equation}

Now let us focus on \eqref{sqfcn-main}. Set $K:=R^{\e^{10}}$ and let $\cT=\{\tau\}$ be a collection of $K^{-1}$ caps in $\ZS^2$. We use the pullbacks of the Gauss maps of the spheres $\{|\xi|=t_j\}_{1\leq j\leq R}$ to define a broad part $\Br\bSq f $ and a narrow part $\{\bSq_\tau f \}_{\tau\in\cT}$. Each narrow square function $\bSq_\tau  f $ is defined as
\begin{equation}
     {\bSq}_\tau f:=\Big(\sum_{j=1}^R|\zeta_{j,\tau}\ast f|^2\Big)^{1/2},
\end{equation}
where the Fourier transform of each kernel $\zeta_{j,\tau}$ is supported in a thin neighbourhood of the portion $\Ga_{j}(\tau)\subset \Ga_j=\{|\xi|=t_j\}$ so that all normal vectors of $\Ga_{j}(\tau)$ belong to the $K^{-1}$-directional cap $\tau$. 
Recall that a main step in the broad-narrow argument is parabolic rescaling. More precisely, one can find a linear map $\cl$ that maps the $K^{-1}$-cap $\Ga_j(\tau)$ to some $1$-cap $\tilde\Ga_j$. However in our setting,
when we try to run broad-narrow argument as in \cite{Bourgain-Guth-Oscillatory}, there is some technical difficulty: it is hard to find a suitable parabolic rescaling here. As one can see, since we are dealing with many slices $\{\Ga_j\}$ instead of a single slice, it's hard to find a common linear map $\mathcal{L}$ for all the $K^{-1}$-caps $\{\Ga_j(\tau)\}$.

In order to get around this difficulty, we define a second-level broad part and a second-level narrow part associated to each of the first-level narrow square function $\bSq_\tau f $. We iterate this idea until each of the existing narrow square function is restricted to a $R^{-1/2}$-directional cap. Since $K=R^{\e^{10}}$, there are in total $\e^{-10}/2$ levels. After this iteration stops, there are many broad parts at different levels. We estimate each remaining broad part directly and add them up eventually.

Suppose that $\Br\bSq_\si f $ is a broad function at any intermediate level $\l$, so $\si$ is a $K^{-\l}$ directional cap. We set $M=K^\l$. The desired estimate for $\Br\bSq_\si f $ is
\begin{equation}
\label{broad-esti-intro}
    \|\Br\bSq_\si f \|_p\leq C_\e R^{\frac{p-3}{p}+\e} M^{\frac{6-2p}{p}}\Big\|\Big(\sum_{q}|\De_q f |^2\Big)^{1/2}\Big\|_p.
\end{equation} 
The right hand side is an auxiliary square function, where $\bq=\{q\}$ is a tiling in the frequency space $\ZR^3$ using $R^{-1/2}$-cubes, and $\De_q f$ is a smooth Fourier restriction of $f$ on the $R^{-1/2}$-cube $q$. Once we prove \eqref{broad-esti-intro}, we can use Littlewood-Paley theorem for congruent cubes to sum up all broad quantities $\|\Br\bSq_\si f \|_p^p$ at all levels. 

To obtain \eqref{broad-esti-intro}, we use the polynomial partitioning iteration to break $\Br\bSq_\si f$ down. Specifically, suppose that there are in total $s$ steps in the iteration. In each intermediate step $u$, we have a collection of step $u$ cells $\co_u=\{O_u\}$ and step $u$ broad functions $\Br(\bSq_\si f)_{O_u}$. All $\|\Br (\bSq_\si f)_{O_u}\|_p^p$ are morally the same, and their sum dominate a significant portion of the original one $\|\Br \bSq_\si f\|_p^p$. In the final step $s$, the $L^p$ norm of $\Br (\bSq_\si f)_{O_s}$ can be bounded by the $L^2$ norm of $ (\bSq_\si f)_{O_s}$ via a bilinear argument. So it remains to sum up $\| (\bSq_\si f)_{O_s}\|_2^2$.

Here is the key ingredient of our proof. Assume that there are in total $s_t$ transverse steps and $s_c$ cell steps in the polynomial partitioning iteration. Using the information from the iteration, we can add up all $\| (\bSq_\si f)_{O_s}\|_2^2$ as
\begin{equation}
    \sum_{O_s\in\co_s}\| (\bSq_\si f)_{O_s}\|^2_2\lesssim R^{O(\e^2)}R^{-s_t\e^2}\| \bSq_\si f\|^2_2.
\end{equation}
While this estimate is not strong enough. We refine this estimate by proving 
\begin{equation}
\label{refinement}
    \sum_{O_s\in\co_s}\| (\bSq_\si f)_{O_s}\|^2_2\lesssim R^{O(\e^2)}R^{-s_t\de}\ka\| \Id_X\bSq_\si f\|^2_2,
\end{equation}
where the factor $\ka\lesssim1$ depends heavily on the set $X$. In \eqref{refinement}, we either have gain $\ka$ in the $L^2$ estimate, or have gain $\Id_X$ in the support, which will be useful when converting $\| \Id_X\bSq_\si f\|^2_2$ to the square function in \eqref{broad-esti-intro}. The proof of our refinement is based on a backward algorithm built up step by step from the last step $s$ to the first step of the polynomial partitioning iteration. Inside the backward algorithm, the key is an estimate in incidence geometry, which helps us create a relation between tubes at adjacent scales. Since more notations are needed to clarify the backward algorithm, we stop here and leave the details to the main body of the paper, in particular in Section \ref{backward}.

\medskip

\noindent {\bf Organization of the paper.} The first part of Section 2 contains some classic result in Harmonic analysis that we will use later. The second part of Section 2, together with Section 3, contain several reductions for our main result---Theorem \ref{sqfcn-thm-intro}. It is reduced to Theorem \ref{local-broad-thm}, which is also our strongest result. Section 4 contains the wave packet decomposition. In Section 5 and 6, we build the polynomial partitioning iteration for some vectors, where the iteration is similar to the one in \cite{Wu}. After that, we prove our main result for some special cases in Section 7. In Section 8 and 9, similar to \cite{Wu}, we establish a backward algorithm and prove our main result for the remaining cases. Section 10 is devoted to a bilinear estimate for some square functions.

\medskip

\noindent {\bf Notations.}

\noindent$\bullet$  We let $a\sim b$ mean that $ca\leq b\leq Ca$ for some unimportant constants $c$ and $C$. We also use $a\lesssim b$ to represent $a\leq Cb$ for an unimportant constant $C$. These constants may change from line to line.

\noindent$\bullet$ We write $A(R) \leq \mathrm{RapDec}(R)B$ to mean that for any power $\beta$, there is a constant $C_{N}$ such that
\begin{equation}
\nonumber
    A(R) \leq C_{N}R^{-N}B \;\; \text{for all $R \geq 1$}.
\end{equation}

\noindent$\bullet$ For every number $R>0$ and set $S$, we denote by $N_{R}(S)$ the $R$-neighborhood of the set $S$.

\noindent$\bullet$ The symbol $B^3(x,r)$ represents the open ball centered at $x$, of radius $r$, in $\ZR^3$. The symbol $B_r$ is reserved for $B^3(0,r)$. $M,N,C$ are (big) constants that may change from line to line.

\noindent$\bullet$ We use $\Id_X$ to denote the characteristic function of a set $X\subset\ZR^3$.

\noindent$\bullet$ By saying a $r$-tube, we mean a rectangle of dimensions $r^{1/2}\times r^{1/2}\times r$.

\noindent$\bullet$ By saying a $r$-cap, we mean a cap of radius $r$ in $\ZS^2$.

\medskip

\noindent
{\bf {Some numbers.}} 
We will encounter many different numbers in the paper. For readers' convenience, we summarize all of them here. Note that here we only give a rough description of these numbers. The precise definition will be given later in the paper.

\noindent$\bullet$ $\e$ is a very small number.

\noindent$\bullet$ $\de=\e^2$.

\noindent$\bullet$ $\beta=\e^{1000}$ is the number used to deal with the rapid-decay tail of wave packet.

\noindent$\bullet$ $K\sim\log R$ is the number appearing in the broad-narrow argument.

\noindent$\bullet$ $A\sim \log\log R$ is the number in the subscript of the broad function $\Br_A f$. Also, we would like $A$ to be of form $2^{\ZN}$.

\noindent$\bullet$ $d=R^{\e^6}$ is the degree of the polynomial in the polynomial partitioning argument.

\medskip

\noindent {\bf Acknowledgement.} The authors would like to thank Larry Guth, Shaoming Guo, Xiaochun Li and Andreas Seeger for valuable discussions.

\section{Preliminaries and a first reduction}

In this section, we first review some classic results in Fourier analysis that we will repeatedly use later in the paper. Then we make our first reduction to the square function $G^\la f$.

\subsection{A local \texorpdfstring{$L^2$} estimate}\hfill

Let $\Ga=\{(\bar\xi,\Psi(\bar\xi)):\bar\xi\in\overline{B^{2}(0,1)}\}$, be the truncated graph of a function $\Psi:\ZR^{2}\to\ZR$ with bounded Hessian. Suppose that there is a vector $e\in \ZS^2$ such that all the normal vectors of $\Ga$ make an angle of $\leq 1$ with respect to $e$. The following lemma is a local $L^2$ estimate for functions whose Fourier transforms are supported in a small neighborhood of $\Ga$. 
\begin{lemma}\label{localL2}
Assume $e=e_3$, the vertical unit vector. Let $f$ be an $L^2$ function such that $\wh{f}$ is supported in $N_{\rho}(\Ga)$, where $\rho$ is a positive number much smaller than 1. Let $1\geq\si\geq\rho$ and let $\cp\subset\ZR^n$ be a tube of length $\si^{-1}$ and radius $r>1$ in the physical space, pointing to the direction $e$. Suppose that $\vp_\cp$ is a smooth cutoff function with respect to $\cp$ that $|\vp_B(x)|\gtrsim 1$ on $B$, $\wh\vp_B$ is supported in a dual slab $B^2(0,3r^{-1})\times[-3\si,3\si]$ and $|\wh\vp_B(\xi)|\gtrsim 1$ in a smaller slab $B^2(0,r^{-1})\times[-\si,\si]$. Then
\begin{equation}
\label{local-l2}
    \|f\vp_B\|_2\lesssim\rho^{1/2}\si^{-1/2}\|f\|_2.
\end{equation}
\end{lemma}

\begin{proof}
The proof is similar to the one in  \cite{Wu}, dealing with the simple case $\si=r$. We omit details here.
\end{proof}

\subsection{A weighted estimate for pseudo-congruent square functions}\hfill

The weighted estimate we are going to prove in this subsection is similar to the result in \cite{Cordoba-LP}. Let $T\subset \ZR^3$ be a fixed rectangular parallelepiped of dimensions $\rho_1\times \rho_2\times \rho_3$ centered at the origin, and let $\{e_1,e_2,e_3\}$ be the corresponding orthonormal basis. For any positive integer $N$, we define two weights 
\begin{align}
\label{weight}
    w_{T,N}(x)&=(1+|x_1/\rho_1|+|x_2/\rho_2|+|x_3/\rho_3|)^{-N}\\
    \wt w_{T,N}(x)&=\frac{1}{|T|}(1+|x_1/\rho_1|+|x_2/\rho_2|+|x_3/\rho_3|)^{-N}.
\end{align}
Using the weight $\wt w_{T,N}$ we can introduce the definition for a class of smooth, rapidly decreasing functions associated to $T$.
\begin{definition}
\label{adapt-def}
We say a finite collection of smooth functions $\{\vp_j\}$ is ``adapted to $T$", if $|\vp_j(x)|\leq C_N \wt w_{T,N}(x)$ uniformly in $j$ for any $N\in\ZN$. 
\end{definition}

Next, in the frequency space, we consider a family of finitely overlapping ``pseudo-congruent" rectangular parallelepipeds $\Om=\{\om\}$ with dimensions $\rho_1^{-1}\times \rho_2^{-1}\times \rho_3^{-1}$. By pseudo-congruent we mean that for any two $\om_1,\om_2\in\Om$, $\om_1\subset C\wt\om_2$, where $\wt\om_2$ is the parallelepiped congruent to $\om_2$ and $\wt\om_2$ has the same center as $\om_1$ does. Assume that for each $\om\in\Om$ there is a smooth function $\wh\vp_\om$ such that the collection $\{\vp_\om\}$ is adapt to a rectangular parallelepiped $T$ centered at the origin. As usual, we let $\{e_1,e_2,e_3\}$ be the orthonormal basis associated to $T$. Now define $\De_\om f:=(\wh\vp_\om\wh f)^{\vee}$ as the smooth Fourier restriction of $f$ on $\om$. Then, we have
\begin{lemma}
\label{weightedlemma}
For any positive measurable function $g$,
\begin{equation}
\label{weighted-L2}
    \sum_{\om\in\Om}\int |\De_\om f|^2g\leq C_N \int |f|^2(\wt w_{T,N}\ast g).
\end{equation}
\end{lemma}
\begin{proof}
Following the idea in \cite{Cordoba-LP}, we are going to use the Poisson summation formula. First, let us introduce some more notations. Let $\theta_T$ the be rectangular parallelepiped dual to $T$ centered at the origin, and let $\theta=C\theta_T$. We choose the absolute constant $C$ large enough so the following is true: If we consider the lattice translations
\begin{equation}
    \{\theta_m:\theta_m:=\theta+(m_1\rho_1^{-1},m_2\rho_2^{-1},m_3\rho_3^{-1}), m\in\ZZ^3\},
\end{equation}
then for each $\om\in\Om$, there is a lattice point $m_\om\in\ZZ^3$ such that $2\om\subset\theta_{m_\om}$. Indeed, we can pick $m_\om\in\rho_1^{-1}\ZZ\times\rho_2^{-1}\ZZ\times\rho_3^{-1}\ZZ$ be an arbitrary lattice point which also belongs to $2\om$. Since parallelepipeds in $\Om$ are finitely overlapping, we have for every $\om_1\in\Om$, all but finitely many $\om_2\in\Om$ satisfy $m_{\om_1}\not=m_{\om_2}$. Let $\wh\phi$ be a bump function associated to $\theta$ so that $\phi$ is dominated by $w_{T,N}$, and let $\wh\phi_m(\xi):=\wh\phi(\xi+m)$ for every translate $m\in\ZZ^3$. Hence $\wh\phi_{m_\om}\wh{\De_\om f}=\wh \De_\om f$, so
\begin{equation}
    |\De_\om f(x)|^2=|\vp_\om\ast(\phi_{m_\om}\ast f)(x)|^2\lesssim (|(\phi_{m_\om}\ast f)|^2\ast|\vp_\om|)(x).
\end{equation}
Since $|\vp_\om|$ is dominated by the weight $\wt w_{T,N}$, and since $m_{\om_1}\not=m_{\om_2}$ for most pairs $(\om_1,\om_2)$, one has
\begin{equation}
\label{weighted-1}
    \sum_{\om\in\Om}\int |\De_\om f|^2g\lesssim\sum_{\om\in\Om}\int |(\phi_{m_\om}\ast f)|^2(\wt w_{T,N}\ast g)\lesssim\sum_{m\in\ZZ^3}\int |(\phi_{m}\ast f)|^2(\wt w_{T,N}\ast g).
\end{equation}

Next, we are going to estimate $\sum_m|\phi_m\ast f|^2$ via the Poisson summation formula. Without loss of generality, we assume $\rho_1=\rho_2=\rho_3=1$. Clearly,
\begin{equation}
\label{Fourier-series}
    \sum_{m\in\ZZ^3}|(\phi_{m}\ast f)|^2=\int_{[0,1]^3}\Big|\sum_{m\in\ZZ^3}(\phi_{m}\ast f)e^{2\pi i m\cdot z}\Big|^2dz.
\end{equation}
Notice that $\phi_m(x)=\phi(x)e^{-2\pi im\cdot x}\phi(x)$. By Poisson summation formula (See for instance \cite{Grafakos-249} Section 3.1.5), one has
\begin{align}
    \sum_{m\in\ZZ^3}(\phi_{m}\ast f)e^{2\pi i m\cdot z}&=\sum_{m\in\ZZ^3}e^{2\pi i m\cdot z}\int \phi(y)f(x-y)e^{-2\pi im\cdot y}dy\\[1ex] \nonumber
    &=\sum_{m\in\ZZ^3}\phi(m+z)f(x-m-z).
\end{align}
Plug this back to \eqref{Fourier-series} so that
\begin{equation}
    \sum_{m\in\ZZ^3}|(\phi_{m}\ast f)|^2=\int_{[0,1]^3}\Big|\sum_{m\in\ZZ^3}\phi(m+z)f(x-m-z)\Big|^2dz,
\end{equation}
which can be bounded using H\"older's inequality as
\begin{equation}
    \sum_{m\in\ZZ^3}|(\phi_{m}\ast f)|^2\lesssim\int_{[0,1]^3}\sum_{m\in\ZZ^3}|\phi(m+z)|\cdot|f(x-m-z)|^2dz\lesssim \wt w_{T,N}\ast(|f|^2)(x). 
\end{equation}
Finally, noting $\wt w_{T,N}(x)=\wt w_{T,N}(-x)$ and $\wt w_{T,N}\ast \wt w_{T,N}\lesssim \wt w_{T,N}$, we combine \eqref{weighted-1} and the above estimate to conclude \eqref{weighted-L2}.
\end{proof}

\subsection{A first reduction for the square function \texorpdfstring{$ G^\la f$}{Lg}} \hfill

In the final part of this section, let us make a reduction for $G^\la f$. It consists of several smaller steps, each of which is stated as a lemma.

\medskip 

The first step consists of partition in the frequency space. Suppose that $R\geq1$ is a big number. For simplicity let us assume $R\in\ZN$. Let $\zeta:[-1,1]\to \ZR$ be a smooth function. We introduce a spherical function 
\begin{equation}
    T_{R,t} f(x):=\int_{\mathbb{R}^3} \wh\eta\Big(\frac{1-|\xi/t|}{R^{-1}}\Big)\wh{f}(\xi)e^{ix\cdot\xi}d\xi\,,
\end{equation}
and a spherical square function $G_R f$ as
\begin{equation}
    G_R f(x):=\Big(\int_0^\infty\Big|T_{R,t} f(x)\Big|^2\frac{dt}{t}\Big)^{1/2}.
\end{equation}
It was shown in \cite{Carbery-MBR} that to prove \eqref{sqfcn-esti-intro}, we only need to prove
\begin{equation}
    \|G_R f\|_p\leq C_\e R^{\frac{3}{2}-\frac{3}{p}+\e}\|f\|_p
\end{equation}
for any $\e>0$ and $p\geq3.25$. Using the Proposition 4.2 in \cite{Guo-Roos-Yung} (which is a generalization of \cite{Seeger-black-box}), it suffices to consider the local square function 
\begin{equation}
\label{local-sqfcn}
    H_R f(x):=\Big(\int_1^2\Big|T_{R,t} f(x)\Big|^2\frac{dt}{t}\Big)^{1/2}
\end{equation}
and prove for any $\e>0$ and $p\geq3.25$ that
\begin{equation}
\label{local-sqfcn-esti}
    \|H_R f\|_p\leq C_\e R^{\frac{3}{2}-\frac{3}{p}+\e}\|f\|_p
\end{equation}
We put the argument above into a lemma
\begin{lemma}[First step]
To prove \eqref{sqfcn-esti-intro}, one only needs to show \eqref{local-sqfcn-esti}.
\end{lemma}

Our next step is to discretize the square function $H_R f$. This step in fact is not necessary, but we feel that it is better to work on the discrete version of the square function. Let $\Phi(\bar\xi;t)$ be a real function with $\bar\xi\in B^2(0,1)$ and $t\in[1,2]$, and let $\{t_j\}_{1\le j\le R}$ be any collection of $R^{-1}$-separated points in the interval $[1,4]$. For each $j$, define the truncated surface $\Ga_j$ as:
\begin{equation}
\label{surface}
    \Gamma_j:=\{ (\bar{\xi},\Phi(\bar\xi;t_j)):|\bar\xi|\le 1/2 \}.
\end{equation}
We need to impose certain regularity conditions and curvature condition on the function $\Phi$. Specifically, assume that when $t\in[1,2]$ and $\bar\xi\in B^2(0,1/2)$, $\Phi$ satisfies:
\begin{equation}
\label{Phi-conditions}
\begin{cases}
    |\partial_{t}\Phi|\sim1,~|\partial_t\nabla_{\bar\xi}\Phi|=O(1),\\[1ex]
    \textup{Both~two~eigenvalues~of~}\nabla^2_{\bar\xi}\Phi~\textup{are}\sim -1, ~{\textup{uniform ~in~}} t.
\end{cases}
\end{equation}
The conditions in the first line basically says that any $\Ga_j$ is roughly a translation copy of another $\Ga_{j'}$ when $j$ and $j'$
are close enough. In particular, if $|j-j'|\le R^{1/2}$ and $|\bar\xi_1-\bar\xi_2|\le R^{-1/2}$, we have $|G_j(\bar\xi_1)-G_{j'}(\bar\xi_2)|\lesssim R^{-1/2}$, where $G_j$ is the Gauss map of $\Ga_j$. The condition in the second line means that the surfaces $\{\Ga_j\}_{1\leq j\leq R}$ are all parabolic. Here are two good examples for $\Phi(\bar\xi;t)$:
\begin{enumerate}
    \item $\Phi(\bar\xi;t)=(t-|\bar\xi|^2)^{1/2}$, a family of of spheres;
    \item $\Phi(\bar\xi,t)=t-|\bar\xi|^2$, a family of translated paraboloids.
\end{enumerate}
For each $j$, let $\eta_j:B^2(0,1/2)\to\ZR$ and let $\psi_j:[-1,1]\to\ZR$ be two smooth functions. Suppose that $\phi:\ZR^3\times[0,1]\to\ZR$ is an arbitrary smooth function that $|\partial^\al\phi|\leq C_{\al}$ for any multi-index $\al\in\ZN^4$. Define a general Fourier multiplier
\begin{equation}
\label{kernel-1}
    \wh{m}_j(\xi):=\eta_j(\bar\xi)\psi_j\Big(\frac{\xi_3-\Phi(\bar\xi;t_j)}{R^{-1}}\phi(\xi;t)\Big)
\end{equation}
and the associated operator 
\begin{equation}
\label{operator-1}
    S_j f(x):={m}_j\ast f(x).
\end{equation}

\begin{remark}
Heuristically, one can think of $\phi(\xi;t)=1$. But for generality, we still keep $\phi(\xi;t)$ in our later argument.
\end{remark}
Now we can introduce the discrete square function associated to the general surfaces $\{\Ga_j\}_{1\leq j\leq R}$ that we defined in \eqref{surface}.

\begin{definition}[Square function]
\label{defsqfunction}
For any Schwartz function $f$, define
\begin{equation}
\label{sqfcn-1}
    \Sq f(x):=\Big(\sum_{j=1}^R|S_jf(x)|^2\Big)^{1/2}.
\end{equation}
\end{definition}

\begin{lemma}[Second step]
Let $\Sq f$ be defined in \eqref{sqfcn-1}. Suppose that for any $\e>0$ and $p\geq3.25$,
\begin{equation}
\label{discrete-sqfcn}
    \|\Sq f\|_p\leq C_\e R^{\frac{p-3}{p}+\e}\|f\|_p.
\end{equation}
Then \eqref{local-sqfcn-esti} is true.
\end{lemma}
\begin{proof}
Using polar coordinate and a smooth partition of unity for $\ZS^2$, we can partition the whole space $\ZR^3$ into smaller conic region, so that any two vectors in a single conic region make an angle $\leq1/100$. Note that the Fourier multiplier $\zeta\big(R(1-|\xi/t|)\big)$ is a radial function. To prove \eqref{local-sqfcn-esti}, we can apply the aforementioned partition to the multiplier so that by rotational symmetry, we only need to consider the operator $\wt S_{R,t}$ defined as
\begin{equation}
    \wt S_{R,t} f(x):=\int_{\mathbb{R}^3}\wh\vp_a(\xi) \zeta\Big(\frac{1-|\xi/t|}{R^{-1}}\Big)\wh{f}(\xi)e^{ix\cdot\xi}d\xi.
\end{equation}
Here $\wh\vp_a=\vp(\xi_1/\xi_3)\vp(\xi_2/\xi_3)$ is a smooth function such that $\vp$ supported in $[-1/10,1/10]$. Note that in \eqref{local-sqfcn}, the variable $t$ only ranges in $[1,2]$. We can freely add two smooth cutoff functions $\wh\psi(|\xi|)$, $\eta(\bar\xi)$ in the Fourier multiplier of $S_{R,t}$, where $\wh\psi$ is supported in $[1/10,10]$ and $\wh\psi=1$ on $[1/9,9]$; $\eta$ is supported in the ball $B^2(0,1/2)$. Namely, 
\begin{equation}
    \wt S_{R,t} f(x)=\int_{\mathbb{R}^3}\wh\vp_a(\xi)\wh\psi(|\xi|)\eta(\bar\xi) \zeta\Big(\frac{1-|\xi/t|}{R^{-1}}\Big)\wh{f}(\xi)e^{ix\cdot\xi}d\xi.
\end{equation}
Since the operator $\vp_a\ast\psi(D)$ is of strong-type $(p,p)$ for any $1\leq p\leq \infty$, one can reduce $\wt S_{R,t}$ to a similar operator $S_{R,t}$ defined as
\begin{equation}
    S_{R,t} f(x)=\int_{\mathbb{R}^3}\eta(\bar\xi) \zeta\Big(\frac{1-|\xi/t|}{R^{-1}}\Big)\wh{f}(\xi)e^{ix\cdot\xi}d\xi.
\end{equation}
Hence to show \eqref{local-sqfcn-esti}, it suffices to prove that for any $\e>0$ and $p\geq3.25$,
\begin{equation}
\label{local-sqfcn-reduced-esti}
    \|S_{R,t}f\|_{L^p_xL^2_t([1,2])}\leq C_\e R^{\frac{3}{2}-\frac{3}{p}+\e}\|f\|_p.
\end{equation}

To express the Fourier multiplier of $S_{R,t}$ via something similar to \eqref{kernel-1}, we take
\begin{equation}
    \Phi(\bar\xi;t)=(t^2-|\bar\xi|^2)^{1/2}, \hspace{.5cm} \phi(\xi;t)=\frac{\xi_3+(t^2-|\bar\xi|^2)}{t(t+|\xi|)},
\end{equation}
so that one has
\begin{equation}
    S_{R,t} f(x)=\int_{\mathbb{R}^3}\eta(\bar\xi) \zeta\Big(\frac{-\xi_3+\Phi(\bar\xi;t)}{R^{-1}}\phi(\xi;t)\Big)\wh{f}(\xi)e^{ix\cdot\xi}d\xi.
\end{equation}
The function $\phi$ is smooth in the support of $\wh{S_{R,t}f}$, and $|\partial^\al\phi|\leq C_{\al}$ for any multi-index $\al\in\ZN^4$. Hence one can deduce \eqref{local-sqfcn-reduced-esti} from \eqref{discrete-sqfcn} via discretizaing the variable $t$. We omit the details.
\end{proof}

In the final step, we break the discrete square function $\Sq f$ in the physical space. For any point $m\in R^{-1/2}\ZZ^3$, we choose a $CR^{-1/2}$-cube $q_m$ for some absolute big constant $C$ only depending on the function $\Phi$. The constant $C$ will be determined later in the proof of Lemma \ref{half-lem}. Denote by the collection of these $CR^{-1/2}$ cube by $\bf q$. Let $\wh\vp:B^3(0,2)\to[0,1]$ be a bump function that $\wh\vp=1$ in the unit ball $B^3(0,1)$. For each $q_m\in\bf q$, we define $\vp_{q_m}:=\vp(2CR^{1/2}(x-m))$ and 
\begin{equation}
\label{lattice-sqfcn}
    \De_q f:=\vp_q\ast f.
\end{equation} 
The estimate we would like to prove in the rest of the paper is the following.

\begin{theorem}
\label{sqfcn-local-thm}
Let $\Sq f$ be defined in \eqref{sqfcn-1}. For any $p\geq3.25$ and $\e>0$, one has
\begin{equation}
\label{sqfcn-local-physics}
    \|\Sq f\|_{L^p(B_R)}\leq C_{\e,N} R^{\frac{p-3}{p}+\e}\Big\|\Big(\sum_{q}|\De_q f |^2\Big)^{\frac{1}{2}}\Big\|_{L^p(w_{B_{R^{1+\e}},N})}
\end{equation}
\end{theorem}

\noindent Note that if \eqref{sqfcn-local-physics} is true for $B_R$, it is also true for any $R$-ball in $\ZR^3$. Hence we can raise both sides in \eqref{sqfcn-local-physics} to $p$-th power and sum up all the $R$-balls in $\ZR^3$ to concludes \eqref{discrete-sqfcn}. We also put it into a lemma.
\begin{lemma}[Third step]
Let $\Sq f$ be defined in \eqref{sqfcn-1}. Then \eqref{sqfcn-local-physics} implies \eqref{discrete-sqfcn}.
\end{lemma}

Combining all the lemmas in this subsection, we conclude the first reduction which we state as a proposition below.
\begin{proposition}[First reduction]
\label{first-reduction}
Theorem \ref{sqfcn-local-thm} implies Theorem \ref{sqfcn-thm-intro}.
\end{proposition}

\section{A second reduction: the broad-narrow reduction}

The second reduction is made for the function $\Id_{B_R}\Sq f$ in \eqref{sqfcn-local-physics}. The reduction is slightly different from the original broad-narrow argument in \cite{Bourgain-Guth-Oscillatory}. We begin with the definition of Gauss maps associated to the surfaces $\{\Ga_j\}$, then we introduce general setups of broad-narrowness for a directional cap at any intermediate scale between $1$ and $R^{-1/2}$. After that, we reduce \eqref{sqfcn-local-physics} to a local estimate for broad part, which is our main result. Because of technical reasons, in the end of this section, we define some auxiliary broad functions and prove some properties for them.

\subsection{Gauss maps}
\begin{definition}[Gauss map]
Let $\Phi$ and $\{\Gamma_j\}$ be given in \eqref{surface}. For each $j$, we define the Gauss map of the surface $\Gamma_j$ as:
\begin{equation}
\label{Gauss-map}
    G_j(\bar\xi)=\frac{(\nabla_{\bar\xi}\Phi(\bar\xi;t_j),1)}{|(\nabla_{\bar\xi}\Phi(\bar\xi;t_j),1)|}.
\end{equation}
\end{definition}

\noindent The next two lemmas contain some properties about the Gauss map $G_j$.
\begin{lemma}
\label{Gauss-map-lem}
For any $1\leq j\leq R$, the Gauss map $G_j$ is smooth and injective. In particular, when $\bar\xi\in B^2(0,1/2)$, one has
\begin{equation}
\label{derivative-Gauss}
    |\partial^\al G_j|\leq C_\al,\hspace{5mm}\al\in\ZN^2
\end{equation}
uniformly for all $G_j$.
\end{lemma}
\begin{proof}
Since $\nabla_{\bar\xi}\Phi=O(1)$, one can check \eqref{derivative-Gauss} directly. The fact that $\Phi(\bar\xi;t_j)$ has positive second fundamental form implies that the intersection between the surface $\{\bar\xi\in B^2(0,1):(\bar\xi;\Phi(\bar\xi;t_j))\}$ and an arbitrary plane in $\ZR^3$ is a two dimensional curve with non-vanishing curvature. Thus, we can prove by contradiction that $G_j$ is injective.
\end{proof}

\begin{lemma}
\label{Gauss-map-spt}
Suppose that $\si\subset\ZS^2$ is any $M^{-1}$-cap with $1\le M\le R^{1/2}$. Then for any $1\leq j\leq R$, the set $G_j^{-1}(\si)$ is morally a $M^{-1}$-ball. That is, there is a $M^{-1}$-ball $B_j$ in $\ZR^2$ such that
\begin{equation}
    cB_j\subset G_j^{-1}(\si)\subset CB_j
\end{equation}
for two absolute constants $c<1$ and $C>1$.
\end{lemma}
\begin{proof}
One only need to use the fact that both two eigenvalues of $\nabla^2_{\bar\xi}\Phi$ are $\sim -1$ uniformly in $t$, as mentioned in \eqref{Phi-conditions}. We omit the details.
\end{proof}

\subsection{Broad-narrow reduction}\hfill

Let us first assume that $\si\subset\ZS^2$ is a $M^{-1}$-cap  with $1\le M\le R^{1/2}$, and assume that for each $1\leq j\leq R$, there is a kernel $m_{j,\si}$ with the formula
\begin{equation}
\label{kernel-2}
    \wh{m}_{j,\si}(\xi)=\eta_{j,\si}(\bar\xi)\psi_j\Big(\frac{\xi_3-\Phi(\bar\xi;t_j)}{R^{-1}}\phi(\xi;t)\Big),
\end{equation}
where $\eta_{j,\si}$ is a smooth function satisfying the following properties:
\begin{enumerate}
    \item $\eta_{j,\si}$ is supported in $G_j^{-1}(2\si)$, where $G_j$ is the Gauss map defined in \eqref{Gauss-map}, 
    \item For any multi-index $\al\in\ZN^2$ that $|\al|=N$, $|\partial^\al\eta_{j,\si}|\leq C_NM^{N}$.
\end{enumerate}
Define $S_{j,\si}f$ and the square function $\Sq_\si f$ for the cap $\si$ as
\begin{equation}
\label{sqfcn-2}
    S_{j,\si}f:=m_{j,\si}\ast f,\hspace{1cm}\Sq_\si f(x):=\Big(\sum_{j=1}^R|S_{j,\si}f(x)|^2\Big)^{1/2}.
\end{equation}

In order to utilize the broad-narrow argument, we consider a collection of slightly smaller caps that form a cover of $\si$. Specifically, let $\cT_\si=\{\tau\}$ be a collection $K^{-1}M^{-1}$-caps in $\ZS^2$ that form a cover of $\si$. Define $\wt\vp_\tau:\ZS^2\to\ZR^+$ as a smooth partition of unity associated to this cover so that $\wt\vp_\tau$ is supported in $2\tau$. The smooth cutoff functions $\wt\vp_\tau$ satisfies the following derivatives estimates: for any multi-index $\al\in\ZN^2$ that $|\al|=N$, $|\partial^\al\wt\vp_\tau|\leq C_N(MK)^{N}$. From Lemma \ref{Gauss-map-lem} we know the Gauss map $G_j$ is injective, so the pullback $\vp_\tau=\wt\vp_\tau\circ G_j$ forms a partition of unity of $\Id_{\textup{supp}(\eta_{j,\si})}(\bar\xi)$. Therefore we can partition $m_{j,\si}f$ as
\begin{equation}
\label{frequencey-decom-1}
    m_{j,\si}=\sum_{\tau\in\cT_\si}m_{j,\tau},\hspace{5mm}\wh{m}_{j,\tau}(\xi):=\vp_\tau(\bar\xi)\eta_{j,\si}(\bar\xi)\psi_j\Big(\frac{\xi_3-\Phi(\bar\xi;t_j)}{R^{-1}}\phi(\xi;t)\Big).
\end{equation}
For some technical issues, we further break $\cT_\si$ into 100 disjoint subcollections $\{\cT_{\si,k}\}_{k=1}^{100}$, such that any two caps $\tau,\tau'$ in an arbitrary subset $\cT_{\si,k}$ satisfy $\textup{dist}(\tau,\tau')\geq8(MK)^{-1}$. We define $m_{j,\cT_k}$ as the sum of $m_{j,\tau}$ for $\tau\in\cT_{\si,k}$, and set
\begin{equation}
\label{sqfcn-separated}
    S_{j,\cT_{\si,k}}f:=m_{j,\cT_k}\ast f,\hspace{1cm}\Sq_{\cT_{\si,k}} f(x):=\Big(\sum_{j=1}^R|S_{j,{\cT_{\si,k}}}f(x)|^2\Big)^{1/2}.
\end{equation}
Similar to \eqref{sqfcn-2}, for each $\tau\in\cT_k$ we define
\begin{equation}
\label{sqfcn-3}
    S_{j,\tau}f:=m_{j,\tau}\ast f,\hspace{1cm}\Sq_\tau f(x):=\Big(\sum_{j=1}^R|S_{j,\tau}f(x)|^2\Big)^{1/2}.
\end{equation}
Now we can introduce the broad function $\BR_A\Sq_{\cT_{\si,k}} f$.
\begin{definition}[Broadness]
Given $A\in\ZN$ and any $x\in B_R$, we define $\BR_A\Sq_{\cT_{\si,k}} f(x)$ as the $A+1$-th largest number in $\{\Sq_{\tau} f(x)\}_{\tau\in\cT_{\si,k}}$. That is, 
\begin{equation}
\label{alpha-broad}
    \BR_A\Sq_{\cT_{\si,k}} f(x)=\min_{\tau_1,\ldots,\tau_{A}\in\cT_{\si,k }}\max_{\substack{\tau\not=\tau_l,\\1\leq l\leq A}}\Sq_\tau f(x).
\end{equation}
\end{definition}
\noindent
It is clear from the definition that for any $x\in B_R$
\begin{equation}
\label{broad-narrow}
    \Sq_{\cT_{\si,k}} f(x)\leq K^2\BR_A\Sq_{\cT_{\si,k}} f(x)+ A\sup_{\tau\in\cT_{\si,k}}\Sq_\tau f(x).
\end{equation}
The first term is called the ``broad part" and the second is called the ``narrow part". Taking the $L^p$-norm on both sides and summing up all $k=1,\ldots,100$, we have
\begin{align}
\label{broad-narrow-lp}
    \int_{B_R}\Sq_{\si} f^p & \lesssim\sum_k\int_{B_R}\Sq_{\cT_{\si,k}} f^p\\ \nonumber
    &\lesssim K^2\sum_k\int_{B_R}\BR_A\Sq_{\cT_{\si,k}} f^p+A\sum_k\sum_{\tau\in\cT_{\si,k}}\int_{B_R}\Sq_\tau f^p.
\end{align}

Next, we set $\si_0=\ZS^2$ and define $m_{j,\si_0}:=m_{j}$ and $\Sq_{\si_0} f:=\Sq f$, where $m_j$ and $\Sq f$ were defined in \eqref{kernel-1} and \eqref{operator-1}. We run the above broad-narrow argument for $\si=\si_0$, to have via \eqref{broad-narrow-lp} that
\begin{equation}
\label{broad-narrow-lp-2}
    \int_{B_R}\Sq_{{\si_0}} f^p\lesssim K^2\sum_k\int_{B_R}\BR_A\Sq_{\cT_{\si_0,k}} f^p+A\sum_k\sum_{\si_1\in\cT_{\si_0,k}}\int_{B_R}\Sq_{\si_1} f^p.    
\end{equation}
We keep the broad part in \eqref{broad-narrow-lp-2}, and run the broad-narrow argument for each narrow part to get
\begin{equation}
\label{broad-narrow-lp-3}
    \int_{B_R}\Sq_{ {\si_1}} f^p\lesssim K^2\sum_k\int_{B_R}\BR_A\Sq_{\cT_{\si_1,k}} f^p+A\sum_k\sum_{\si_2\in\cT_{\si_1,k}}\int_{B_R}\Sq_{\si_2} f^p.
\end{equation}
One can repeats this argument for $\sim\log R^{1/2}/\log K$ steps to break down all the narrow parts, except the ones at the final step. Let us conclude this broad-narrow decomposition into a proposition.

\begin{proposition}[Broad-narrow decomposition]
Let $K(K<R)$ be a big number and let $A=[\log K]$. Define $\ka:=\log R^{1/2}/\log K$. Then, we can break down $\|\Id_{B_R}\Sq f\|_p^p$ into
\begin{align}
\label{broad-narrow-final}
    \int_{B_R}\Sq f^p  \lesssim \sum_{\ell=1}^{\ka-1} A^\ell K^2\sum_{\si_\ell\in\Si_\ell}\sum_k\int_{B_R}\BR_A\Sq_{\cT_{\si_\ell,k}} f^p+A^\ka\sum_{\theta\in\Theta}\int_{B_R}\Sq_\theta f^p.
\end{align}
Here $\Si_\ell$ is the collection of finitely overlapping $K^\ell$-caps in $\ZS^2$ that $|c_\si-e_3|\leq1/10$ for any $\si\in\Si_\ell$; $\Theta$ is a collection of finitely overlapping $R^{-1/2}$-caps in $\ZS^2$ that $|c_\theta-e_3|\leq1/10$. The set $\cT_{\si_\ell, k}$ is a collection of $8K^{-\ell-1}$-separated $K^{-\ell-1}$-caps in $\si$ for any $1\leq k\leq100$. Recall \eqref{sqfcn-separated} and \eqref{sqfcn-3} that $\Sq_{\cT_{\si_\ell,k}} f$ is defined as
\begin{equation}
\label{sigma-sqfcn}
    S_{j,\cT_{\si_\ell,k}}f:=m_{j,\cT_{\si_\ell,k}}\ast f,\hspace{1cm}\Sq_{\cT_{\si_\ell,k}} f(x):=\Big(\sum_{j=1}^R|S_{j,{\cT_{\si_\ell,k}}}f(x)|^2\Big)^{1/2},
\end{equation}
where, recalling \eqref{frequencey-decom-1},
\begin{equation}
\label{sigma-kernel}
    m_{j,\cT_{\si_\ell,k}}:=\sum_{\tau\in\cT_{\si_\ell,k}}m_{j,\tau},\hspace{5mm}\wh{m}_{j,\tau}(\xi):=\eta_{j,\tau}(\bar\xi)\psi_j\Big(\frac{\xi_3-\Phi(\bar\xi;t_j)}{R^{-1}}\phi(\xi;t)\Big)
\end{equation}
for a smooth function $\eta_{j,\tau}$ satisfying
\begin{enumerate}
    \item $\eta_{j,\tau}$ is supported in $G_j^{-1}(2\tau)$, where $G_j$ is the Gauss map defined in \eqref{Gauss-map}, 
    \item For any multi-index $\al\in\ZN^2$ that $|\al|=N$, $|\partial^\al\eta_{j,\si}|\leq C_NK^{N(\ell+1)}$.
\end{enumerate}
The function $\Sq_\theta f$ is defined similarly as
\begin{align}
\label{half-sqfcn}
    &S_{j,\theta}f:=m_{j,\theta}\ast f,\hspace{1cm}\Sq_\theta f(x):=\Big(\sum_{j=1}^R|S_{j,\theta}f(x)|^2\Big)^{1/2}, \\ \label{kernel-half}
    &\wh{m}_{j,\theta}(\xi):=\eta_{j,\theta}(\bar\xi)\psi_j\Big(\frac{\xi_3-\Phi(\bar\xi;t_j)}{R^{-1}}\phi(\xi;t)\Big),
\end{align}
where the smooth function $\eta_{j,\theta}$ satisfies
\begin{enumerate}
    \item $\eta_{j,\theta}$ is supported in $G_j^{-1}(2\theta)$, where $G_j$ is the Gauss map defined in \eqref{Gauss-map}, 
    \item For any multi-index $\al\in\ZN^2$ that $|\al|=N$, $|\partial^\al\eta_{j,\si}|\leq C_NR^{N/2}$.
\end{enumerate}
Finally, the broad function $\BR_A\Sq_{\cT_{\si_\ell,k}} f$ was defined \eqref{broad-narrow}.
\end{proposition}
The kernel $m_{j,\theta}$ defined in \eqref{kernel-half} decays rapidly outside a $R^{1/2}\times R^{1/2}\times R$ tube with direction $c_\theta$. In fact, we have:
\begin{lemma}\label{lemdecaykernel}
Let $m_{j,\theta}$ be defined in \eqref{kernel-half}. For any $\be>0$, let $T$ be a $R^{1/2+\be}\times R^{1/2+\be}\times R^{1+\be}$ tube with direction $c_\theta$. Then recalling \eqref{weight}, for any $x\in\ZR^3\setminus T$, one has
\begin{equation}
\label{kernelestimate}
    |m_{j,\theta}(x)|\lesssim\rap(R)w_{T,N}(x).
\end{equation}
The implicit constant in \rap(R) depends on $\be$. In particular, the family of smooth functions $\{R^{-O(\be)}m_{j,\theta}\}$ is adapted to $T$ (Recall Definition \ref{adapt-def}).
\end{lemma}

\begin{proof}
From the definition of $m_{j,\theta}$ in \eqref{half-sqfcn}, we know that one can write $\eta_{j,\theta}(\bar\xi)=\eta(R^{1/2}(\bar\xi-c_{j,\theta}))$ for some smooth function $\eta$ supported in $B^2(0,3/4)$, where the vector $(c_{j,\theta},\Phi(c_{j,\theta}))$ is parallel to $c_\theta$. Hence one has
\begin{equation}
    m_{j,\theta}(x)=\int_{\ZR^3}e^{ix\cdot\xi}\eta(R^{1/2}(\bar\xi-c_{j,\theta}))\psi_j\Big(\frac{\xi_3-\Phi(\bar\xi;t_j)}{R^{-1}}\phi(\xi;t)\Big)d\xi.
\end{equation}
Consider the following change of variables:
\begin{equation}
    \cl_\tau:(\bar{\om},\om_3)=\big(\frac{\bar\xi-{c_{j,\theta}}}{R^{-1/2}},\frac{\xi_3-\nabla\Phi({c_{j,\theta}})\cdot\bar\xi-\Phi({c_{j,\theta}})+\nabla\Phi({c_{j,\theta}})\cdot{c_{j,\theta}}}{R^{-1}}\big).
\end{equation}
which leads to the linear transform
\begin{equation}
    \begin{cases}
    \bar\xi=R^{-1/2}\bar\om+c_{j,\theta}\\
    \xi_3=R^{-1}\om_3+R^{-1/2}\nabla\Phi({c_{j,\theta}})\cdot\bar\om+\Phi({c_{j,\theta}}).
    \end{cases}
\end{equation}
Now let us introduce a new function $\Psi(\bar\om)$ defined as
\begin{equation}
    \Psi(\bar\om):=R[\Phi({c_{j,\theta}}+R^{-1/2}\bar\om)-\Phi({c_{j,\theta}})-R^{-1/2}\nabla\Phi({c_{j,\theta}})\cdot\bar\om],
\end{equation}
so that we can rewrite $|m_{j,\theta}(x)|$ as
\begin{equation}
    |m_{j,\theta}(x)|=R^{-2}\Big|\int_{\ZR^3}e^{iR^{-1/2}(x_3\nabla\Phi(c_{j,\theta})+\bar x)\cdot\bar\om}e^{iR^{-1}x_3\om_3}a(\om;t_j)d\om\Big|.
\end{equation}
Here $a(\om;t_j)=\eta(\bar\om)\psi_j((\om_3-\Psi(\bar\om;t_j))\wt\phi(\om;t_j))$ is a smooth function, with $\wt\phi(\om;t_j)=\phi(\xi;t_j)$. Finally, one can obtain \eqref{kernelestimate}
by the method of (non) stationary phase. We leave out the details.
\end{proof}

In order to get \eqref{sqfcn-local-physics}, we only need to bound the right hand side of \eqref{broad-narrow-final}. Before going even further, let us introduce one more notation and one more lemma.

\begin{definition}
Let $\bq=\{q\}$ be the collection of finitely overlapping $CR^{-1/2}$-cubes in the frequency space that was introduced in the paragraph above \eqref{lattice-sqfcn}. For any cap $\tau\subset\ZS^2$ and any $1\leq j\leq R$, let $\Ga_{j}(\tau)$ be the subset of $\Gamma_j$ where the normal directions lie in $2\tau$:
\begin{equation}
    \Ga_{j}(\tau):=\{\xi\in\Ga_j:G_j(\bar\xi)\in2\tau\},
\end{equation}
where $G_j$ is the Gauss map introduced in \eqref{Gauss-map}.
Define $\bq(\tau)\subset\bq$ as
\begin{equation}
\label{q-tau}
    \bq(\tau):=\{q\in\bq:{\rm{there~is~a}}~j,~1\leq j\leq R,~q\cap \Ga_{j}(\tau)\not=\varnothing\}.
\end{equation}
\end{definition}
\noindent When $\Phi(\bar\xi;t)=t-|\bar\xi|^2/2$ or $\Phi(\bar\xi;t)=(t-|\bar\xi|^2)^{1/2}$ is the family of paraboloids or spheres, cubes in $\bq(\tau)$ are roughly contained in a tube, whose radius is the same as the radius of the cap $\tau$. For general $\Phi$ satisfying \eqref{Phi-conditions}, cubes in $\bq(\tau)$ are morally contained in a curved tube instead. The exact distribution of cubes in $\bf q(\tau)$ does not concern us. What we need for $\bf q(\tau)$ is the next lemma.

\begin{lemma}
\label{finiely-overlap-lattice-cube}
Suppose that $\cT$ is a collection of finitely overlapping $E^{-1}$-caps in $\ZS^2$ with $1<E\leq R^{1/2}$. Then the collection $\{\bf q(\tau)\}_{\tau\in\cT}$ is also finitely overlapped.
\end{lemma}
\begin{proof}
Fix a cap $\tau_1$ and an arbitrary $CR^{-1/2}$-cube $q_1\in\bq(\tau_1)$. We claim that if $|\tau_1-\tau_2|\geq C'E^{-1}$ for some large enough constant $C'$, then there does not have a $CR^{-1/2}$ cube $q_2\in\bq(\tau_2)$ such that $q_2\cap q_1\not=\varnothing$. This suffices to prove our lemma. 

To prove the our claim, we first note that by the condition $|\partial_{t}\Phi|\sim1$ in \eqref{Phi-conditions}, there are $O(R^{1/2})$ many $j$ such that $q_1\cap \Ga_j\not=\varnothing$. Let us denote by $J$ the collection of these $j$. For a fixed $j\in J$, from Lemma \ref{Gauss-map-spt} we know that the pullbacks $\{G_j^{-1}(\tau)\}_{\tau\in\cT}$ are finitely overlapped. Also, via \eqref{Phi-conditions} we know that $\dist(G_{j}^{-1}(\tau),G_{j'}^{-1}(\tau))\lesssim R^{-1/2}$ if $|j-j'|\lesssim R^{-1/2}$. These two arguments in particular implies that if $\dist(\tau_1,\tau_2)\geq C' R^{-1/2}$ for some large enough constant $C'$, $\dist(G_{j}^{-1}(\tau_1),G_{j'}^{-1}(\tau_2))\geq 100C R^{-1/2}$ for any $j'\in J$. It gives our claim. 
\end{proof}

The second part of \eqref{broad-narrow-final} is easy to handle by Lemma \ref{finiely-overlap-lattice-cube} and the next lemma.
\begin{lemma}
\label{half-lem}
Let $\Sq_\theta f$ be defined in \eqref{half-sqfcn} and let $\De_q f$ be defined in \eqref{lattice-sqfcn}. Suppose that $\theta\in\Theta$ is a $R^{-1/2}$-cap. Then for any $\e>0$ and $2\leq p<\infty$,
\begin{equation}
\label{half-esti}
    \|\Sq_\theta f\|_{L^p(B_R)}\leq C_\e R^{\e}\Big\|\Big(\sum_{q\in\bq(\theta)}|\De_q f|^2\Big)^{\frac{1}{2}}\Big\|_{L^p(w_{B_{R^{1+\e}},N})}.
\end{equation}
\end{lemma}
\begin{proof}
We let $\be=\e^{1000}$ and let $T$ be the $R^{1/2+\be}\times R^{1/2+\be}\times R^{1+\be}$ rectangular tube centered at the origin with direction $c_\theta$, where $c_\theta\in\ZS^2$ is the center of $\theta$. Then by Definition \ref{adapt-def} and \eqref{kernelestimate}, the kernels $\{R^{-O(\be)}m_{j,\si}\}_j$ are all adapt to the tube $T$. Also, recalling that $q$ is a $CR^{-1/2}$-cube, we choose $C$ large enough such that if $\Ga_j(\theta)\cap q\not=\varnothing$, then $\Ga_j(\theta)\subset q$. Hence for each $CR^{-1/2}$ cube $q\in\bq(\theta)$, by Lemma \ref{weightedlemma} one has
\begin{equation}
\label{one-q}
    \sum_{j:\Ga_j(\theta)\cap q\not=\varnothing}\int|m_{j,\theta}\ast f|^2g\leq C_{N,\be} R^{O(\be)}\int|\De_q f|^2w_{T,N}\ast g 
\end{equation}
Note that there is a 1-bounded, positive function $g\in L^{p'/2}(B_R)$ with $\|g\|_{p'/2}=1$, such that
\begin{equation}
    \|\Sq_\theta f\|_{L^p(B_R)}^2=\sum_{j}\int|m_{j,\theta}\ast f|^2 g.
\end{equation}
We sum up all $j\in\{1,2,\ldots R\}$ in \eqref{one-q} to have
\begin{equation}
    \|\Sq_\theta f\|_{L^p(B_R)}^2\leq C_{2N,\be} R^{O(\be)}\int\sum_{q\in\bq(\theta)}|\De_q f|^2\wt w_{T,2N}\ast g.
\end{equation}
To deal with the weight $\wt w_{T,2N}$, one can check directly that uniformly for any $y\in B_R$, the new weight $\bar w(x,y):= \wt w_{T,2N}(x-y)/w_{B_{R^{1+\be}},N}(x)$ is integrable. In particular, we have $\|\bar w(\cdot,y)\|_1=O(1)$. Hence
\begin{equation}
\nonumber
    \int\sum_{q\in\bq(\theta)}|\De_q f|^2\wt w_{T,2N}\ast g=\int\sum_{q\in\bq(\theta)}|\De_q f(x)|^2w_{B_{R^{1+\be}},N}(x)\Big(\int \wt w(x,y)g(y)dy\Big)dx.
\end{equation}
Notice that one the other hand $\|\bar w(x,\cdot)\Id_{B_R}(\cdot)\|_1=O(1)$. One can thus use Young's inequality for integral operators to conclude that $\|\int \wt w(\cdot,y)g(y)dy\|_{p'/2}\lesssim\|g\|_{p'/2}$. Finally, we use H\"older's inequality to conclude from the above two estimates that
\begin{equation}
     \|\Sq_\theta f\|_{L^p(B_R)}\leq C_{\be}R^{O(\be)} \Big\|\Big(\sum_{q\in\bq(\theta)}|\De_q f|^2\Big)^{\frac{1}{2}}\Big\|_{L^p(w_{B_{R^{1+\be}},N})}.
\end{equation}
Since $\be=\e^{1000}<\e$, it proves \eqref{half-esti}
\end{proof}

The first part of \eqref{broad-narrow-final} is our main focus. In fact, we will prove the following result.

\begin{lemma}
\label{main-broad-thm}
Let $K\sim \log R$ be a big number and let $A=[\log K]$. Recall the definition of $\Sq_{\cT_{\si,k}}f$ in \eqref{sigma-sqfcn} and \eqref{sigma-kernel}. Suppose that $M=K^\ell$ for some $1\leq\ell<\log R^{1/2}/\log K$ and $\si\in\Si_\ell$ is a $M^{-1}$-cap. Then for any $\e>0$, $1\leq k\leq100$, $4\geq p\geq3.25$, 
\begin{equation}
\label{broad-main}
    \|\Br_{A}\Sq_{\cT_{\si,k}} f\|_{L^p(B_R)}\leq C_\e R^{\frac{p-3}{p}+\e} M^{\frac{6-2p}{p}}\Big\|\Big(\sum_{q\in\bq(\si)}|\De_q f|^2\Big)^{\frac{1}{2}}\Big\|_{L^p(w_{B_{R^{1+\e}},N})}.
\end{equation}
The function $\De_q f$ was introduced in \eqref{lattice-sqfcn}. 
\end{lemma}

Via Lemma \ref{finiely-overlap-lattice-cube}, in \eqref{broad-narrow-final}, we can sum up the first part using Lemma \ref{main-broad-thm}, and sum up the second part using Lemma \ref{half-lem} to conclude \eqref{sqfcn-local-physics}. Since Lemma \ref{finiely-overlap-lattice-cube} and Lemma \ref{half-lem} are already verified, we have the following lemma
\begin{lemma}
\label{intermediate-reduction}
Lemma \ref{main-broad-thm} implies Theorem \ref{sqfcn-local-thm}.
\end{lemma}

From now on, let us fix the factors $M,\si$. To save notations, we use $\Sq f$ in place of $\Sq_{\cT_{\si_\ell,k}} f$. Comparing \eqref{sigma-sqfcn} and \eqref{sigma-kernel}, the new square function $\Sq f$ is defined as
\begin{align}
\label{sigma-sqfcn-2}
    &S_{j}f:=m_{j}\ast f,\hspace{1cm}\Sq f(x):=\Big(\sum_{j=1}^R|S_{j}f(x)|^2\Big)^{1/2}\\ \label{sigma-kernel-2}
    &m_{j}:=\sum_{\tau\in\cT}m_{j,\tau},\hspace{5mm}\wh{m}_{j,\tau}(\xi):=\eta_{j,\tau}(\bar\xi)\psi_j\Big(\frac{\xi_3-\Phi(\bar\xi;t_j)}{R^{-1}}\phi(\xi;t)\Big).
 \end{align}
Here the set $\cT$ is a collection of $8(MK)^{-2}$-separated $(MK)^{-1}$-caps in an ambient $M^{-1}$-cap $\si$, and the smooth function $\eta_{j,\tau}$ satisfies
\begin{enumerate}
    \item $\eta_{j,\tau}$ is supported in $G_j^{-1}(2\tau)$, where $G_j$ is the Gauss map defined in \eqref{Gauss-map}, 
    \item For any multi-index $\al\in\ZN^2$ that $|\al|=N$, $|\partial^\al\eta_{j,\si}|\leq C_N(MK)^{N(\ell+1)}$.
\end{enumerate}
Under this new notations, comparing to \eqref{broad-main}, we would like to prove
\begin{equation}
\label{broad-main-2}
    \|\Br_{A}\Sq f\|_{L^p(B_R)}\leq C_\e R^{\frac{p-3}{p}+\e} M^{\frac{6-2p}{p}}\Big\|\Big(\sum_{q\in\bq}|\De_q f|^2\Big)^{\frac{1}{2}}\Big\|_{L^p(w_{B_{R^{1+\e}},N})}.
\end{equation}
The two estimate \eqref{broad-main} and \eqref{broad-main-2} are parallel, so we only focus on \eqref{broad-main}.

\medskip

Let us introduce one more definition to further reduce \eqref{broad-main-2} to a more local estimate, which is our main result.

\begin{definition}[Rescaled balls]
\label{recalsed-ball}
Let $\si$, $M>1$ be the ambient factors in \eqref{broad-main-2} (See also \eqref{broad-main}). For any radius $r>1$, we say a geometric object in $\ZR^3$ is a ``rescaled ball" radius $r$, if it is a tube of length $r$ and radius $M^{-1}r$, pointing to the direction $c_\si$. We use $\cp(x,r)$ to denote a rescaled $r$ ball centered at $x\in\ZR^3$. For simplicity, we use $\cp_r$ to denote $\cp(0,r)$.
\end{definition}

\noindent Note that in \eqref{sigma-kernel-2}, the kernel $m_{j}$ is essentially supported in the rescaled ball $\cp_{R}$. This is our motivation for introducing the term ``rescaled ball". Under the orthonormal coordinate $\{e_1,e_2,e_3\}$ with $e_3=c_\si$, we define an associated weight $w_{\cp_R}$ as
\begin{equation}
\label{weight-cp-R}
    w_{\cp_R}(x):=\Big(1+\frac{M|x_1|}{R}+\frac{M|x_2|}{R}+\frac{|x_3|}{R}\Big)^{-N_0}.
\end{equation}
Here $N_0$ is a big number.

\begin{theorem}[The main result for broad functions]
\label{local-broad-thm}
Let $\Sq f$ be defined in \eqref{sigma-sqfcn-2} with two ambient factors $\si$ and $M$. Suppose that $A=[\log\log R]$. Then for any $\e>0$ and $4\geq p\geq3.25$, 
\begin{align}
\label{local-main}
    \int_{\cp_R}\Br_{A}\Sq f^p\leq &\, C_\e R^{p-3+p\e} M^{6-2p}\int\Big(\sum_{q\in\bq }|\De_q f|^2\Big)^{p/2}w_{\cp_R}.
\end{align}
\end{theorem}
\noindent By summing up certain translations, Theorem \ref{local-broad-thm} implies \eqref{broad-main-2} and hence Lemma \ref{main-broad-thm}. Thus, via Lemma \ref{intermediate-reduction}, we can conclude our second reduction in the next proposition.

\begin{proposition}[Second reduction]
\label{second-reduction}
Theorem \ref{local-broad-thm} implies Theorem \ref{sqfcn-local-thm}.
\end{proposition}
Combining Proposition \ref{first-reduction} and \ref{second-reduction}, we know that Theorem \ref{local-broad-thm} implies our main result Theorem \ref{sqfcn-thm-intro}. From now on, let us focus on the broad function $\Br_{A}\Sq f$ and the estimate \eqref{local-main}.

\subsection{Auxiliary broad functions} \hfill

Finally, we are going to define a new broad functions for technical reasons. We begin with the definition of a square function for vectors.

\begin{definition}[Square function for vector-valued functions]\label{defsqfunc}
Suppose that we are given a vector-valued function $\vg=\{g_1,\cdots,g_R\}$. We define the square function of $\vg$ as:
\begin{equation}
    \Sq\vg(x):=\Big(\sum_{j}|g_j(x)|^2\Big)^{1/2}.
\end{equation}
\end{definition}

\noindent To see this new definition of square function coincides with the square function $\Sq f$ in \eqref{sigma-sqfcn-2} at some point, we define $f_j$ and $\vf$ as
\begin{equation}
\label{vector-f}
    f_j:=S_{j}f,\hspace{1cm}\vf:=\{f_1,\ldots,f_R\},
\end{equation}
then $\Sq f=\Sq\vf$. Since any cap $\tau\in\cT$ is independent of the vertical factor $j$, we can gather all the vertical components together and write in the vector-valued form as
\begin{equation}
    f_{j,\tau}:=S_{j,\tau}f, \hspace{5mm}     \vf_\tau:=\{f_{1,\tau},\ldots,f_{R,\tau}\}.
\end{equation}

\begin{remark}
\rm
Our notation is nothing mysterious, but just to make our formula not too lengthy. It turns out we will treat each $j$- slice in the same way, so it will be convenient to omit the subscript $j$ and write them as $\vf_\tau$, as we will see later.
\end{remark}

Recall that in \eqref{sigma-kernel-2}, we already assume that those $(MK)^{-1}$-caps in $\cT$ are $8(MK)^{-1}$-separated.
For each $\tau\in\cT$, we define a map $P_{j,\tau}:L^p(\ZR^3)\to L^p(\ZR^3)$ as
\begin{equation}
    P_{j,\tau}g(x)=\int_{\ZR^3}e^{2\pi ix\cdot\xi}\wh{g}(\xi)\Id_{5\tau}\circ G_j(\bar\xi)d\xi.
\end{equation}
Basically, the map $P_{j,\tau}$ restricts the Fourier support of $g$ to a vertical stripe. In this stripe, the normal vectors of the surface $\Ga_j$ belong to $5\tau\subset\ZS^2$. For any vector $\vg=\{g_1,\ldots,g_R\}$, we further define $P_\tau\vg:=\{P_{1,\tau}g_1,\ldots,P_{R,\tau}g_R\}$. The map $P_\tau$ may not look natural, but it would help us define the broad function rigorously in the rest of the paper. As an example, we have $P_\tau\vf=\vf_\tau$\,.

Now we can introduce the broad part for any vector $\vg=\{g_1,\ldots,g_R\}$.
\begin{definition}
\label{broad-vector}
Given $A\in\ZN$ and any $x\in \cp_R$, we define $\BR_A\Sq \vg(x)$ as the $A+1$ largest number in $\{\Sq(P_\tau\vg)(x)\}_{\tau\in\cT}$. That is, 
\begin{equation}
\label{A-broad-vector}
    \BR_A\Sq \vg(x):=\min_{\tau_1,\ldots,\tau_{A}\in\cT}\max_{\substack{\tau\not=\tau_l,\\1\leq l\leq A}}\Sq(P_\tau\vg)(x).
\end{equation}
\end{definition}
\noindent Note that in particular one has $\Br_A \Sq f(x)=\Br_A\Sq \vf(x)$ when $x\in\cp_R$.

We will not need to define broad function for an arbitrary vector $\vg$, but only for vectors appear in the one-step polynomial partitioning algorithm in Section \ref{modified-poly} and in the iteration built up in Section \ref{iteration}. These vectors have similar geometric patterns as the vector $\vf$ does.

\medskip

Probably the most important property  we will use for the broad function is the following weak version of triangle inequality. It allows us to keep the ``broad" property when decomposing the original function.
\begin{lemma}[Triangle inequality]
If $\vg=\vg_1+\vg_2$, then 
\begin{equation}
\label{broad-triangle}
    \BR_A\Sq\vg(x)\leq \BR_{A/2}\Sq\vg_1(x)+\BR_{A/2}\Sq\vg_2(x)
\end{equation}
\end{lemma}
\begin{proof}
By the definition of the broad function, suppose that there are two collection of caps $\{\tau_i\}_{i=1}^{A/2}, \{\tau_i'\}_{i=1}^{A/2}$ such that

\begin{align}
    \BR_{A/2}\Sq\vg_1(x)=\max_{\substack{\tau\not=\tau_k,\\1\leq k\leq A/2}}\Sq(P_\tau\vg_1)(x), \\
    \BR_{A/2}\Sq\vg_2(x)=\max_{\substack{\tau\not=\tau'_k,\\1\leq k\leq A/2}}\Sq(P_\tau\vg_2)(x).
\end{align}
Since $\Sq(P_\tau\vg_1+P_\tau\vg_2)\le \Sq(P_\tau\vg_1)+\Sq(P_\tau\vg_2)$, we have
\begin{equation}
    \BR_{A}\Sq\vg(x)\le \max_{\substack{\tau\not=\tau_k,\tau\not=\tau_k',\\1\leq k\leq A/2}}\Sq(P_\tau\vg)(x)\le \BR_{A/2}\Sq\vg_1(x)+\BR_{A/2}\Sq\vg_2(x)
\end{equation}
as desired.
\end{proof}

The broad function is dominated by a certain bilinear function. First, we define
\begin{equation}
\label{bilinear-def}
    \bil\vg(x):=\sum_{\tau_1,\tau_2\in\cT}|\Sq(P_{\tau_1}\vg)(x)|^{1/2}|\Sq(P_{\tau_2}\vg)(x)|^{1/2}.
\end{equation}
\begin{lemma}
\label{broad-blinear}
For any $A\ge 2$, we have $\BR_A\Sq\vg(x)\leq\bil\vg(x)$.
\end{lemma}
\begin{proof}
The proof is just by definition.
\end{proof}

\section{Wave packet decomposition}
\label{wave-packet}

In this section, we build wave packet decomposition for some special types of functions (Including the vector $\vf$ introduced in \eqref{vector-f}) at different scales. These functions will appear in the next two sections when establishing our iteration. Here is a sketch of the idea: We first use pullback of Gauss map to build up a partition in the frequency space. Then for each part of the previous partition, we introduce an associated partition of unity in the physical space to finish the wave packet decomposition.

\subsection{Wave packet decomposition at the largest scale}\hfill

We first consider the wave packet decomposition at the largest scale $R$. Recall that each component of the vector $\vf$ in \eqref{vector-f} has the expression
\begin{equation}
    f_j(x)=\int_{\ZR^3} e^{2\pi ix\cdot\xi}\wh{f}(\xi)\eta_j(\bar\xi)\psi_j\Big(\frac{\xi_3-\Phi(\bar\xi;t_j)}{R^{-1}}\phi(\xi;t)\Big)d\xi.
\end{equation}
Cover $\ZS^2$ using a collection of two dimensional $R^{-1/2}$ caps $\Theta=\{\theta\}$, and let $\wt\vp_\theta:\ZS^2\to\ZR^+$ be a smooth partition of unity associated to the cover $\Theta$. Each function $\wt\vp_\theta$ is supported in $2\theta$. Also, for any multi-index $\al\in\ZN^2$ that $|\al|=N$, $|\partial^\al\wt\vp_\theta|\leq C_NR^{N/2}$. Since each restricted map $G_j$ is injective as proved in Lemma \ref{Gauss-map-lem}, the pullback $\vp_\theta=\wt\vp_\theta\circ G_j$ forms a partition of unity of $\Id_{\textup{supp}(\eta_j)}(\bar\xi)$. Therefore we can partition $f_j$ as
\begin{equation}
\label{frequencey-decom-2}
    f_j=\sum_{\theta\in\Theta}f_{j,\theta},\hspace{5mm}f_{j,\theta}(x):=\int_{\ZR^3} e^{2\pi ix\cdot\xi}\wh{f}(\xi)\vp_\theta(\bar\xi)\eta_j(\bar\xi)\psi_j\Big(\frac{\xi_3-\Phi(\bar\xi;t_j)}{R^{-1}}\phi(\xi;t)\Big)d\xi.
\end{equation}
For convenience, we write
\begin{equation}
\label{m-j-theta}
    f_{j,\theta}=m_{j,\theta}\ast f,\hspace{5mm}m_{j,\theta}(x):=\int_{\ZR^3} e^{2\pi ix\cdot\xi}\vp_\theta(\bar\xi)\eta_j(\bar\xi)\psi_j\Big(\frac{\xi_3-\Phi(\bar\xi;t_j)}{R^{-1}}\phi(\xi;t)\Big)d\xi.
\end{equation}
We remark that because of the cutoff $\eta_j$, only those caps $\theta$ whose center $c_\theta\in\ZS^2$ is transverse to the horizontal plane make contribution in \eqref{frequencey-decom-2}.

Next, let us partition the physical space. Fix a cap $\theta\in\Theta$. After rotating $c_\theta$, without loss of generality we assume $c_\theta=e_3$. Let $\{v\}\in\ZZ^3$ be the set of lattice points, so we can find a partition of unity $\{\eta_v\}$ associated to it that $\eta_v(x)=\eta(x-v)$. Here $\eta:\ZR^3\to\ZR^+$ is a smooth function whose Fourier transform is supported in the unit ball in the frequency space, and $\eta$ decays rapidly outside the unit ball in the physical space.

We tailor the partition of unity $\{\eta_v\}$ a little bit. For any given $\e$ in Theorem \ref{local-broad-thm}, let $\be:=\e^{1000}$ be a small number from now on, and let $\{u\}$ be the collection of lattice $R^{\be}\ZZ^3$ points. For each point $u$, we define 
\begin{equation}
    \eta_u(x)=\sum_{v\in E(u)}\eta_v(x), \hspace{3mm}E_u:=\{v:-R^{\be}/2<v_l-u_l\leq R^{\be}/2,~l=1,2,3\}.
\end{equation}
Hence $\{\eta_u\}$ also forms a smooth partition of unity. One advantage about this new partition of unity is that the functions $\{\eta_u\}$ are morally orthogonal. In fact, $|\eta_u(u')|=\rap(R)$ for all but finitely many points $u'\in R^{\be/2}\ZZ^3$.

Now that each function $\wt\eta_u(x):=\eta_u(R^{-1/2}(x_1,x_2),R^{-1}x_3)$ is essentially supported in an $R^{1/2+\be}\times R^{1/2+\be}\times R^{1+\be}$ rectangular tube. In fact, if we define a tube associated to $\wt\eta_u$ as
\begin{equation}
    T_u:=\{x\in\ZR^3,R^{-1/2}|x_l-u_l|\leq R^\be,~l=1,2;~R^{-1}|x_3-u_3|\leq R^\be\},
\end{equation}
then for any $x\in\ZR^3\setminus T_u$ and any $N\geq1$,
\begin{equation}
    \wt\eta_u(x)\leq \rap(R)C_N\Big(1+\frac{|x_1-u_1|}{R^{1/2}}+\frac{|x_2-u_2|}{R^{1/2}}+\frac{|x_3-u_3|}{R}\Big)^{-N}
\end{equation}
The coreline of the tube $T_u$ is parallel to the vector $c_\theta=e_3$. Define 
\begin{equation}
\label{scale-R-tube-set}
    \ZT_\theta[R]:=\{T_u\}
\end{equation}
as the collection of all these rectangular tubes, and define $\Id_{T_u}^\ast(x)=\wt\eta_u(x)$ so that $\{\Id_{T}^\ast\}_{T\in\ZT_\theta[R]}$ forms a partition of unity of $\ZR^3$ as well. This completes the physical partition for a fixed cap $\theta$. For a different cap $\theta'\in\Theta$, one can similarly construct a partition of unity $\{\Id_{T_{\theta'}}^\ast\}$. Using the partition of unity for every cap $\theta\in\Theta$, we further partition $f_j(x)$ as

\begin{equation}
\label{wave-packet-1}
    f_j=\sum_{\theta\in\Theta}\sum_{T_\theta\in\ZT_\theta[R]}f_{j,T_\theta},\hspace{5mm}f_{j,T_\theta}:=f_{j,\theta}\Id_{T_\theta}^\ast.
\end{equation}
This is the scale $R$ wave packet decomposition we are looking for. Each $f_{j,T_\theta}$ is a single wave packet. Since the $\theta$ and $T_\theta$ are both independent to the factor $j$, we define $\vf_\theta:=\{f_{1,\theta},\ldots,f_{R,\theta}\}$ and $\vf_{T_\theta}=\vf_{\theta}\Id^\ast_{T_\theta}:=\{f_{1,T_\theta},\ldots,f_{R,T_\theta}\}$. We also call $\vf_{T_\theta}$ a single wave packet.

\medskip
\begin{definition}
\label{dual}
A rectangular box $\om$ is said to be ``dual" to another rectangular box $T$ of dimension $\rho_1\times\rho_2\times\rho_3$, if $\om$ has dimensions $\sim\rho_1^{-1}\times\rho_2^{-1}\times\rho_3^{-1}$ and the $j$-th side of $\om$ is parallel to the $j$-th side of $T$.
\end{definition}
The next two lemmas contain some useful properties about the wave packet decomposition \eqref{wave-packet-1}. Their proof follows directly by definition and Plancherel.

\begin{lemma}[Fourier support]
\label{fourier-support-big}
The Fourier transform of each wave packet $f_{j,T_\theta}$ is contained a slab $S_j(\theta)$ of dimensions $\sim R^{-1}\times R^{-1/2}\times R^{-1/2}$ that is dual to $T_\theta$. The shortest side of $S_j(\theta)$ is parallel to the direction $c_\theta$ and  $S_j(\theta)\subset N_{CR^{-1/2}}(\Ga_j)$. Also, the collection $\{S_j(\theta)\}_{\theta\in\Theta}$ is finitely overlapped.
\end{lemma}

\begin{lemma}[$L^2$-orthogonality]
\label{l2-orthogonality-big}
For an arbitrary collection $\ZT'\subset\cup_\theta\ZT_\theta[R]$, one has
\begin{equation}
    \Big\|\sum_{T\in\ZT'}f_{j,T}\Big\|_2^2\lesssim \sum_{T\in\ZT'}\|f_{j,T}\|_2^2.
\end{equation}
This estimate is uniform for all $1\leq j\leq R$.
\end{lemma}

Let us conclude the wave packet decomposition above into a proposition.
\begin{proposition}[Wave packet decomposition at the largest scale]
Let $\vf$ be defined in \eqref{vector-f}. Then for each $R^{-1/2}$-cap $\theta\in\Theta$, there is a collection of finitely overlapping tubes $\ZT_\theta[R]$ defined in \eqref{scale-R-tube-set}, and a smooth, positive partition of unity $\{\Id_{T_\theta}^\ast\}_{T_\theta\in\ZT_\theta[R]}$, such that
\begin{equation}
\label{wave-packet-big}
    f_j=\sum_{\theta\in\Theta}\sum_{T_\theta\in\ZT_\theta[R]}f_{j,T_\theta},\hspace{5mm}f_{j,T_\theta}:=f_{j,\theta}\Id_{T_\theta}^\ast.
\end{equation}
Also, recalling \eqref{weight}, the smooth function $\Id_{T_\theta}^\ast$ satisfies that
\begin{equation}
    \Id_{T_\theta}^\ast(x) \leq \rap(R)C_Nw_{T_\theta,N}(x),
\end{equation}
and the Fourier transform of $\Id^\ast_{T_\theta}$ is supported in a dual rectangular box of $T_\theta$ that is centered at the origin. In addition, Lemma \ref{fourier-support-big} and Lemma \ref{l2-orthogonality-big} are true.
\end{proposition}

\subsection{Wave packet decomposition at smaller scales}\hfill\label{wpdsmallscale}

We will state the wave packet decomposition in a smaller scale $\rho$, $R^{\e/10}\leq\rho< R$,  for some special types of functions. Assume that $r>\rho>r^{1/2}$ and $\vg=\{g_1,\ldots,g_R\}$ is a vector valued function satisfying the following properties:
\begin{enumerate}
    \item \label{property1} The Fourier transform of each $g_{j}$ is contained in $N_{\rho^{-1}}(\Ga_j)$.
    \item \label{property2} Uniformly for $1\leq j\leq R$, $g_j$ already has a decomposition 
    \begin{equation}
    \label{prop2}
        g_j=\sum_{\overline\om}g_{j,\overline\om},
    \end{equation}
    where $\overline\om$ is a $r^{-1/2}$ cap in $\ZS^2$ and the Fourier transform of $g_{j,\overline\om}$ is supported in a $\rho^{-1}\times r^{-1/2}\times r^{-1/2}$ -slab whose shortest side is parallel to the direction $c_{\overline\om}$. 
\end{enumerate}

We partition $\ZS^2$ into $\rho^{-1/2}$-caps $\Om=\{\om\}$. 
We are going to define a relationship between $\Om=\{\om\}$ and $\bar\Om=\{\bar\om\}$.

\begin{definition}\label{caprelation}
For any $\bar\om\in\bar\Om$, we pick one $\om\in\Om$ so that $\bar\om\subset 2\om$ (the choice for $\om$ may not be unique), and denote it by
\begin{equation}
    \bar\om<\om.,
\end{equation}
so that if we let $\bar\Om (\om):=\{\bar\om\in\bar\Om: \bar\om<\om \}$, we have
\begin{equation}\label{partitioningproperty}
    \bar\Om=\sqcup_{\om\in\Om} \bar\Om(\om). 
\end{equation} 
\end{definition}




\noindent Intuitively, we can think of $\bar\om<\om$ as $\bar\om\subset\om$. For each $g_j$, we define
\begin{equation}
\label{wave-packet-nested}
    g_{j,\om}=\sum_{\overline\om<\om}g_{j,\overline\om}.
\end{equation}
By \eqref{partitioningproperty}, we have
\begin{equation}
    g_{j}=\sum_{\om\in\Om}g_{j,\om}.
\end{equation}

Now for each cap $\om$, define
\begin{equation}
\label{scale-rho-tube-set}
    \ZT_{\om}[\rho]:=\{T_\om\}
\end{equation}
as a collection of finitely overlapping $\rho^{1/2+\be}\times \rho^{1/2+\be}\times \rho^{1+\be}$ rectangular tubes whose direction is $c_\om$, and who form a cover of $\ZR^3$. We introduce an associated partition of unity $\{\Id_{T_\om}^\ast\}$  as we did in last subsection, so we can partition $g_{j}$ as
\begin{equation}
\label{wave-packet-2}
    g_{j}=\sum_{\om\in\Om}\sum_{T_\om\in\ZT_\om[\rho]}g_{j,T_\om},\hspace{5mm}g_{j,T_\om}:=g_{j,\om}\Id_{T_\om}^\ast.
\end{equation}
This is the scale $\rho$ wave packet decomposition and $g_{j,T_\om}$ is a single wave packet. We also define $\vg_\om:=\{g_{1,\om},\ldots,g_{R,\om}\}$ and $\vg_{T_\om}=\vg_{\om}\Id^\ast_{T_\om}:=\{g_{1,T_\om},\ldots,g_{R,T_\om}\}$, and call $\vg_{T_\om}$ a single wave packet.

Similarly, the next two lemmas contain some useful properties about the wave packet decomposition \eqref{wave-packet-2}.


\begin{lemma}[Fourier support]
\label{Fourier-support-small-scale}
The Fourier transform of each wave packet $g_{j,T_\om}$ is contained a $\rho^{-1}\times \rho^{-1/2}\times \rho^{-1/2}$ -slab $S_j(\om)$ whose shortest side is parallel to the direction $c_\om$. Also, $S_j(\om)\subset N_{C\rho^{-1/2}}(\Ga_j)$, and the slabs $\{S_j(\om)\}_{\om\in\Om}$ are fnitely overlapped.
\end{lemma}

\begin{lemma}[$L^2$-orthogonality]
Suppose that $\ZT'\subset\cup_\om\ZT_\om[\rho]$. Then
\begin{equation}
\label{L2-orthogonality-small-1}
    \Big\|\sum_{T\in\ZT'}g_{j,T}\Big\|_2^2\lesssim \sum_{T\in\ZT'}\|g_{j,T}\|_2^2.
\end{equation}
The estimate is uniform for all $1\leq j\leq R$.
\end{lemma}

We can similarly conclude the wave packet decomposition at a smaller scale into a proposition.

\begin{proposition}[Wave packet decomposition at a smaller scale]
\label{scale-rho-wpt}
Let $\vg$ a vector satisfying \eqref{property1} and \eqref{property2}. Then for each $\rho^{-1/2}$-cap $\om\in\Om$, there is a collection of finitely overlapping tubes $\ZT_\om[\rho]$ defined in \eqref{scale-rho-tube-set}, and a smooth, positive partition of unity $\{\Id_{T_\om}^\ast\}_{T_\om\in\ZT_\om[\rho]}$, such that
\begin{equation}
\label{wave-packet-small}
        g_{j,\om}=\sum_{\overline\om<\om}g_{j,\overline\om},\hspace{5mm}g_{j}=\sum_{\om\in\Om}\sum_{T_\om\in\ZT_\om[\rho]}g_{j,T_\om},\hspace{5mm}g_{j,T_\om}:=g_{j,\om}\Id_{T_\om}^\ast..
\end{equation}
Also, recalling \eqref{weight}, the smooth function $\Id_{T_\theta}^\ast$ satisfies that
\begin{equation}
    \Id_{T_\om}^\ast(x) \leq \rap(R)C_Nw_{T_\om,N}(x),
\end{equation}
and the Fourier transform of $\Id^\ast_{T_\om}$ is supported in a dual rectangular box of $T_\om$ that is centered at the origin. In addition, Lemma \ref{Fourier-support-small-scale} and Lemma \ref{L2-orthogonality-small-1} are true.
\end{proposition}

\begin{remark}
{\rm
Since $\rho\geq R^{\e/10}$, we can always view ${\rm RapDec}(\rho)$ as ${\rm RapDec}(R)$.
}
\end{remark}

\section{Modified polynomial partitioning}
\label{modified-poly}
\noindent{\bf Some notations.} Starting in this section, we set $\de:=\e^2$, $d:=R^{\e^6}$, $\be:=\e^{1000}$.

\medskip

Unlike the case in \cite{Wu}, here in the desired estimate \eqref{local-main}, the target function $\Id_{\cp_R}\Br_{A}\Sq f^p$ is assumed to be supported in the rescaled ball $\cp_R$ instead of $B_R$. Hence we need to modify the original polynomial partitioning accordingly. First, let us introduce some definitions regarding to the zero set of polynomials.
\begin{definition}
Suppose $Q_1,\ldots,Q_k$ are polynomials in $\ZR^n$. We say $Z(Q_1,\ldots,Q_k)$ is a transverse complete intersection if for any $x\in Z(Q_1,\ldots,Q_k)$, the vectors $\nabla Q_1(x),\ldots,\nabla Q_k(x)$ are linearly independent.
\end{definition}
\begin{definition}
We say a polynomial $P$ in $\ZR^n$ is non-singular, if $\nabla P(x)\not=0$ for any $x\in Z(P)$.
\end{definition}

The modified polynomial partitioning we need is the following.
\begin{proposition}
\label{final-partition}
Let $\rho>0$. Suppose that $g$ is a non-negative $L^1$ function in $\ZR^3$ supported in a rescaled $\rho$-ball $\cp$. Then for any $d\in\ZZ^+$, there exists a polynomial $P$ with degree $O(d)$, such that
\begin{enumerate}
    \item There are $\sim d^3$ many cells $O$ contained in $\cp\setminus Z(P)$, satisfying 
    \begin{equation}
        \int_{O} g\sim d^{-3}\int_{\ZR^n}g\,.
    \end{equation}
    \item Each of these cells $O$ in {\rm (1)} lie in a smaller rescaled ball of radius $\rho d^{-1}$.
\end{enumerate}
\end{proposition}
\noindent We omit the proof of this proposition. In fact, by a horizontal non-isotropic scaling, we can reduce Proposition \ref{final-partition} to the case where the support of $g$ is the usual ball $B_\rho$. In this case, the corresponding polynomial argument was showed in \cite{Wang-restriction-R3}.

\begin{remark}
\rm
In Proposition \ref{final-partition}, the second outcome not seems to be necessary in our proof. While we still use it later since we feel that it is nice property to emphasize.
\end{remark}

\medskip

Next, we state the one-step polynomial partitioning algorithm based on the idea in \cite{Guth-R3}. This may be less intuitive, but one of its advantage is that when building up our iteration in next section, we can apply this algorithm directly.

\begin{algorithm}\hfill
\label{algorithm}
\rm

{\bf Inputs: } $(\rho,U,\vg_U,A)$. More precisely our inputs are:
\begin{enumerate}
    \item A scale $\rho$ ($1\le\rho\le R$);
    \item A set $U$ which is contained in a $\rho$-rescaled ball $\cp_U$;
    \item A vector-valued function $\vg_U=\{ g_{1,U},\cdots,g_{R,U} \}$. Each component $g_{j,U}$ has Fourier support in $N_{C\rho^{-1}}\big(\Gamma_j(\si)\big)$, and satisfies  \eqref{prop2} in Section 4.2;
    \item An integer $A\geq10$. 
\end{enumerate}

{\bf Outputs: }
We obtain a polynomial $P$ of degree $O(d)$ where $d:=R^{\e^6}$, and a wall $W=U\cap N_{\rho^{1/2+\be}}Z(P)$. Also, we distinguish three scenarios: cellular case, transverse case and algebraic case.

\textit{ Cellular case}:
\begin{enumerate}
    \item We obtain a collection of cells $\co=\{O\}$ which we call cellular cells. They satisfies: each $O$ is contained in a $\rho d^{-1}$- rescaled ball $\cp_O$ and
    \begin{equation}\label{cellular1}
        |\co|\sim d^3.
    \end{equation}
    
    \item We obtain tube sets $\{\T_O\}_{O\in\co}$ and a set of functions $\{\vg_{O}\}_{O\in\co}$ which are indexed by $\co$. They satisfy an $L^2$-estimate
    \begin{equation}\label{cellular2}
        \sum_{O\in\co}\|\Sq\vg_{O}\|^2\le \|\Sq\vg_U\|^2_2;
    \end{equation}
    and a broad estimate
    \begin{equation}\label{cellular3}
        \int_U|\Br_A\Sq\vg_U|^p\sim d^3 \int_O |\Br_{A}\Sq\vg_O|^p \textup{~for~each~}O\in\co.
    \end{equation}
    \item The Fourier transform of each component $g_{j,O}$ of $\vg_{O}$ satisfies:
    \begin{equation}\label{cellular4}
        \textup{supp~}\widehat{g_{j,O}} \subset N_{C(\rho d^{-1})^{-1}} (\Ga_{j}(\si)).  
    \end{equation}
    
\end{enumerate}

\textit{Transverse case}:
\begin{enumerate}
    \item We obtain a collection of cells $\cb=\{B\}$ which we call transverse cells. They satisfies: each $B$ is a subset of $W$ and each $B$ is contained in a $\rho R^{-\de}$-rescaled ball $\cp_B$ with $\de=\e^2$.\label{transverse1}
    \item We obtain tube sets $\{\T_{B,trans}\}_{B\in\cb}$ and a set of functions $\{\vg_{B,trans}\}_{B\in\cb}$ which are indexed by $\cB$. They satisfy an $L^2$-estimate
    \begin{equation}\label{transverse2}
        \sum_{B\in\cb}\|\Sq\vg_{B,trans}\|^2_2\le {\poly} (d) R^{-\de}\|\Sq\vg_{U}\|^2_2;
    \end{equation}
    and a broad estimate
    \begin{equation}\label{transverse3}
        \int_U|\Br_A\Sq\vg_U|^p\lesssim \log R \sum_{B\in\cb}\int_B|\Br_{A/2}\Sq\vg_{B,trans}|^p.
    \end{equation}
    Also, for each $B\in\cb$, the quantity
    \begin{equation}
        \int_B|\Br_{A/2}\Sq\vg_{B,trans}|^p
    \end{equation}
    are the same up to a constant factor.
    \item The Fourier transform of each component $g_{j,B,trans}$ of $\vg_{B,trans}$ satisfies:
    \begin{equation}\label{transverse4}
        \textup{supp~}\widehat{g}_{j,B,trans} \subset N_{C(\rho R^{-\de})^{-1}} (\Ga_{j}(\si)).  
    \end{equation}
    
\end{enumerate}

\textit{Tangent case}:
\begin{enumerate}
    \item We obtain a collection of cells $\cb=\{B\}$ which we call tangent cells. They satisfies: each $B\subset W$, each $B$ is contained in a $\rho R^{-\de}$-rescaled ball $\cp_B$.\label{tangent1}
    \item We obtain tube sets $\{\T_{B,tang}\}_{B\in\cb}$ and a set of functions $\{\vg_{B,tang}\}_{B\in\cb}$ which are indexed by $\cB$. They satisfy an $L^2$-estimate
    \begin{equation}\label{tangent2}
        \sum_{B\in\cb}\|\Sq\vg_{B,tang}\|^2_2\le {\poly} (d) R^{-\de}\|\Sq\vg_{U}\|^2_2;
    \end{equation}
    and a broad estimate
    \begin{equation}\label{tangent3}
        \int_U|\Br_A\Sq\vg_U|^p\lesssim \log R \sum_{B\in\cb}\int_O|\Br_{A/2}\Sq\vg_{B,tang}|^p.
    \end{equation}
    \item The Fourier transform of each component $g_{j,B,tang}$ of $\vg_{B,tang}$ satisfies:
    \begin{equation}\label{tangent4}
        \textup{supp~}\widehat{g}_{j,B,tang} \subset N_{C(\rho R^{-\de})^{-1}} (\Ga_{j}(\si)).  
    \end{equation}
    
\end{enumerate}
\end{algorithm}

\begin{proof}

The rest of this section is devoted to the proof of the Algorithm \ref{algorithm}.
We apply Proposition \ref{final-partition} to the function $\Id_U |\BR_{A}\Sq\vg_U|^p$ to obtain a
polynomial $P$. Also, we obtain a collection of cells $\wt\co=\{\wt O\}$ such that: $|\wt \co|\sim d^3$, each $\wt O$ is contained in a $\rho d^{-1}$-rescaled ball, and 
\begin{equation}\label{cellequal}
    \int_{\wt O}|{\rm Br}_A \Sq\vg_U|^p\sim d^{-3}\int_{U}|{\rm Br}_A \Sq\vg_U|^p
\end{equation}
for each cell $\wt O$. 

Let the wall be $W=U\bigcap N_{\rho^{1/2+\beta}}Z(P)$. We define the shrunken cell $O:=\wt O\setminus W$ and the collection of them $\co:=\{O\}$. Since $U=(\sqcup O)\sqcup W$ , we have the following inequality:
\begin{equation}
\label{cell-only-1}
    \int_{U}|{\rm Br}_A \Sq\vg_U|^p\lesssim \sum_{O\in\co}\int_{O}|{\rm Br}_A \Sq\vg_U|^p+\int_W|{\rm Br}_A \Sq\vg_U|^p.
\end{equation}
Invoking the wave packet decomposition at scale $\rho$, we can write
\begin{equation}
\label{wpt-algorithm}
    \vg_U=\sum_{T\in\T} (\vg_U)_{T},
\end{equation}
where $\T$ is a set of $\rho$-tubes coming from the wave packet decomposition at scale $\rho$. In the following discussion, we will define $\T_O, \T_{B,trans}$ and $\T_{B,tang}$ which are subsets of $\T$. 

First, we analyze the first term on the right hand side of \eqref{cell-only-1}.
For each cell $O\in\co$, define $\ZT_{O}$ as the collection of tubes such that $O\cap T\not=\varnothing$. It is proved in \cite{Guth-R3} that
\begin{lemma}\label{cellulartube}
Each $T\in\T$ belongs to at most $deg(P)+1=O(d)$ many sets $\T_O$.
\end{lemma}

\noindent Pick a smooth cutoff function $\vp_O$ satisfying the following properties: \begin{enumerate}
    \item $\vp_O(x)\sim 1$ when $x\in \cp_O$ ($\cp_O$ is the $\rho d^{-1}$-rescaled ball containing the cell $O$).
    \item \label{rescaled-ball-dual} If $\wh{\cp}_O$ is the dual slab of $\cp_O$ centered at the origin whose dimensions are $ Md \rho^{-1}\times Md \rho^{-1}\times d\rho^{-1}$, we require $\wh\vp_O(\xi)\gtrsim 1$ when $\xi\in\wh{\cp}_O$ and $\wh\vp_O$ is supported in $2\wh{\cp}_O$.
\end{enumerate}
Note that the set $\ZT_{O}$ is independent of any component $g_j$. We define the vector-valued function $\vg_O$ associated to the cell $O$ as
\begin{equation}\label{cellularfunction}
    \vg_{O}:=
    \vp_O\sum_{T\in\T_O}(\vg_U)_{T}.
\end{equation}
From the definition of $\vp_O$, we see that the Fourier support of each component $g_{j,O}$ of $\vg_O$ is contained $\wh\cp_O+N_{C\rho^{-1}}(\Ga_{j}(\si))\subset N_{C'(\rho d^{-1})^{-1}}(\Ga_j(\si))$. This verifies \eqref{cellular4}.

\begin{remark}
\rm

By the definition of rescaled call, $\cp_O$ is of dimensions $M^{-1}\rho d^{-1}\times M^{-1}\rho d^{-1}\times \rho d^{-1}$, so its dual $\wh\cp_O$ is of dimensions 
$M(\rho d^{-1})^{-1}\times M(\rho d^{-1})^{-1}\times (\rho d^{-1})^{-1}$. Since $\Gamma_j(\si)\subset \Ga_j$ is a $M^{-1}$-cap, it is true that $\wh\cp_O+N_{C\rho^{-1}}(\Ga_{j}(\si))\subset N_{C'(\rho d^{-1})^{-1}}(\Ga_j(\si))$. If $\wt\cp_O$ is of dimensions $M'(\rho d^{-1})^{-1}\times M'(\rho d^{-1})^{-1}\times (\rho d^{-1})^{-1}$ with $M'>>M$, we no longer have $\wt\cp_O+N_{C\rho^{-1}}(\Ga_{j}(\si))\subset N_{C'(\rho d^{-1})^{-1}}(\Ga_j(\si))$.
\end{remark}

\begin{remark}
\rm

The reason that we multiply the function $\vp_O$ is to make $\vg_O$ essentially supported in $\cp_O$. That's the same reason for $\vg_{B,trans}$ and $\vg_{B,tang}$ which will be defined later.
\end{remark}

\noindent Since for any $x\in O$, one has $\vp_O(x)\sim 1$. We deduce that for $x\in O$,
\begin{equation}
\label{broadness-cell}
    {\rm Br}_A \Sq(\vg_U)(x)\sim{\rm Br}_{A} \Sq(\vg_{O})(x).
\end{equation}
Consequently, we have
\begin{equation}\label{cellular}
    \sum_{O\in\co}\int_{O}|{\rm Br}_A\Sq\vg_U|^p\lesssim \sum_{O\in\co}\int_{O}|{\rm Br}_A\Sq\vg_O|^p.
\end{equation}

Next, let us analyze the second term on the right hand side of \eqref{cell-only-1}. We choose a collection of $\rho R^{-\de}$-rescaled balls $\wt\cb=\{\wt B\}$ that form a finitely overlapping cover of $U$. Define 
\begin{equation}\label{defineB}
    B:=\wt B\cap W \textup{~and~} \cb:=\{B\}.
\end{equation}
For each $B$, we define $\ZT_{B,tang}$ and $\ZT_{B,trans}$ which are subsets of $\T$ as follows.
\begin{definition}
\label{tangent-definition}
$\ZT_{B,tang}$ is the set of $\rho$-tubes $T\in\T$ obeying the following two conditions:
\begin{enumerate}
    \item [$\bullet$]$T\cap B\not=\varnothing$, 
    \item [$\bullet$]If $z$ is any non-singular point of $Z$ lying in $10 \wt B\cap 10 T$, then 
    \begin{equation}
    \label{angular-condition-1}
        |{\rm Angle}(v(T),T_zZ(P)|\leq \rho^{-1/2}R^{\de}.
    \end{equation}
\end{enumerate}
\end{definition}
\begin{definition}
\label{transverse-definition}
$\ZT_{B,trans}$ is the set of $\rho$-tubes $T\in\ZT$ obeying the following two conditions:
\begin{enumerate}
    \item [$\bullet$]$T\cap B\not=\varnothing$.
    \item [$\bullet$]There exists a non-singular point $z$ of $Z$ lying in $10 \wt B\cap 10 T$, such that 
    \begin{equation}
    \label{angular-condition-2}
        |{\rm Angle}(v(T),T_zZ(P)|> \rho^{-1/2}R^{\de}.
    \end{equation}
\end{enumerate}
\end{definition}
\noindent The following lemma for the transverse tubes in proved in \cite{Guth-R3}.
\begin{lemma}\label{transversetube}
Each $T\in\T$ belongs to at most ${\poly}(deg(P))={\poly}(d)$ many sets $\T_{B,trans}$.
\end{lemma}

We will not use the lemma below until Section 10. While we feel that this is the best place to state it. First, we cover $B$ using finitely overlapping rescaled balls $\{Q\}$ of radius $M\rho^{1/2+\be}$ that $Q\cap B\not=\varnothing$. Then we define
\begin{equation}
    \ZT_{B,tang,Q}:=\{T\in\ZT_{B,tang},  T\cap 2Q\not=\varnothing\}.
\end{equation}
Then for any $Q$, all tangent tubes intersecting $Q$ are morally lie in a thin neighborhood of a plane. This is proved in the next lemma.
\begin{lemma}
\label{tangent-lemma}
There is a plane $V\subset\ZR^3$ such that every tubes in $\ZT_{B,tang,Q}$ is contained in $N_{\rho^{1/2}R^{O(\de)}}(V)$. 
\end{lemma}

\begin{proof}
A crucial fact we use in the proof is that the directions of all tangent tubes are roughly $c_\si$, which is also the direction of the rescaled ball $Q$. We pick an arbitrary tube $T\in\ZT_{B,tang,Q}$. Since $Q\cap W\not=\varnothing$, there is a non-singular point $z\in Z(P)\cap 10 \wt B$ such that $z\in 2Q$. Let $V=T_z Z(P)$ be the tangent plane at $z$, so from the angular condition \eqref{angular-condition-1}, clearly we have $T\subset N_{\rho^{1/2}R^{O(\de)}}(V)$.

It remains to verify $T'\subset N_{\rho^{1/2}R^{O(\de)}}(V)$ for every other tube $T'\in\ZT_{B,tang,Q}$. Indeed, since $T'\cap Q$ is not empty, and since the direction of $T'$ and the direction of $Q$ make an angle less than $M^{-1}$, the dilated set $2Q$ is contained in $6T'$. Hence for the same non-singular point $z\in Z(P)\cap 10\wt B$ chosen above, $z\in 6T'$. Employing the angular condition \eqref{angular-condition-1} again, we see that $T'\subset N_{\rho^{1/2}R^{O(\de)}}(V)$ as desired.
\end{proof}

Similar to the cellular case, for each $B$ we pick a smooth cutoff function $\vp_{B}$ associated to the $\rho R^{-\de}$-rescaled ball $\wt B$ containing $B$, such that $\vp_B$ satisfies the following properties:
\begin{enumerate}
    \item $\vp_B(x)\sim 1$ when $x\in \wt B$.
    \item If $\wh{B}$ is the dual slab of $\wt B$ centered at the origin whose dimensions are $ MR^\de\rho^{-1}\times MR^\de\rho^{-1}\times R^\de\rho^{-1}$, we require $\wh\vp_B(\xi)\gtrsim1$ when $\xi\in\wh{B}$ and $\wh\vp_B$ is supported in $2\wh{B}$.
\end{enumerate}
Now we define $\vg_{B,tang}$ and $\vg_{B,trans}$ as
\begin{equation}
\label{tangent-transverse-vector}
    \vg_{B,trans}:=\vp_B\sum_{T\in\T_{B,trans}}(\vg_U)_{T},\hspace{5mm}\vg_{B,tang}:=\vp_B \sum_{T\in\T_{B,tang}}(\vg_U)_{T}.
\end{equation}
From the definition of $\vp_B$, we see that the Fourier support of each component $g_{j,B,trans}$ of $\vg_{B,trans}$ is contained in $\wh B+N_{C\rho^{-1}}\Ga_{j}(\si)\subset N_{C'(\rho R^{-\de})^{-1}}\Ga_j(\si)$. This verifies \eqref{transverse4}. Similarly, we can verify \eqref{tangent3}.

By the triangle inequality for broad function \eqref{broad-triangle}, for any $x\in B$, we have
\begin{equation}
\label{broadness-algebraic}
    {\rm Br}_A \Sq\vg_U(x)\leq {\rm Br}_{A/2} \Sq\vg_{B,trans}(x)+{\rm Br}_{A/2} \Sq\vg_{B,tang}(x).
\end{equation}
Consequently, we have
\begin{equation}\label{transverse}
    \int_W|\Br\Sq\vg_U|^p\lesssim \sum_B\int_B|\Br\Sq\vg_{B,trans}|^p+\sum_B\int_B|\Br\Sq\vg_{B,tang}|^p.
\end{equation}

Combining \eqref{cell-only-1}, \eqref{cellular} and \eqref{transverse}, we get
\begin{align}
\label{algorithm-crude}
    \int_{U}|{\rm Br}_A \Sq\vg_U|^p\lesssim&\sum_{O\in\co}\int_{O}|{\rm Br}_{A} \Sq\vg_{O}|^p\\ \nonumber
    &+\sum_{B\in\cb}\int_{B}| {\rm Br}_{A/2} \Sq\vg_{B,trans}|^p\\ \nonumber
    &+\sum_{B\in\cb}\int_{B}| {\rm Br}_{A/2} \Sq\vg_{B,tang}|^p.
\end{align}
Now we determine which one of the three cases we are in according to which term on the right hand side of \eqref{algorithm-crude} dominates.

\textit{Cellular case}:
If the first term on the right hand side of \eqref{algorithm-crude} dominates, we say ``we are in the cellular case". Together with \eqref{cellequal} and \eqref{broadness-cell}, one has
\begin{equation}
\nonumber
    \int_{U}|{\rm Br}_A \Sq\vg_U|^p\lesssim\sum_{O\in\co}\int_{O}|{\rm Br}_{A} \Sq\vg_{O}|^p\le  \sum_{\wt O\in\wt \co}\int_{\wt O}|{\rm Br}_{A} \Sq\vg_{O}|^p\sim \int_{U}|{\rm Br}_A \Sq\vg_U|^p,
\end{equation}
which implies
\begin{equation}
    \int_{U}|{\rm Br}_A \Sq(\vg_U)|^p\sim\sum_{O\in\co}\int_{O}|{\rm Br}_{A} \Sq(\vg_{O})|^p.
\end{equation}
Also noting that $\int_{O}|{\rm Br}_{A} \Sq(\vg_{O})|^p\le \int_{\wt O}|{\rm Br}_{A} \Sq(\vg_{O})|^p$ and $|\co|=|\wt\co|\sim d^3$,
by pigeonholing, we can choose a subset of $\co$ which we still denote by $\co$ such that \eqref{cellular3} holds for every $O\in \co$, and \eqref{cellular1} also holds. To show that \eqref{cellular2} is true, we first use the local $L^2$ estimate, Lemma \ref{localL2}, to obtain
\begin{equation}
\nonumber
    \sum_{O\in\co}\|\Sq\vg_O\|_2^2=\sum_{O\in\co}\big\|\Sq\big(\vp_O\sum_{T\in\T_O}(\vg_U)_{T}\big)\big\|^2\lesssim d^{-1}\sum_{O\in\co}\big\|\Sq\big(\sum_{T\in\T_O}(\vg_U)_{T}\big)\big\|_2^2.
\end{equation}
Then, from Lemma \ref{cellulartube} we note that each tube $T$ belongs to $O(d)$ many sets $\T_O$. This implies
\begin{equation}
    \sum_{O\in\co}\big\|\Sq\big(\sum_{T\in\T_O}(\vg_U)_{T}\big)\big\|_2^2\lesssim d\big\|\Sq(\sum_{T\in\T}(\vg_U)_{T})\big\|^2_2=d\|\Sq\vg_U\|^2_2.
\end{equation}
Combining the above two inequalities, we prove \eqref{cellular2}.

\textit{Transverse case}: If the second term on the right hand side of \eqref{algorithm-crude} dominates, we say ``we are in the transverse case". In this case, we have
\begin{equation}
    \int_{U}|{\rm Br}_A \Sq(\vg_U)|^p\lesssim 
    \sum_{B\in\cb}\int_{B}| {\rm Br}_{A/2} \Sq(\vg_{B,trans})|^p. 
\end{equation}
In order to satisfy \eqref{transverse3}, we use pigeonhole principle on $\int_B|{\rm Br}_{A/2} \Sq(\vg_{B,trans})|^p$ to find a subset of $\cb$, which we still denoted by $\cb$, such that
\begin{equation}
    \int_B| {\rm Br}_{A/2} \Sq(\vg_{B,trans})|^p \textup{~are~comparable~for~}B\in\cb,  
\end{equation}
and
\begin{equation}
    \int_{U}|{\rm Br}_A \Sq(\vg_U)|^p\lesssim 
    \log R \sum_{B\in\cb}\int_{B}| {\rm Br}_{A/2} \Sq(\vg_{B,trans})|^p.
\end{equation}
This verifies \eqref{transverse3}.    
    
To show \eqref{transverse2}, from Lemma \ref{transversetube} we note that that each $T$ belongs to at most ${\poly}(d)$ many sets $\T_{B,trans}$. Combining this fact with Lemma \ref{localL2} and using the same reasoning as in the cellular case, one can show
\begin{equation}
\nonumber
    \sum_{B\in\cb}\|\Sq\vg_{B,trans}\|_2^2\lesssim R^{-\de}\sum_{B\in\cb}\big\|\Sq\big(\sum_{T\in\T_{B,trans}}(\vg_U)_{T}\big)\big\|_2^2
\end{equation}
and
\begin{equation}
\nonumber
    \sum_{B\in\cb}\big\|\Sq\big(\sum_{T\in\T_{B,trans}}(\vg_U)_{T}\big)\big\|_2^2\lesssim {\poly}(d)\big\|\Sq\big(\sum_{T\in\T}(\vg_U)_{T}\big)\big\|^2_2={\poly}(d)\|\Sq\vg_U\|^2_2.
\end{equation}
Combining the above two inequalities, we prove
\eqref{transverse2}.

\textit{Tangent case}: If the third term on the right hand side of \eqref{algorithm-crude} dominates, we say ``we are in the tangent case". The tangent case is easier than cellular case and transverse case, because we don't need to prove the $L^2$-relation like \eqref{cellular2} or \eqref{transverse2}. The proof of \eqref{tangent2} the same as in the transverse case, so we omit the details here.
\end{proof}
\begin{remark}
\rm

Actually, we will only encounter one tangent case in our iteration process, so the estimate for the tangent case is not that important.
\end{remark}

\section{Polynomial partitioning iteration}
\label{iteration}

In this section, we repeatedly use Algorithm \ref{algorithm} to build our iteration. At each step of the iteration, we endow one of the states: cellular state, transverse state and tangent state.
The iteration end when we arrive in the tangent state, or the scale is very small (slightly larger than $M^2$). We will discuss more carefully about these two scenarios later. 

\subsection{Iteration outputs} \hfill

Recall \eqref{sigma-sqfcn-2} and \eqref{vector-f} that $\vf$ implicitly depends on the two factors $\si$ and $M$. Let us do scale $R$ wave packet decomposition for $\vf$, and make the following definition.

\begin{definition}
\label{f-cp-R}
Define $\vf_{\cp_R}$ as the sum of the wave packets $\vf_T$ that $2T\cap \cp_R\not=\varnothing$.
\end{definition}

\noindent The definition implies that $\Sq\vf\lesssim\Sq\vf_{\cp_R}$ when $x\in\cp_R$. Therefore, recalling the fact $\Br_A \Sq f(x)=\Br_A\Sq \vf(x)$ if $x\in\cp_R$, which was derived just after Definition \ref{broad-vector}, we have
\begin{equation}
\label{localized-rescaled-ball}
    \int_{\cp_R}|{\rm Br}_{A} \Sq f|^p=\int_{\cp_R}|{\rm Br}_{A} \Sq\vf|^p\lesssim\int_{\cp_R}|{\rm Br}_{A} \Sq\vf_{\cp_R}|^p.
\end{equation}

\begin{theorem}
{\rm
Fix an integer $A$. Let $\vf=\{f_1,\cdots,f_R\}$ be defined in \eqref{vector-f} so that the Fourier support of each $f_j$ is contained in $N_{CR^{-1}}(\Ga_j(\si))$. Then we have the following outputs:

\medskip

\noindent$\bullet$ There exists an integer $s$ $(0\le s\le \e^{-10})$ which is the total number of iteration steps. There is a function STATE which we use to record the state of each step:
\begin{equation}
    \textup{STATE}:\{1,2,\cdots,s+1\}\rightarrow \{\textup{cell,~trans,~tang}\}.
\end{equation}
We require the tangent case appear at most once, and if it appears, it should only appear at the last step. That is: $\textup{STATE}(u)=$ tang implies $u=s+1$.

\medskip

\noindent$\bullet$ At each step $u$, $u\in\{1,\cdots,s+1\}$, we have:

\smallskip

\noindent{\bf 1.} A scale $r_u$ for which the explicit formula is
\begin{equation}\label{it1.1}
    r_u=Rd^{-s_c(u)}R^{-\de s_t(u) },
\end{equation}
where $r_0=R$ and the two parameters $s_c(u)$ and $s_t(u)$ are defined as
\begin{align}
\nonumber
    &s_c(u):=\#\{1\le i\le u:~\textup{STATE}(i)=\textup{cell}\}, \\ \nonumber
    &s_t(u):=\#\{1\le i\le u:~\textup{STATE}(i):=\textup{trans}\}.
\end{align}

\smallskip

\noindent{\bf 2.} A the number $a(u)$ defined as
\begin{equation}\label{it1.2}
    a(u)=\#\{1\le i\le u:~\textup{STATE}(i)=\textup{trans~or~tang}\}.
\end{equation}
For convenience we also set $a(0)=0$. We will see later that we consider the broad operator $\Br_{A/2^{a(u)}}$ at step $u$.

\smallskip

\noindent{\bf 3.} A set of cells $\co_u=\{O_u\}$ such that each $O_u$ is contained in a $r_u$-rescaled ball $\cp_{O_u}$. For convenience, we set $O_0=\cp_R$. Each $O_u$ has a unique parent $O_{u-1}\in\co_{u-1}$, which we denoted by 
\begin{equation}
    O_u<O_{u-1}.
\end{equation}
Moreover we have the nested property for these cells. That is, for any cell $O_{s+1}\in\co_{s+1}$, there exist unique $O_u\in \co_u$ $(u=1,2,\cdots,s)$ such that
\begin{equation}
    O_{s+1}<O_s<\cdots<O_1<O_0=\cp_R. 
\end{equation}

\smallskip

\noindent{\bf 4.} A set of $r_{u-1}$-tubes $\T_{O_u}[r_{u-1}]$ and a set of functions $\{\vf_{O_u}\}_{O_u\in\co_{u}}$ defined by
\begin{equation}\label{it3}
    \vf_{O_u}:=\vp_{O_u}\sum_{T\in\T_{O_u}[r_{u-1}]}(\vf_{O_{u-1}})_T.
\end{equation}
Here $(\vf_{O_{u-1}})_T$ is a scale $r_{u-1}$ wave packet, and $\vp_{O_u}$ is a smooth cutoff of the cell $O_u$ that $\wh\vp_{O_u}$ is supported in $\wh\cp_{O_u}$ ($\wh\cp_{O_u}$ was defined in item \eqref{rescaled-ball-dual} below Lemma \ref{cellulartube}).

\smallskip

\noindent{\bf 5.} There are three possible cases for each step $u$: cellular case, transverse case and tangent case. The outputs for each case are the following:

\noindent\textit {Cellular state}: If $\textup{STATE}(u)=$ cell, we have the following outputs.

\begin{enumerate}
    \item[i.] The cells $\co_u$ at step $u$ and the preceding cells $\co_{u-1}$ satisfy the following quantitative relation:
    \begin{equation}\label{itcellular1}
        |\co_u|\sim d^3 |\co_{u-1}|.
    \end{equation}

    \item[ii.] We have the following $L^2$-relation between two nearby steps:
    \begin{equation}\label{itcellular2}
        \sum_{O_u}\|\Sq\vf_{O_{u}}\|^2_2\lesssim  \sum_{O_{u-1}}\|\Sq\vf_{O_{u-1}}\|^2_2.
    \end{equation}

    \item[iii.] All the $\int_{O_u}|\Br_{A/2^{a(u)}}\Sq(\vf_{O_u})|^p$ ($O_u\in\co_u$) are same up to a constant factor, and
    \begin{equation}\label{itcellular3}
        \sum_{O_{u-1}}\int_{O_{u-1}}|\Br_{A/2^{a(u-1)}}\Sq\vf_{O_{u-1}}|^p\lesssim \sum_{O_{u}}\int_{O_u}|\Br_{A/2^{a(u)}}\Sq\vf_{O_u}|^p
    \end{equation}
    
    \item[iv.] The Fourier transform of each component $f_{j,O_u}$ of $\vf_{O_u}$ satisfies:
    \begin{equation}\label{itcellular4}
       \textup{supp} \wh{f}_{j,O_u}\subset N_{Cr_u^{-1}}(\Ga_{j}(\si)).  
    \end{equation}
\end{enumerate}

\noindent\textit{Transverse state}: If $\textup{STATE}(u)=$ trans, we have the following outputs.
\begin{enumerate}
    \item[i.] There exists a number $\mu_u$ such that for any $O_{u-1}\in \co_{u-1}$, the quantity $\#\{O_u\in\co_u: O_u<O_{u-1}\}$ is either $0$ or $\sim \mu_u$. In particular, we have
    \begin{equation}\label{ittransverse1}
        |\co_u|\le \mu_u |\co_{u-1}|.
    \end{equation}
    
    \item[ii.] We have the following $L^2$-relation between nearby steps:
    \begin{equation}\label{ittransverse2}
        \sum_{O_u}\|\Sq\vf_{O_{u}}\|^2_2\lesssim  {\poly}(d)R^{-\de}\sum_{O_{u-1}}\|\Sq\vf_{O_{u-1}}\|^2_2.    
    \end{equation}
    
    \item[iii.] All the $\int_{O_u}|\Br_{A/2^{a(u)}}\Sq(\vf_{O_u})|^p$ ($O_u\in\co_u$) are same up to a constant factor, and
    \begin{equation}\label{ittransverse3}
        \sum_{O_{u-1}}\int_{O_{u-1}}|\Br_{A/2^{a(u-1)}}\Sq\vf_{O_{u-1}}|^p\lesssim (\log R)^3 \sum_{O_{u}}\int_{O_u}|\Br_{A/2^{a(u)}}\Sq\vf_{O_u}|^p.
    \end{equation}
    
    \item[iv.] The Fourier transform of each component $f_{j,O_u}$ of $\vf_{O_u}$ satisfies:
    \begin{equation}\label{ittransverse4}
       \textup{supp} \wh{f}_{j,O_u}\subset N_{Cr_u^{-1}}(\Ga_{j}(\si)).  
    \end{equation}
\end{enumerate}

\noindent\textit{Tangent state}: If $\textup{STATE}(u)=$ tang, which means $u=s+1$, then we have the following outputs.
\begin{enumerate}
    \item[i.] The cells $\co_{s+1}$ satisfies
    \begin{equation}\label{ittangent1}
        |\co_{s+1}|\lesssim R^{3\de} |\co_{s}|.
    \end{equation}
    
    \item[ii.] We have the following $L^2$-relation:
    \begin{equation}\label{ittangent2}
        \sum_{O_{s+1}}\|\Sq\vf_{O_{s+1}}\|^2_2\lesssim  {\poly}(d)R^{-\de}\sum_{O_{s}}\|\Sq\vf_{O_{s}}\|^2_2.    
    \end{equation}
    
    \item[iii.] All the  $\int_{O_{s+1}}|\Br_{A/2^{a(s)+1}}\Sq\vf_{O_{s+1}}|^p$ ($O_{s+1}\in\co_{s+1}$) are same up to a constant factor, and
    \begin{equation}\label{ittangent3}
        \sum_{O_{s}}\int_{O_{s}}|\Br_{A/2^{a(s)}}\Sq\vf_{O_{s}}|^p\lesssim (\log R)^3 \sum_{O_{s+1}}\int_{O_{s+1}}|\Br_{A/2^{a(s+1)}}\Sq\vf_{O_{s+1}}|^p.
    \end{equation}
    
    \item[iv.] The Fourier transform of each component $f_{j,O_u}$ of $\vf_{O_u}$ satisfies:
    \begin{equation}\label{ittangent4}
       \textup{supp} \wh{f}_{j,O_u}\subset N_{Cr_u^{-1}}(\Ga_{j}(\si)).  
    \end{equation}
    
\end{enumerate}

}
\end{theorem}

\begin{proof} We are going to iteratively apply the Algorithm \ref{algorithm} from the last section. Let us begin with the initial step $u=1$.

\medskip

{\bf{Initial step}}: For convenience, we set $r_0=R$, $O_0=\cp_R$, $\vf_{O_0}=\vf_{\cp_R}$.  Note that we already have a wave packet decomposition on $\vf_{\cp_R}$ (See Definition \ref{f-cp-R}). Apply Algorithm \ref{algorithm} to the group $(R,O_0,\vf_{O_0},A)$, but without the wave packet decomposition \eqref{wpt-algorithm}. Then there are three possible states: cellular case, transverse case and tangent case. We discuss them separately.

\smallskip

\textit{Cellular state}: If Algorithm \ref{algorithm} results in cellular case, we set 
\begin{equation}
    \textup{STATE}(1)=\textup{cell},\ \ \ r_1=Rd^{-1}.
\end{equation}
From the cellular case in Algorithm \ref{algorithm}, we obtain a collection of cells $\co$, tube sets $\{\T_O\}_{O\in\co}$ and a set of functions $\{\vg_{O}\}_{O\in\co}$ that satisfy \eqref{cellular1}, \eqref{cellular2}, \eqref{cellular3} and \eqref{cellular4}. We define
\begin{equation}
\nonumber
    \co_1:=\co,\ \ \{\T_{O_1}[r_0]\}_{O_1\in\co_1}:=\{\T_O\}_{O\in\co},\ \   \{\vf_{O_1}\}_{O_1\in\co_1}:=\{\vg_O\}_{O\in\co}.
\end{equation}
Also, note that $a(1)=0$ by \eqref{it1.2}.
We can check that \eqref{itcellular1}, \eqref{itcellular2}, \eqref{itcellular3} and \eqref{itcellular4} all holds for $u=1$, since they are exactly \eqref{cellular1}, \eqref{cellular2}, \eqref{cellular3} and \eqref{cellular4} respectively. We can also verify \eqref{it3}, since it is exactly \eqref{cellularfunction}.

\smallskip

\textit{Transverse state}: If Algorithm results in transverse case, we set 
\begin{equation}
    \textup{STATE}(1)=\textup{trans},\ \ \ r_1=R R^{-\de}.
\end{equation}
From the transverse case in Algorithm \ref{algorithm}, we obtain a collection of cells $\cb$, tube sets $\{\T_{B,trans}\}_{B\in\cb}$ and a set of functions $\{\vg_{B,trans}\}_{B\in\cb}$ that satisfy \eqref{transverse2}, \eqref{transverse3} and \eqref{transverse4}. We define
\begin{equation}
\nonumber
    \co_1:=\cb,\ \ \{\T_{O_1}[r_0]\}_{O_1\in\co_1}:=\{\T_{B,trans}\}_{B\in\cb},\ \ \{\vf_{O_1}\}_{O_1\in\co_1}:=\{\vg_{B,trans}\}_{B\in\cb}.
\end{equation}
Also, note that $a(1)=1$ by \eqref{it1.2}.
We can check that \eqref{ittransverse2}, \eqref{ittransverse3} and \eqref{ittransverse4} all holds for $u=1$, since they are exactly \eqref{transverse2}, \eqref{transverse3} and \eqref{transverse4} respectively. To check \eqref{ittransverse1}, we just set
\begin{equation}
    \mu_1=|\co_1|,
\end{equation}
and note that there is only one cell in $\co_0$: $\co_0=\{\cp_R\}$.
We can also verify \eqref{it3}, since it is exactly \eqref{tangent-transverse-vector}.

\smallskip

\textit{Tangent state}: If Algorithm \ref{algorithm} results in cellular case, the iteration stop, so $s=0$. Actually, no matter the tangent state appears in the first step or in the later step, the discussion of tangent state is the same. We include the proof here for clarity.
We set 
\begin{equation}
    \textup{STATE}(1)=\textup{tang},\ \ \ r_1=RR^{-\de}.
\end{equation}
The proof is essentially the same as for the transverse state.
From the tangent case in Algorithm \ref{algorithm}, we obtain a collection of cells $\cb$, tube sets $\{\T_{B,tang}\}_{B\in\cb}$ and a set of functions $\{\vg_{B,tang}\}_{B\in\cb}$ that satisfy \eqref{tangent2}, \eqref{tangent3}. We set
\begin{equation}
\nonumber
    \co_1:=\cb,\ \ \{\T_{O_1}[r_0]\}_{O_1\in\co_1}:=\{\T_{B,tang}\}_{B\in\cb},\ \ \{\vf_{O_1}\}_{O_1\in\co_1}:=\{\vg_{B,tang}\}_{B\in\cb}.
\end{equation}
Also, note that $a(1)=1$ by \eqref{it1.2}.
We can check that \eqref{ittangent2}, \eqref{ittangent3} and \eqref{ittangent4} all holds for $u=1$, since they are exactly \eqref{tangent2}, \eqref{tangent3} and \eqref{tangent4} respectively. To check \eqref{ittangent1}, we just note that the number of tangent cells is $\lesssim R^{3\de}$.
We can also verify \eqref{it3}, since it is exactly \eqref{tangent-transverse-vector}.

\smallskip

The discussion of {\bf{Initial step}} is finished.

\medskip

Next let us move onto any intermediate step. We are going to see how to pass from step $u$ to step $u+1$.

\medskip

{\bf{Iteration step}}: Suppose the iteration is done for step $u$, so we have a scale $r_{u}$, a set of cells $\co_{u}$, tube sets $\{\T_{O_{u}}[r_{u-1}]\}_{O_{u}\in\co_{u}}$ and a set of functions $\{\vf_{O_{u}}\}_{O_{u}\in\co_{u}}$. For each $O_u$, we apply Algorithm \ref{algorithm} to the group $(r_u,O_u,\vf_{O_u},A/2^{a(u)})$. To verify this is a valid input, we note that $O_u$ is contained in a $r_u$-rescaled ball and also \eqref{itcellular4} and \eqref{ittransverse4} verify the condition on the Fourier support of $\vf_{O_u}$.

There are three possible states: cellular state, transverse state and tangent state, depending on the outputs of Algorithm on each $O_u$. We discuss them separately.

\smallskip

\textit{Cellular state}: We say ``the step $u+1$ is in the cellular state", if at least $1/3$ of the cells $O_u\in\co_u$ are in the cellular case. Denote these cells by $\co'_{u}$ so we have $|\co'_u|\ge \frac{1}{3}|\co_u|$.
We set
\begin{equation}
    \textup{STATE}(u+1)=\textup{cell},\ \ \ r_{u+1}=r_u d^{-1}.
\end{equation}
Also note that 
\begin{equation}
    a(u+1)=a(u).
\end{equation}
For each cell $O_{u}\in\co'_u$, we have following outputs due to the Algorithm \ref{algorithm}.
\begin{enumerate}
    \item We obtain a collection of cells denoted by $\co_{u+1}(O_u)=\{O\}$. They satisfy: $|\co_{u+1}(O_u)|\sim d^3$; each $O\in \co_{u+1}(O_u)$ is contained in a $r_{u+1}$- rescaled ball $\cp_O$.
    \item We obtain tube sets $\{\T_O\}_{O\in\co_{u+1}(O_u)}$ and function sets $\{\vg_{O}\}_{O\in\co_{u+1}(O_u)}$ that are indexed by cells in $\co_{u+1}(O_u)$. They satisfy an $L^2$ estimate
    \begin{equation}\label{itcel1}
        \sum_{O\in\co_{u+1}(O_u)}\|\Sq\vg_{O}\|^2\le \|\Sq\vf_{O_u}\|^2_2,
    \end{equation}
    and for each $O\in\co_{u+1}(O_u)$, a broad estimate 
    \begin{equation}\label{itcel2}
        \int_{O_u}|\Br_{A/2^{a(u)}}\Sq\vf_{O_u}|^p\sim d^{3}\int_O |\Br_{A/2^{a(u+1)}}\Sq\vg_O|^p.
    \end{equation}
    
    \item The Fourier support of each component $g_{j,O}$ of $\vg_{O}$ is contained in the set $N_{Cr_{u+1}^{-1}} (\Ga_j(\si))$.  
\end{enumerate}
Now we define respectively
\begin{equation}
\label{celdef1}
    \co_{u+1}:=\bigcup_{O_u\in\co'_u}\co_{u+1}(O_u),
\end{equation}
\begin{equation}\label{celdef2}
    \{\T_{O_{u+1}}[r_{u}]\}_{O_{u+1}\in\co_{u+1}}:=\{\T_O\}_{O\in\co_{u+1}},
\end{equation}
\begin{equation}\label{celdef3}
    \{\vf_{O_{u+1}}\}_{O_{u+1}\in\co_{u+1}}:=\{\vg_O\}_{O\in\co_{u+1}}.
\end{equation}

With all the properties enumerated above, we can prove the following results.
\begin{enumerate}
    \item Note that $|\co_u'|\ge\frac{1}{3}|\co_u|$ and $|\co_{u+1}(O_u)|\sim d^3$. We have
    \begin{equation}
        |\co_{u+1}|=\sum_{O_u\in\co'_u}|\co_{u+1}(O_u)|\sim d^3 |\co'_{u}|  \sim d^3 |\co_{u}|,
    \end{equation}
which verifies \eqref{itcellular1}.
    \item By inequality \eqref{itcel1} and the definition in \eqref{celdef3}, we have
    \begin{align}
        \sum_{O_{u+1}\in\co_{u+1}}\|\Sq\vf_{O_{u+1}}\|^2_2&=\sum_{O_u\in\co_u'}\sum_{O\in\co_{u+1}(O_u)}\|\Sq\vf_{O_{u+1}}\|^2_2\\
        \nonumber&\lesssim\sum_{O_{u}\in\co_u'}\|\Sq\vf_{O_{u}}\|^2_2\le \sum_{O_{u}\in\co_u}\|\Sq\vf_{O_{u}}\|^2_2,
    \end{align}
which verifies \eqref{itcellular2}.

    \item By \eqref{itcel2} and note that all the $\int_{O_u}|\Br_{A/2^{a(u)}}\Sq(\vf_{O_u})|^p$ ($O_u\in\co_u$) are same up to a constant factor, we have that all the
    $\int_{O_{u+1}}|\Br_{A/2^{a(u+1)}}\Sq(\vf_{O_{u+1}})|^p$ ($O_{u+1}\in\co_{u+1}$) are the same up to a different constant. Also note that $|\co_{u}'|\ge\frac{1}{3}|\co_u|$. Thus, we have
    \begin{align}
        \sum_{O_{u}\in\co_u}\int_{O_{u}}|\Br_{A/2^{a(u)}}\Sq\vf_{O_{u}}|^p&\le 3\sum_{O_{u}\in\co'_u}\int_{O_{u}}|\Br_{A/2^{a(u)}}\Sq\vf_{O_{u}}|^p\\
        \nonumber &\lesssim\sum_{O_{u+1}\in\co_{u+1}}\int_{O_{u+1}}|\Br_{A/2^{a(u+1)}}\Sq\vf_{O_{u+1}}|^p.
    \end{align}
This verifies \eqref{itcellular3}.
    \item One also sees that \eqref{itcellular4} holds for $u+1$.
\end{enumerate}
At this point we finish the proof for \textit{Cellular state}.

\smallskip

\textit{Transverse state}: We say ``the step $u+1$ is in the transverse state", if at least $1/3$ of the cells $O_u\in\co_u$ are in the transverse case. Denote these cells by $\co'_u$ so we have $|\co'_u|\ge\frac{1}{3}|\co_u|$. We set
\begin{equation}
    \textup{STATE}(u+1)=\textup{trans},\ \ \ r_{u+1}=r_u R^{-\de},
\end{equation}
and note that 
\begin{equation}
    a(u+1)=a(u)+1.
\end{equation}
For each cell $O_{u}\in\co'_u$, we have the following outputs due to Algorithm \ref{algorithm}.

\begin{enumerate}
    \item We obtain a collection of cells denoted by $\co_{u+1}(O_u)=\{B\}$. Each $B$ is contained in a $r_{u+1}$-rescaled ball.
    \item We obtain tube sets $\{\T_{B,trans}\}_{B\in\co_{u+1}(O_u)}$ and corresponding function sets $\{\vg_{B,trans}\}_{B\in\co_{u+1}(O_u)}$, which are indexed by cells in $\co_{u+1}(O_u)$. They satisfy an $L^2$ estimate
    \begin{equation}\label{ittrans2}
        \sum_{B\in\co_{u+1}(O_u)}\|\Sq\vg_{B,trans}\|^2_2\le {\poly} (d) R^{-\de}\|\Sq\vf_{O_u}\|^2_2,
    \end{equation}
    and for each $B\in\co_{u+1}(O_u)$, a broad estimate
    \begin{align}\label{ittrans3}
        \int_{O_u}|\Br_{A/2^{a(u)}}\Sq\vf_{O_u}|^p\lesssim\log R\sum_{B\in\co_{u+1}(O_u)}\int_{B}|\Br_{A/2^{a(u+1)}}\Sq\vg_{B,trans}|^p.
    \end{align}
    Also, the quantity 
    \begin{equation}
        \int_{B}|\Br_{A/2^{a(u+1)}}\Sq\vg_{B,trans}|^p
    \end{equation}
    are the same up to a constant factor.
    \item The Fourier transform of each component $g_{j,B,trans}$ of $\vg_{B,trans}$ is contained in $N_{C(r_{u+1})^{-1/2}} \Ga_j(\si)$.  
    \end{enumerate}
    
To derive our outputs, we need to work a bit more harder than in the cellular state. We use pigeonhole principle twice to guarantee the uniformity properties \eqref{ittransverse1}, \eqref{ittransverse3}. First we note that \eqref{ittrans3} implies
\begin{equation}
\nonumber
    \sum_{O_u\in\co'_u}\int_{O_u}\!\!\!|\Br_{A/2^{a(u)}}\Sq\vf_{O_u}|^p\lesssim \log R\sum_{O_u\in\co'_u}\sum_{B\in\co_{u+1}(O_u)}\int_{B}|\Br_{A/2^{a(u+1)}}\Sq\vg_{B,trans}|^p.
\end{equation}
Dyadic pigeonholing on $\int_{B}|\Br_{A/2^{a(u+1)}}\Sq(\vg_{B,trans})|^p$, we can find a refinement
\begin{equation}
    \co_{u+1}\subset \bigcup_{O_u\in\co'_u}\co_{u+1}(O_u),
\end{equation}
such that all the $\int_{B}|\Br_{A/2^{a(u+1)}}\Sq(\vf_{B,trans})|^p$ ($B\in \co_{u+1}$) are same up to a constant factor, and
\begin{equation}
\nonumber
    \sum_{O_u\in\co'_u}\int_{O_u}|\Br_{A/2^{a(u)}}\Sq\vf_{O_u}|^p\lesssim (\log R)^2\sum_{B\in\co_{u+1}}\int_{B}|\Br_{A/2^{a(u+1)}}\Sq\vg_{B,trans}|^p.
\end{equation}
Now we replace each $\co_{u+1}(O_u)$ by the refinement $\co_{u+1}\cap\co_{u+1}(O_u)$, but still denote it by $\co_{u+1}(O_u)$. 

The second pigeonhole argument is on the size of $\co_{u+1}(O_u)$. By dyadic pigeonholing, we can find a dyadic number $\mu_{u+1}$ and a subset $\co_u''\subset \co'_u$ such that
\begin{equation}
\label{pigeonhole2.1}
    |\co_{u+1}(O_u)|\sim\mu_{u+1} \textup{~for~every~}O_u\in\co''_u,
\end{equation}  
and
\begin{align}
\label{pigeonhole2.2}
    &\sum_{O_u\in\co'_u}\int_{O_u}|\Br_{A/2^{a(u)}}\Sq\vf_{O_u}|^p\\ \nonumber
    \lesssim& (\log R)^3\sum_{O_u\in\co''_u}\sum_{B\in\co_{u+1}(O_u)}\int_{B}|\Br_{A/2^{a(u+1)}}\Sq\vg_{B,trans}|^p.
\end{align} 
We define respectively
\begin{equation}\label{transdef1}
    \co_{u+1}:=\bigcup_{O_u\in\co''_u}\co_{u+1}(O_u),
\end{equation}
\begin{equation}\label{transdef2}
    \{\T_{O_{u+1}}[r_{u}]\}_{O_{u+1}\in\co_{u+1}}:=\{\T_{B,trans}\}_{B\in\co_{u+1}},
\end{equation}
\begin{equation}\label{transdef3}
    \{\vf_{O_{u+1}}\}_{O_{u+1}\in\co_{u+1}}:=\{\vg_{B,trans}\}_{B\in\co_{u+1}}.
\end{equation}

With all the properties enumerated above, we can show the following results.
\begin{enumerate}
    \item From the second pigeonhole argument \eqref{pigeonhole2.1} and recall that $\co_{u+1}(O_u)=\{O_{u+1}\in\co_{u+1}:O_{u+1}<O_u\}$, we verify \eqref{ittransverse1}.
    
    \item By
    inequality \eqref{ittrans2} and the definition in \eqref{transdef3}, we have
    \begin{align}
        \sum_{O_{u+1}\in\co_{u+1}}\|\Sq\vf_{O_{u+1}}\|^2_2=\sum_{O_u\in\co_u''}\sum_{O\in\co_{u+1}(O_u)}\|\Sq\vf_{O_{u+1}}\|^2_2\\
        \nonumber\lesssim{\poly}(d)R^{-\de}\sum_{O_{u}\in\co_u''}\|\Sq\vf_{O_{u}}\|^2_2\le {\poly}(d)R^{-\de}\sum_{O_{u}\in\co_u}\|\Sq\vf_{O_{u}}\|^2_2,
    \end{align}
which verifies \eqref{ittransverse2}.

    \item By \eqref{ittrans3} and note that all $\int_{O_u}|\Br_{A/2^{a(u)}}\Sq(\vf_{O_u})|^p$ ($O_u\in\co_u$) are same up to a constant factor, we have that all $\int_{O_{u+1}}|\Br_{A/2^{a(u+1)}}\Sq(\vf_{O_{u+1}})|^p$ ($O_{u+1}\in\co_{u+1}$) are the same up to another constant. Also note that $|\co_{u}'|\ge\frac{1}{3}|\co_u|$ and inequality \eqref{pigeonhole2.2}. Hence we have
    \begin{align}
        \sum_{O_{u}\in\co_u}\int_{O_{u}}|\Br_{A/2^{a(u)}}\Sq\vf_{O_{u}}|^p&\le 3\sum_{O_{u}\in\co'_u}\int_{O_{u}}|\Br_{A/2^{a(u)}}\Sq\vf_{O_{u}}|^p\\ \nonumber &\lesssim\sum_{O_{u+1}\in\co_{u+1}}\int_{O_{u+1}}|\Br_{A/2^{a(u+1)}}\Sq\vf_{O_{u+1}}|^p.
    \end{align}
This verifies \eqref{ittransverse3}.
    \item One also sees that \eqref{ittransverse4} holds for $u+1$.
\end{enumerate}
We finish the proof for \textit{Transverse state}.

\smallskip  

\textit{Tangent state}: We say ``the step $u+1$ is in the tangent state", if at least $1/3$ of the cells $O_u\in\co_u$ are in the tangent case. Actually, one may not encounter the tangent state throughout the iteration. However, once the tangent state appears, the iteration stops and so we have $u=s$. 

We can proceed in exactly the same way as we did for the \textup{transverse state}. We define $\co_{s+1}$, $\{f_{O_{s+1}}\}_{O_{s+1}\in\co_{s+1}}$ and $\{\T_{O_{s+1}}[r_s]\}$ in the same way as we did for the \textup{transverse state} (see \eqref{transdef1}, \eqref{transdef2} and \eqref{transdef3}). 
The proofs for the properties \eqref{ittangent2}, \eqref{ittangent3} and \eqref{ittangent4} are the same as in the \textup{transverse case}, so we omit the details. For the proof of \eqref{ittangent1}, we just note that each cell $O_s$ has at most $R^{3\de}$ children $O_{s+1}$ from $\co_{s+1}$.
\end{proof}

\subsection{Iteration formulae between functions at different scales}\hfill

We also want to keep track of the  expression of the step-$u$ function $\vf_{O_u}$. Note that at each step $u$, there is a family of cells $\cO_u$. For each cell $O_u\in\co_u$, there is a vector-valued function
\begin{equation}
    \vf_{O_u}=\{f_{1,O_u},\cdots,f_{R,O_u}\}
\end{equation}
given in \eqref{it3} that
\begin{equation}\label{functioniteration}
    \vf_{O_u}=\vp_{O_u}\sum_{T\in\T_{O_u}[r_{u-1}]}(\vf_{O_{u-1}})_T.
\end{equation}
In step $u$ of our iteration, we do polynomial partitioning with respect to the pair $(O_u,|\Br_{A/2^{a(u)}}\Sq\vf_{O_u}|^p)$. After doing so, we obtain another family of cells $\co_{u+1}$, tube sets $\{\ZT_{O_{u+1}}[r_u]\}_{O_{u+1}\in \co_{u+1}}$ at scale $r_u$, and function sets $\{\vf_{O_{u+1}}\}_{O_{u+1}\in\co_{u+1}}$. Also recall that we say ``$O_{u+1}$ is the child of $O_{u}$", if $O_{u+1}$ is obtained from doing polynomial partitioning with respect to $O_{u}$, and we denote this relation by
\begin{equation}
    O_{u+1}<O_u.
\end{equation}
We also have the nested property for cells. That is, for any cell $O_{s+1}\in\co_{s+1}$, there exists a unique $O_u\in\co_u$ ($u=1,2\cdots,s$) such that
\begin{equation}
     O_{s+1}<O_s<\cdots<O_1<O_0=\cp_R.
\end{equation} 

To track some finer structures of the function $\vf_{O_u}$, recall that in \eqref{it3} we have introduced a collection of $r_u$-tubes, $\ZT_{O_u}[r_{u-1}]$. In this collection, each tube is pointing to a direction $c_{\om_{u-1}}$ for some $r_{u-1}$-cap $\om_{u-1}$. For such cap $\om_{u-1}$, let us make the following definitions.

\begin{definition}\label{smalltube1}
Suppose $1\leq u\leq s+1$ and suppose that we have the convention $\om_0=\theta$ for some $R^{-1/2}$-cap $\theta$ (See also \eqref{scale-R-tube-set} and \eqref{scale-rho-tube-set}). Let $\om_{u-1}\subset\ZR^2$ be an $r_{u-1}$-cap. Define
\begin{equation}
    \ZT_{O_{u},\om_{u-1}}[r_{u-1}]:=\{T:T\in\ZT_{O_{u}}[r_{u-1}],~T\textup{~has~direction~}c_{\om_{u-1}}\}.
\end{equation}
\end{definition}




For any $u\in\{0,1,\cdots,s,s+1\}$, we define
\begin{equation}\label{caprelation1}
    \Om_u:=\{\om_u:\ZT_{O_{u+1},\om_{u}}[r_u]\not=\varnothing,~\text{for~some~}O_{u+1}\in\co_{u+1}\},
\end{equation}
so $\Om_u$ is the collection of $r_u$-cap appear in the scale $r_u$ wave packet decomposition. Note that we can endow a nested property with $\{\Om_u\}_{u=1}^{s+1}$ (See Definition \ref{caprelation}): For any $\om_{u-1}\in\Om_{u-1}$, there is a unique $\om_{u}\in\Om_{u}$ so that 
\begin{equation}\label{caprelation2}
     \om_{u-1}<\om_{u}.
\end{equation}
This implies that for any cap $\om_{1}\in \Om_{1}$, there exists a unique $\om_u\in\Om_u$ ($u=2,\cdots,s+1$) such that
\begin{equation}\label{caprelation3}
     \om_{1}<\om_2<\cdots<\om_{s+1}\subset\ZS^2.
\end{equation} 

 Now we can derive a finer version of \eqref{functioniteration}. 

\begin{definition}
For $2\leq u\leq s+1$, $O_u\in\co_u$ and $\om_u\in\Om_u$, define
\begin{equation}\label{inductiveformula1}
    \vf_{O_{u},\om_u}:=\vp_{O_u}\!\!\!\sum_{\om_{u-1}<\om_{u}}\sum_{T\in\ZT_{O_{u},\om_{u-1}}[r_{u-1}]}\!\!\!(\vf_{O_{u-1}})_T=\vp_{O_u}\sum_{\om_{u-1}<\om_u}\vf_{O_{u-1},\om_{u-1}}.
\end{equation}
For $u=1$, $O_1\in\co_1$ and $\om_1\in\co_1$, define
\begin{equation}
\label{inductiveformulaR}
    \vf_{O_1,\om_1}:=\vp_{O_1}\sum_{\theta<\om_1}\sum_{T\in\ZT_{O_1,\theta}[R]}\vf_T=\vp_{O_1}\sum_{\theta<\om_1}\sum_{T\in\ZT_{O_1,\theta}[R]}\vf_{\theta}\Id_{T}^\ast
\end{equation}

\end{definition}

\noindent One can check that the Fourier support of each component $f_{j,O_u,\om_u}$ of $\vf_{O_u,\om_u}$ is contained in $N_{C r_u^{-1}}\Ga_{j}(\om_u)$, which is roughly a $r_u^{-1/2}\times r_u^{-1/2}\times r_u^{-1}$-slab. Also, 
\begin{equation}\label{inductiveformula2}
    \vf_{O_{u}}=\vp_{O_u}\sum_{\om_{u}\in\Om_{u}}\Big(\sum_{\om_{u-1}<\om_{u}}\sum_{T\in\ZT_{O_{u},\om_{u-1}}[r_{u-1}]}\!\!\!\!\!\vf_{O_{u-1},\om_{u-1}}\Id_T^\ast\Big) 
    =\sum_{\om_{u}\in\Om_{u}}\vf_{O_{u},\om_{u}}.
\end{equation}
Together with Definition \ref{f-cp-R}, we also have
\begin{lemma}
\label{f-cp-R-lem}
Let $\vf_{O_{u},\om_u}$ be defined \eqref{inductiveformula1}. Then 
\begin{equation}
\label{f-cp-R-esti}
    \vf_{O_{u},\om_u}\lesssim\vf_{\cp_R}.
\end{equation}
\end{lemma}

\subsection{The main results from the iteration}\hfill

Let us make a conclusion of the polynomial partitioning iteration. There are two scenarios. The first one is when STATE$(s+1)$=tang, which means the iteration ends up in the tangent state. The second one is when STATE$(s+1)\neq$tang, which means the iteration ends up with a tiny scale $r_s\le R^{\e/10} M^2$.

We define two numbers 
\begin{align}
    s_c &:=\#\{1\le u\le s: \textup{STATE}(u)=\textup{cell}\},\\  s_t &:=\#\{1\le u\le s: \textup{STATE}(u)=\textup{trans}\}. 
\end{align}
It implies that 
\begin{equation}
    r_s=Rd^{-s_c}R^{-\de s_t}.
\end{equation}
From now on we write $r=r_s$ for simplicity. Henceforth, recalling \eqref{localized-rescaled-ball}, we can draw a conclusion:

\smallskip

If we are in the \textit{first scenario}, we have the theorem:
\begin{theorem}[Tangent case]\label{scenario1}
If \textup{STATE}$(s+1)$=\textup{tang}, and $s_c,s_t,r$ is defined as above, we have

\begin{equation}\label{scenario1.1}
\int_{\cp_R}|\Br_A\Sq f|^p\lesssim (\log R)^{3s_t}\sum_{O_{s+1}\in\co_{s+1}}\int_{O_{s+1}}|\Br_{A/2^{a(s+1)}}\Sq\vf_{O_{s+1}}|^p.
\end{equation}

\begin{equation}\label{scenario1.2}
\int_{O_{s+1}}\!\!\!|\Br_{A/2^{a(s+1)}}\Sq\vf_{O_{s+1}}|^p~(O_{s+1}\in\co_{s+1})\textup{~are~the~same~up~to~a ~constant}.
\end{equation}

\begin{equation}\label{scenario1.3}
    \sum_{O_{s+1}\in\co_{s+1}}\|\Sq\vf_{O_{s+1}}\|^2_2\lesssim {\poly}(d)^{s_t}R^{-\de s_t}\|\Sq\vf_{\cp_R}\|_2^2.
\end{equation}

\begin{equation}\label{scenario1.4}
    |\co_{s+1}|\gtrsim d^{3s_c}\prod_{\textup{STATE}(u)=\textup{trans}} \mu_u.
\end{equation}
\end{theorem}

If we are in the \textit{second scenario}, we have the theorem:
\begin{theorem}[Small-$r$ case]\label{scenario2}
If \textup{STATE}$(u)\neq$ \textup{tang} for all $u\in\{1,\cdots,s\}$, and $s_c,s_t,r$ is defined as above, we have

\begin{equation}\label{scenario2.1}
\int_{\cp_R}|\Br_A\Sq f|^p\lesssim (\log R)^{3s_t}\sum_{O_{s}\in\co_{s}}\int_{O_{s}}|\Br_{A/2^{a(s)}}\Sq\vf_{O_{s}}|^p.
\end{equation}

\begin{equation}\label{scenario2.2}
\int_{O_{s}}|\Br_{A/2^{a(s)}}\Sq\vf_{O_{s}}|^p~ (O_{s}\in\co_{s})\textup{~are~the~same~up~to~a~constant}.
\end{equation}

\begin{equation}\label{scenario2.3}
    \sum_{O_{s}\in\co_{s}}\|\Sq\vf_{O_{s}}\|^2_2\lesssim {\poly}(d)^{s_t}R^{-\de s_t}\|\Sq\vf_{\cp_R}\|_2^2.
\end{equation}

\begin{equation}\label{scenario2.4}
    |\co_{s}|\gtrsim d^{3s_c}\prod_{\textup{STATE}(u)=\textup{trans}} \mu_u.
\end{equation}

\begin{equation}\label{scenario2.5}
    r\lesssim R^{\e/10} M^2.
\end{equation}

\end{theorem}

\section{End when the radius \texorpdfstring{$r_s$}{Lg} is small}
In this section we discuss the second scenario which is much easier than the first scenario.
We will use the ineqaulities stated in Theorem \ref{scenario2}. First, we consider \eqref{scenario2.1}. Since terms on the right hand side of \eqref{scenario2.1} are same up to a constant factor, by pigeonholing, there exists a cell $O_s$ with minimal $\|\Sq\vf_{O_s}\|_2^2$ such that
\begin{equation}
\label{lp-end-small}
    \int_{\cp_R}\!\!|{\rm Br}_{A} \Sq f|^p\lesssim R^{O(\be)}|\co_s|\int_{O_{s}}\!\!|\Br_{A/2^{a(s)}}\Sq\vf_{O_{s}}|^p\lesssim R^{O(\be)}|\co_s|\int_{O_s}\!\!|\Sq\vf_{O_s}|^p.
\end{equation}
The minimality of $\|\Sq\vf_{O_s}\|_2^2$ and \eqref{scenario2.3} yield that
\begin{equation}
\label{l2-end-small}
    \|\Sq\vf_{O_s}\|_2^2\lesssim R^{O(\de)}R^{-s_t\de}|\co_s|^{-1}\|\Sq\vf_{\cp_R}\|_2^2.
\end{equation}

To estimate each term on the right hand side of \eqref{lp-end-small}, we need the following lemma.
\begin{lemma}
For $2\le p\le\infty$,
\begin{equation}
\label{estimate-end-small}
    \|\Id_{O_s}\Sq\vf_{O_s}\|_p\lesssim M^{4(\frac{1}{p}-\frac{1}{2})}\|\Sq\vf_{O_s}\|_2
\end{equation}
\end{lemma}

\begin{proof}
Recalling the definition of $\Sq\vf_{O_s}$ in Definition \ref{defsqfunc}, we have
\begin{equation}
    \|\Id_{O_s}\Sq\vf_{O_s}\|_p=\Big(\int_{O_s}\Big(\sum_j |f_{j,O_s}|^2\Big)^{p/2}\Big)^{1/p}.
\end{equation}
By H\"older's inequality, it suffices to prove \eqref{estimate-end-small} when $p=2$ and $p=\infty$. The case $p=2$ follows easily from Placherel.

Let us consider the case for $p=\infty$. We note that $\wh{f}_{j,O_s}$ is supported in a slab of dimensions $M^{-1}\times M^{-1}\times M^{-2}$, so via Bernstein's inequality,
\begin{equation}
\label{bernstein}
    \|f_{j,O_s}\|_\infty\lesssim M^{-2}\|f_{j,O_s}\|_2.
\end{equation}
Now we can estimate $\|\Sq(\vf_{O_s})\Id_{O_s}\|_p$ as
\begin{equation}
    \|\Sq(\vf_{O_s})\Id_{O_s}\|_\infty=\sup_{x\in O_s}\Big(\sum_j |f_{j,O_s}(x)|^2\Big)^{1/2}\le \Big(\sum_j \sup_{x\in O_s}|f_{j,O_s}(x)|^2\Big)^{1/2},
\end{equation}
which, using the Berstein's estimate \eqref{bernstein}, is bounded above by
\begin{equation}
    M^{-2}\Big(\sum_j \|f_{j,O_s}\|_2^2\Big)^{1/2}=M^{-2}\|\Sq\vf_{O_s}\|_2.
\end{equation}
This is the desired estimate when $p=\infty$.
\end{proof}
Combining \eqref{lp-end-small}, \eqref{l2-end-small} and \eqref{estimate-end-small}, we have
\begin{equation}
    \int_{\cp_R}|{\rm Br}_{A} \Sq f|^p\lesssim R^{O(\be)}R^{-\frac{ps_t\de }{2}}|\co_s|^{1-\frac{p}{2}}M^{4-2p}\|\Sq\vf_{\cp_R}\|_2^p.
\end{equation}
Now recalling $d^{s_c}R^{s_t\de}=Rr^{-1}\ge M^{-2}R^{1-\e/10}$ and
$|\co_s|\gtrsim d^{3s_c}$ in \eqref{scenario2.4}, simple calculations yield that for $p>3$,
\begin{equation}
\label{end-small-auxiliary-1}
    \int_{\cp_R}|\Br_{A}\Sq f|^p\lesssim R^{p\e/3}R^{-p/2}M^{4-p}\|\Sq\vf_{\cp_R}\|_2^p.
\end{equation}
Recall that $\vf_{\cp_R}$ is the sum of the  wave packet $\vf_T$ that $2T\cap \cp_R\not=\varnothing$. Hence one can choose a function $\vp_{\cp_R}$ such that
\begin{enumerate}
    \item $\vp_{\cp_R}\lesssim w_{\cp_R}$, recalling \eqref{weight-cp-R},
    \item $\wh\vp_{\cp_R}$ is supported a the dual slab of $\cp_R$ centered at the origin,
    \item $\vf_{\cp_R}\lesssim R^{O(\be)}\vp_{\cp_R}\vf$.
\end{enumerate}
It implies $\|\Sq\vf_{\cp_R}\|_2^2\lesssim R^{O(\be)}\|\vp_{{\cp_R}}\Sq\vf\|_2^2$ from \eqref{weight-cp-R}. By $L^2$ orthogonality, 
\begin{equation}
\label{end-small-auxiliary-2}
    \|\vp_{{\cp_R}}\Sq\vf\|_2^2\lesssim\int\sum_{\theta\in\Theta}\sum_{j=1}^R|m_{j,\theta}\ast f|^2\vp_{\cp_R}\lesssim\int\sum_{\theta\in\Theta}\sum_{j=1}^R|m_{j,\theta}\ast f|^2w_{\cp_R}.
\end{equation}
We let $\wt T_{\theta}$ be the tube in $\ZT_{\theta}$ that centered at the origin. Then the kernel estimate \eqref{lemdecaykernel} yields that for fixed $\theta$, the collection of smooth functions $\{R^{-O(\be)}m_{j,\theta}\}_{1\leq j\leq R}$ is adapted to $\wt T_{\theta}$. By Lemma \ref{weightedlemma} and recalling \eqref{lattice-sqfcn}, \eqref{q-tau}, we can argue similarly as in Lemma \ref{half-lem} to get
\begin{equation}
\label{end-small-auxiliary-3}
    \int\sum_{\theta\in\Theta}\sum_{j=1}^R|m_{j,\theta}\ast f|^2w_{\cp_R}\lesssim R^{O(\be)}\int\sum_{\theta\in\Theta}\sum_{q\in\bq(\theta)}|\De_qf|^2\wt w_{\wt T_{\theta},N}\ast w_{\cp_R}.
\end{equation}
Note that tube $\wt T_\theta$ is contained in the rescaled ball $\cp_R$. Hence $\sup_\theta \wt w_{\wt T_{\theta},N}\ast w_{\cp_R}\lesssim w_{\cp_R}$, which, via Lemma \ref{finiely-overlap-lattice-cube}, implies 
\begin{equation}
\label{end-small-auxiliary-4}
    \int\sum_{\theta\in\Theta}\sum_{q\cap\Ga(\theta)\not=\varnothing}|\De_qf|^2 \wt w_{\wt T_{\theta},N}\ast w_{\cp_R}\lesssim\int \sum_{q\in\bq}|\De_qf|^2 w_{\cp_R}.
\end{equation}
Combining \eqref{end-small-auxiliary-1}, \eqref{end-small-auxiliary-2}, \eqref{end-small-auxiliary-3} and \eqref{end-small-auxiliary-4}, we end up with
\begin{equation}
\nonumber
    \int_{\cp_R}|{\rm Br}_{A} \Sq f|^p\leq C_\e R^{p\e} R^{p-3} M^{6-2p}\int\Big(\sum_{q\in\bq}|\De_qf|^2\Big)^{p/2}w_{{\cp_R}},
\end{equation}
which is just \eqref{local-broad-thm}. \qed


\section{Backward algorithm}\label{backward}
In this section, we discuss the first scenario. That is, the iteration ends with STATE($s+1$)=Tang. In this case, we use Theorem \ref{scenario1}. However, when doing so, we lose some global information among cells in $\co_{s+1}$. The new ingredient is to build up a backward algorithm as in \cite{Wu} to sum up $\|\Sq\vf_{O_{s+1}}\|^2_2$ for $O_{s+1}\in\co_{s+1}$ efficiently. We remark that the formulas \eqref{inductiveformula1}, \eqref{inductiveformulaR} and \eqref{inductiveformula2} that relate functions at different scales are very important in this section.

\smallskip

Let us introduce more notations for the backward algorithm.
\begin{definition}[partial order]
\label{partial-order-tube}
For two tubes $T$ and $T'$ at two different scales, we say $T<T'$ if
\begin{equation}
    T\subset 10T'\ and\ \om_{T'}<\om_T. 
\end{equation}
Here $\om_T\subset\ZS^2$ denotes the dual cap of $T$.
\end{definition}

Recall that we defined $\vf_{O_u,\om_u}$ $(1\le u\le s+1)$ in \eqref{inductiveformula1}. Since the step $s+1$ is special, we make the following definition.

\begin{definition}
Given $O_{s+1}\in\co_{s+1}$. Suppose $O_{s+1}<O_s$ and suppose that $\om_s\in\Om_s$ is an $r_s^{-1/2}$-cap. Define
\begin{equation}\label{f_s+1}
    \vf_{O_{s+1},\om_s}:=\sum_{T\in\T_{O_{s+1},\om_s}[r_s]}(\vf_{O_s})_{T}=\sum_{T\in\T_{O_{s+1},\om_s}[r_s]}\vf_{O_s,\om_s}\Id_T^\ast.
\end{equation}
\end{definition}
\begin{remark}
\rm

$\vf_{O_{s+1},\om_s}$ is a sum of parallel wave packets at scale $r_s$ and with direction $c_{\om_s}$. Also note that we didn't multiply the cutoff function $\vp_{O_{s+1}}$ in \eqref{f_s+1}.
To get familiar with this definition, readers could check
\begin{equation}
    \vf_{O_{s+1},\om_{s+1}}=\vp_{O_{s+1}}\sum_{\om_s<\om_{s+1}}\vf_{O_{s+1},\om_s}.
\end{equation}

\end{remark}

Recall that $\be=\e^{1000}$ was fixed in Section 4. For any $a\in[-10R^{1+\be},10R^{1+\be}]$, define $L_a$ as the horizontal plane
\begin{equation}
    L_a:=\{x:x_3=a\}.
\end{equation}
For any tube $T$, we define $\mathring T$ as a stretch of $T$: the tube $\mathring T$ has the same coreline and cross section as $T$, but has infinite length. 

\begin{remark}
\rm

The annoying factor $\be$ is used for handling Schwartz tails. To grasp the main idea, one may set $\be=0$ throughout this section.

\end{remark}

In the backward algorithm, we are going to find a refinement $\bar{\cO}_u\subset\co_u$ for each step $u$. And for each $O_u\in\bar\co_u$, we will build up an auxiliary tube set $\bar\ZT_{O_u}[r_{u-1}]$ and vectors $\vh_{O_u,\om_u}$ for each directional cap $\om_u\in\Om_u$. This is discussed in the next theorem.



\begin{theorem}
\label{backward-algorithm}
For every $u\in\{1,\cdots,s,s+1\}$, we can find a refinement of cell $\bar{\cO}_u\subset\co_u$ and a refinement of tubes $\{\bar\ZT_{O_u}[r_{u-1}]\}_{O_u\in\bar{\cO}_u}$ with $\bar\ZT_{O_u}[r_{u-1}]\subset\ZT_{O_u}[r_{u-1}]$. Recalling \eqref{inductiveformula1} and \eqref{inductiveformula2}, when $u\geq2$, define the vectors $\{\vh_{O_u,\om_u}\}_{O_u\in\bar{\cO}_u}$ for each $\om_u\in\Om_u$ via the tube set $\bar\ZT_{O_u}[r_{u-1}]$ as
\begin{equation}
\label{formula-for-h}
    \vh_{O_{u},\om_{u}}:=\vp_{O_u}\sum_{\om_{u-1}<\om_{u}}\sum_{T\in\bar\ZT_{O_{u},\om_{u-1}}[r_{u-1}]}\vf_{O_{u-1},\om_{u-1}}\Id_{T}^\ast;
\end{equation}
when $u=1$, recalling \eqref{inductiveformulaR}, define $\vh_{O_1,\om_1}$ as
\begin{equation}
\label{formula-for-h-2}
    \vh_{O_1,\om_1}:=\vp_{O1}\sum_{\theta<\om_1}\sum_{T\in\bar\ZT_{O_1,\theta}[R]}\vf\Id_{T}^\ast;
\end{equation}

The vectors $\vh_{O_u,\om_u}$ satisfy the following properties depending on whether we are in cellular state or transverse state at step $u$, where $1\leq u<s$.

{\bf{Cellular state}}:
If \textup{STATE}$(u+1)=$\textup{cell}, we can also find an integer $v_{u+1}$, so that the following two properties are satisfied:

\begin{align}
\label{cell1}
    \sum_{O\in \bar{\co}_{u+1}}\sum_{\om\in\Om_{u+1}}\|\Sq\vh_{O,\om}\|_2^2\lesssim&\, R^{O(\be)}d^{-1}v_{u+1}\sum_{O\in\bar{\cO}_{u}}\sum_{\om\in\Om_u}\|\Sq\vh_{O,\om}\|_2^2\\ \nonumber
    &+\rap(R)\|\Sq\vf_{\cp_R}\|_2^2.
\end{align}
Define $L_{a,O_u}:=L_a\bigcap \{\cup_{T\in \bar\ZT_{O_u}[r_{u-1}]}\mathring T\}$. Then uniformly for $a\in [-R^{1+10\beta},R^{1+10\beta}]$,
\begin{equation}
\label{cell2}
    \max_{O_u\in\bar{\cO}_u}\!\!\big|N_{R^{1+10\be}r_{u}^{-1/2}}L_{a,O_{u}}\big|\lesssim R^{O(\be)}d^3v_{u+1}^{-1}\!\!\max_{O_{u+1}\in\bar{\cO}_{u+1}}\big|N_{R^{1+10\be}r_{u+1}^{-1/2}}L_{a,O_{u+1}}\big|.
\end{equation}

{\bf{Transverse state}}: If \textup{STATE}$(u+1)=$\textup{trans}, the following two properties are satisfied:

\begin{equation}
\label{trans1}
    \sum_{O\in \bar{\co}_{u+1}}\sum_{\om\in\Om_{u+1}}\|\Sq\vh_{O,\om}\|_2^2\lesssim {\rm{Poly}}(d)R^{-\de}\sum_{O\in\bar{\cO}_{u}}\sum_{\om\in\Om_u}\|\Sq\vh_{O,\om}\|_2^2.
\end{equation}
Define $L_{a,O_u}:=L_a\bigcap \{\cup_{T\in \bar\ZT_{O_u}[r_{u-1}]}\mathring T\}$. Then uniformly for $a\in[-R^{1+10\be},R^{1+10\be}]$, 
\begin{equation}
\label{trans2}
    \max_{O_u\in\bar{\cO}_u}\big|N_{R^{1+10\be}r_{u}^{-1/2}}L_{a,O_{u}}\big|\lesssim R^{O(\be)} \mu_{u+1}\max_{O_{u+1}\in\bar{\cO}_{u+1}}\big|N_{R^{1+10\be}r_{u+1}^{-1/2}}L_{a,O_{u+1}}\big|.
\end{equation}

{\bf {Tangent state}}: If \textup{STATE}$(u+1)$=\textup{tang} which means $u=s$, the following two properties are satisfies:
\begin{equation}\label{tang1}
\sum_{O_{s+1}\in \bar{\co}_{s+1}}\sum_{\om_s\in\Om_{s}}\|\Sq\vf_{O_{s+1},\om_s}\|_2^2\lesssim R^{O(\de)}\sum_{O\in\bar{\cO}_{s}}\sum_{\om\in\Om_s}\|\Sq\vh_{O,\om}\|_2^2.  
\end{equation}
Define $L_{a,O_s}:=L_a\bigcap \{\cup_{T\in \bar\ZT_{O_s}[r_{s-1}]}\mathring T\}$. Then uniformly for $a\in[-R^{1+10\be},R^{1+10\be}]$, 
\begin{equation}
\label{tang2}
    \max_{O_s\in\bar{\cO}_s}\big|N_{R^{1+10\be}r_{s}^{-1/2}}L_{a,O_{s}}\big|\lesssim R^{O(\de)} \max_{O_{s+1}\in\bar{\cO}_{s+1}}\big|N_{R^{1+10\be}r_{s+1}^{-1/2}}L_{a,O_{s+1}}\big|.
\end{equation}

\end{theorem}

\begin{remark}
\rm
One can see that \eqref{trans2} and \eqref{tang2} are the same, but \eqref{trans1} and \eqref{tang1} are different because the summand on the left hand side of \eqref{tang1} is $\vf_{O_{s+1},\om_s}$, not $\vf_{O_{s+1},\om_{s+1}}$. The readers can recall the definition of $\vf_{O_{s+1},\om_s}$ in \eqref{f_s+1}. 
\end{remark}

We divide the proof of Theorem \ref{backward-algorithm} into two main cases: cell case and transverse case. The proof relies on a backward induction from the last step $s$ to the first step. Let us first settle down the base case $s+1\to s$.

\subsection{The base case \texorpdfstring{$s+1\to s$}{Lg}.} \hfill

In step $s$, which is the first step of the backward algorithm, we set $\bar\co_{s}:=\co_{s}$. For each $O_s\in\bar\co_s$, the refined tube set $\bar\ZT_{O_s}[r_{s-1}]$ is defined as
\begin{equation}
\label{base-refined-tubes}
    \bar\ZT_{O_s}[r_{s-1}]:=\{T'\in\ZT_{O_s}[r_{s-1}]:\exists T\in\ZT_{O_{s+1}}[r_s], T<T',~\textup{for some}~O_{s+1}<O_s\}.
\end{equation}
We use the set $\bar\ZT_{O_s}[r_{s-1}]$ to define the vector $\vh_{O_s,\om_s}$ as in \eqref{formula-for-h}. Then, recalling the definition of $\vf_{O_{s+1},\om_s}$ in \eqref{f_s+1} and inductive formula \eqref{inductiveformula1}, we have when $O_{s+1}<O_s<O_{s-1}$,
\begin{align}
    \vf_{O_{s+1},\om_s}&=\sum_{T\in\ZT_{O_{s+1},\om_s}[r_s]}\vf_{O_s,\om_s}\Id^\ast_T\\ \nonumber
    &=\sum_{T\in\ZT_{O_{s+1},\om_s}[r_s]}(\varphi_{O_s}\sum_{\om_{s-1}<\om_s}\sum_{T'\in\T_{O_s,\om_{s-1}}[r_{s-1}]}\vf_{O_{s-1},\om_{s-1}}\Id^*_{T'})\Id^*_T
\end{align}
Note that in the above equation, $\Id^*_{T'}\Id^*_T=\rap(R)$ if $T\cap 10T'=\varnothing$, which means that $\Id^*_{T'}\Id^*_T$ is negligible unless $T<T'$. Hence, we can replace the summation  $\T_{O_s,\om_{s-1}}[r_{s-1}]$ by $\bar\T_{O_s,\om_{s-1}}[r_{s-1}]$ and write
\begin{align}
\label{rapiddecay}
    \vf_{O_{s+1},\om_s}=&\sum_{T\in\ZT_{O_{s+1},\om_s}[r_s]}(\varphi_{O_s}\sum_{\om_{s-1}<\om_s}\sum_{T'\in\bar\T_{O_s,\om_{s-1}}[r_{s-1}]}\vf_{O_{s-1},\om_{s-1}}\Id^*_{T'})\Id^*_T\\[1ex] \nonumber
    &+\rap (R)\vf_{O_{s-1},\om_{s-1}}\\[1ex]
    \nonumber
    =&\sum_{T\in\ZT_{O_{s+1},\om_s}[r_s]}\vh_{O_s,\om_s}\Id^\ast_T+\rap (R)\vf_{O_{s-1},\om_{s-1}}.
\end{align}
From Lemma \ref{f-cp-R-lem}, one gets $\Sq\vf_{O_{s-1},\om_{s-1}}\lesssim\Sq\vf_{\cp_R}$, which yields 
\begin{equation}
    \|\Sq\vf_{O_{s+1},\om_s}\|_2^2\le \|\Sq\vh_{O_s,\om_s}\|_2^2+ \rap(R)\|\Sq\vf_{\cp_R}\|_2^2.
\end{equation}

Therefore, noticing that for fixed $O_s\in\co_s$ there are $R^{O(\de)}$ many step-$(s+1)$ cells $O_{s+1}$ with $O_{s+1}<O_s$, one gets that on one hand
\begin{align}
\label{l2-base-case}
    \sum_{O\in \bar{\co}_{s+1}}\sum_{\om\in\Om_{s}}\|\Sq\vf_{O,\om}\|_2^2\lesssim& R^{O(\de)}\sum_{O\in\bar{\cO}_{s}}\sum_{\om\in\Om_s}\|\Sq\vh_{O,\om}\|_2^2\\ \nonumber
    &+\rap(R)\|\Sq\vf_{\cp_R}\|_2^2.
\end{align}
One the other hand, if we define $L_{a,O_{s+1}}:=L_a\bigcap \{\cup_{T\in \bar\ZT_{O_{s+1}}[r_s]}\mathring T\}$, then
\begin{equation}
\label{base-tangent}
    \max_{O_s\in\bar{\cO}_s}\big|N_{R^{1+10\be}r_{s}^{-1/2}}L_{a,O_{s}}\big|\lesssim R^{3\de}\max_{O_{s+1}\in\co_{s+1}}\big|N_{R^{1+10\be}r_{s+1}^{-1/2}}L_{a,O_{s+1}}\big|.
\end{equation}
\medskip

At step $s$, we have successfully defined the refined cell set $\bar\co_s$ and the refined tube sets $\{\bar\ZT_{O_s}[r_{s-1}]\}_{O_s\in\co_s}$ in \eqref{base-refined-tubes}. Also, \eqref{l2-base-case} and \eqref{base-tangent} are the base estimates for our backward algorithm. Generally, assume that we have constructed the refined cell set $\bar\co_{u+1}$ and the refined tube sets $\{\bar\ZT_{O_{u+1}}[r_{u+1}]\}_{O_{u+1}\in\co_{u+1}}$, and proved \eqref{cell1}, \eqref{cell2} for step $u+1$. Now at every intermediate step $u$, we consider two cases.

\subsection{Cellular state at step \texorpdfstring{$u+1$}{Lg}} \hfill

Suppose we are in the situation that \textup{STATE}$(u+1)=$\textup{trans}. Fix a cell $O_u\in\co_u$. To build  $\bar\ZT_{O_u}[r_{u-1}]$, our idea is to study the relations between bigger tubes $\ZT_{O_{u}}[r_{u-1}]$ and smaller tubes $\{\bar\ZT_{O_{u+1}}[r_{u}]\}_{O_{u+1}<O_{u}}$. Define
\begin{equation}
    \bar{\cO}_{u+1}(O_u):=\{ O_{u+1}\in\bar{\cO}_{u+1}:O_{u+1}<O_u \}.
\end{equation}
By induction hypothesis, we have
\begin{equation}
    \vh_{O_{u+1},\om_{u+1}}=\vp_{O_{u+1}}\sum_{\om_{u}<\om_{u+1}}\sum_{T\in\bar\ZT_{O_{u+1},\om_{u}}[r_{u}]}\vf_{O_{u},\om_{u}}\Id_{T}^\ast.
\end{equation}
Since we are in the cell case, the function $\vp_{O_{u+1}}$ is a smooth cutoff function of a rescaled ball of radius $r_{u+1}$. So via the local $L^2$ estimate \eqref{local-l2} and the $L^2$ orthogonality \eqref{L2-orthogonality-small-1}, one gets
\begin{align}
    \|\Sq\vh_{O_{u+1},\om_{u+1}}\|_2^2&\lesssim d^{-1}\int\Sq\Big(\sum_{\om_{u}<\om_{u+1}}\sum_{T\in\bar\ZT_{O_{u+1},\om_{u}}[r_{u}]}\vf_{O_{u},\om_{u}}\Id_{T}^\ast\Big)^2\\ \nonumber
    &\lesssim d^{-1}\sum_{\om_{u}<\om_{u+1}}\sum_{T\in\bar\ZT_{O_{u+1},\om_{u}}[r_{u}]}\int|\Id_T^\ast\Sq\vf_{O_{u},\om_{u}}|^2.
\end{align}
Summing over all directional caps in $\Om_{u+1}$ and all cells in $\bar\co_{u+1}(\co_u)$, we have
\begin{align}
\label{cell-backward-1}
    &\sum_{O\in\bar\co_{u+1}(O_u)}\sum_{\om\in\Om_{u+1}}\|\Sq\vh_{O,\om}\|_2^2\\ \nonumber
    &\lesssim d^{-1}\sum_{\om_u\in\Om_u}\sum_{O\in\bar\co_{u+1}(O_u)}\sum_{T\in\bar\ZT_{O,\om_u}[r_{u}]}\int|\Id_T^\ast\Sq\vf_{O_{u},\om_u}|^2.
\end{align}
Now recall the inductive formula \eqref{inductiveformula1}
\begin{equation}
\nonumber
    \vf_{O_{u},\om_{u}}:=\vp_{O_u}\sum_{\om_{u-1}<\om_{u}}\sum_{T\in\ZT_{O_{u},\om_{u-1}}[r_{u-1}]}\vf_{O_{u-1},\om_{u-1}}\Id_{T}^\ast.
\end{equation}
We plug it back to \eqref{cell-backward-1} so that
\begin{align}
\label{cell-backward-2}
    \eqref{cell-backward-1}\lesssim &\,d^{-1}\sum_{\om_u\in\Om_u}\sum_{O\in\bar\co_{u+1}(O_u)}\sum_{T\in\bar\ZT_{O,\om_u}[r_{u}]}\\ \nonumber &\int\Sq\Big(\vp_{O_u}\sum_{\om_{u-1}<\om_{u}}\sum_{T'\in\ZT_{O_{u},\om_{u-1}}[r_{u-1}]}\vf_{O_{u-1},\om_{u-1}}\Id_{T'}^\ast\Big)^2|\Id_T^\ast|^2.    
\end{align}

Let us take a look at the first line of $\eqref{cell-backward-2}$. Note that a bigger tube $T'\in\ZT_{O_u,\om_{u-1}}[r_{u-1}]$ would only make contribution if there exists a smaller tube $T\in\cup_{\om_u}\cup_{O\in\bar\co_{u+1}}\bar\ZT_{O,\om_u}[r_u]$, such that $T\cap T'\not=\varnothing$. Thus, if the collections of smaller tubes $\{\bar\ZT_{O,\om_u}[r_u]\}_{O\in\bar\co_{u+1}(O_u)}$ are highly overlapped, then there fewer larger tubes in  $\ZT_{O_u,\om_{u-1}}[r_{u-1}]$ would make contribution, which yields a better support estimate in \eqref{cell2}. On the contrary, if the sets $\{\bar\ZT_{O,\om_u}[r_u]\}_{O\in\bar\co_{u+1}(O_u)}$ only overlap a little, then there is an immediate gain in \eqref{cell-backward-2} when counting the multiplicity of small tubes. Since cells in $\bar\co_{u+1}(O_u)$ are all children of $O_u$, and since we had cell case in step $u$, any tube $T\in\ZT_{O_u}[r_u]$ belongs to $d+1$ sets in $\{\bar\ZT_{O,\om_u}[r_u]\}_{O\in\bar\co_{u+1}(O_u)}$, which gives an upper bound of the multiplicity. 

To provide a rigorous argument, we sort the bigger tubes in $\ZT_{O_u}[r_{u-1}]$ via the following definition.

\begin{definition}\label{defsorttube}
For a dyadic integer $v$, $1\leq v\leq O(d)$, we define a subcollection $\ZT_{O_u}^v[r_{u-1}]\subset\ZT_{O_u}[r_{u-1}]$ as
\begin{equation}
\nonumber
    \ZT_{O_u}^v[r_{u-1}]:=\{T': \#\{O_{u+1}: T'>T{\rm{~for~some}}\  T\in\ZT_{O}[r_u],O\in\bar\co_{u+1}(O_u)\}\sim v \}.
\end{equation}
As usual, we also define
\begin{equation}
    \ZT_{O_u,\om_{u-1}}^v[r_{u-1}]:=\ZT_{O_u}^v[r_{u-1}]\cap \ZT_{O_u,\om_{u-1}}[r_{u-1}].
\end{equation}
\end{definition}

\begin{remark}
\rm 
The readers can check that
\begin{equation}
    \ZT_{O_u}[r_{u-1}]=\bigcup_v\ZT_{O_u}^v[r_{u-1}],\hspace{3mm}\ZT_{O_u,\om_{u-1}}[r_{u-1}]=\bigcup_v\ZT_{O_u,\om_{u-1}}^v[r_{u-1}]
\end{equation}
and  
\begin{align}
\label{tuberelation}
    \ZT_{O_u,\om_{u-1}}^v[r_{u-1}]:=\{T':& \#\{O_{u+1}: T'>T{\rm{~for~some}}\  T\in\ZT_{O_{u+1},\om_u}[r_u],\\ \nonumber
    &O_{u+1}\in\bar\co_{u+1}(O_u)\}\sim v \}.
\end{align}

\end{remark}

Let us return to \eqref{cell-backward-2}. By pigeonholing and the triangle inequality, there is a dyadic number $v=v_{O_u}$ such that
\begin{align}
\label{cell-backward-3}
    \eqref{cell-backward-1}\lesssim &\,d^{-1}(\log R)^2\sum_{\om_u\in\Om_u}\sum_{O\in\bar\co_{u+1}(O_u)}\sum_{T\in\bar\ZT_{O,\om_u}[r_{u}]}\\ \nonumber &\int\Sq\Big(\vp_{O_u}\sum_{\om_{u-1}<\om_{u}}\sum_{T'\in\ZT_{O_{u},\om_{u-1}}^v[r_{u-1}]}\vf_{O_{u-1},\om_{u-1}}\Id_{T'}^\ast\Big)^2|\Id_T^\ast|^2.
\end{align}
After changing the summation and the integration, we can rewrite the summation in \eqref{cell-backward-3} as
\begin{align}
\label{cell-backward-4}
    \sum_{\om_u\in\Om_u}&\int\Sq\Big(\vp_{O_u}\sum_{\om_{u-1}<\om_{u}}\sum_{T'\in\ZT_{O_{u},\om_{u-1}}^v[r_{u-1}]}\vf_{O_{u-1},\om_{u-1}}\Id_{T'}^\ast\Big)^2\\ \label{cell-backward-5} &\Big(\sum_{O\in\bar\co_{u+1}(O_u)}\sum_{T\in\bar\ZT_{O,\om_u}[r_{u}]}|\Id_T^\ast|^2\Big).
\end{align}
We will show 
\begin{align}
\label{cell-backward-6}
    \eqref{cell-backward-4}\lesssim&\, v\sum_{\om_u\in\Om_u}
    \int\Sq\Big(\vp_{O_u}\sum_{\om_{u-1}<\om_{u}}\sum_{T\in\ZT_{O_{u},\om_{u-1}}^v[r_{u-1}]}\vf_{O_{u-1},\om_{u-1}}\Id_{T}^\ast\Big)^2\\ \nonumber
    &+\rap(R)\|\Sq\vf_{\cp_R}\|_2^2.
\end{align}
To do this, we make two observations:

{\emph {First observation:}}  Note that for $T'$ and $T$ in the integral \eqref{cell-backward-4}, if $T\cap 10T'=\varnothing$, then $\Id^*_{T'}\Id^*_T=\rap(R)$. This means that $\Id^*_{T'}\Id^*_T$ is negligible unless $T<T'$. Since the square function in \eqref{cell-backward-4} is bounded above by $\Sq\vf_{\cp_R}$, by adding a rapidly decreasing factor, we can discard those $T$'s that have negligible contribution to the integral. Hence the remaining $T$'s are those satisfying:
\begin{equation}
\nonumber
    T<T'\ \textup{for some}\ T'\in\ZT_{O_{u},\om_{u-1}}^v[r_{u-1}].
\end{equation}

{\emph{Second observation:}} We want to rewrite the sum in \eqref{cell-backward-5} for the remaining $T$'s. Note that for an $r_u$ tube $T$, it may belong to many $\bar\ZT_{O,\om_u}[r_{u}]$ for different $O\in\bar\co_{u+1}(O_u)$. This means such $T$ can appear many times in the sum. Let $n(T):=\#\{ O\in\bar\cO_{u+1}(O_u):T\in\bar\ZT_{O,\om_u}[r_{u}] \}$. Then, recalling \eqref{scale-rho-tube-set}, one has
\begin{equation}
    \sum_{O\in\bar\co_{u+1}(O_u)}\sum_{T\in\bar\ZT_{O,\om_u}[r_{u}]}|\Id_T^\ast|^2\leq\sum_{T\in\ZT_{\om_u}[r_u]}n(T)|\Id^*_T|^2\lesssim \sup_{T\in\ZT_{\om_u}}n(T). 
\end{equation}
From the first observation, one has that $T<T'$ for some $T'\in \ZT_{O_{u},\om_{u-1}}^v[r_{u-1}]$. This yields $\sup_{T\in\ZT_{\om_u}}n(T)\lesssim v$ due to \eqref{tuberelation}, and hence proves \eqref{cell-backward-6}.

\smallskip

Recall \eqref{cell-backward-1}, \eqref{cell-backward-2}, \eqref{cell-backward-3} and \eqref{cell-backward-4}. By pigeonholing again, we can find a dyadic number $v=v_{u+1}$ and hence a refined collection of step $u$ cells $\bar{\cO}_u=\{O_u\in\co_u:v_{O_u}=v\}$ such that
\begin{align}
    \sum_{O\in\bar\co_{u+1}}\sum_{\om\in\Om_{u+1}}\|\Sq\vh_{O,\om}\|_2^2\lesssim&\, vd^{-1}(\log R)^3\sum_{O\in\bar\co_u}\sum_{\om\in\Om_{u}}\|\Sq\vh_{O,\om}\|_2^2\\ \nonumber
    &+\rap(R)\|\Sq\vf_{\cp_R}\|_2^2.
\end{align}
Here for each $O_u\in\bar{\cO}_u$, we define the refined set of tubes
\begin{equation}
\nonumber
    \bar\ZT_{O_{u},\om_{u-1}}[r_{u-1}]:=\ZT_{O_{u},\om_{u-1}}^{v_{u+1}}[r_{u-1}],
\end{equation}
which therefore define the auxiliary vector $\vh_{O_u,\om_u}$ as
\begin{equation}
\nonumber
    \vh_{O_u,\om_u}:=\vp_{O_u}\sum_{\om_{u-1}<\om_{u}}\sum_{T\in\ZT_{O_{u},\om_{u-1}}^v[r_{u-1}]}\vf_{O_{u-1},\om_{u-1}}\Id_{T}^\ast.
\end{equation}
This proves the first part \eqref{cell1}.

\medskip

Next, we prove \eqref{cell2}. Actually we will prove the following stronger inequality:
\begin{equation}
\nonumber
    \max_{O_u\in\bar{\cO}_u}\!\!\big|N_{R^{1+10\be}r_{u+1}^{-1/2}}L_{a,O_{u}}\big|\lesssim R^{O(\be)} d^3v_{u+1}^{-1}\!\!\max_{O_{u+1}\in\bar{\cO}_{u+1}}\big|N_{R^{1+10\be}r_{u+1}^{-1/2}}L_{a,O_{u+1}}\big|,
\end{equation}
where the $N_{R^{1+10\be}r_{u}^{-1/2}}$ on the left hand side of \eqref{cell2} is replaced by $N_{R^{1+10\be}r_{u+1}^{-1/2}}$.

Fix a $O_u\in\bar{\cO}_u$. We choose a maximal collection of $Rr_{u+1}^{-1/2}$ separated points in $L_{a,O_u}$, denoted by $\{y_l\}_{l=1}^{m}$. For each $y_l$, we can pick one tube $T_l\in \bar\ZT_{O_{u}}[r_{u-1}]=\ZT_{O_{u}}^{v_{u+1}}[r_{u-1}]$ so that $y_l\in\mathring T_l\cap L_a$. By the definition of $\ZT_{O_{u}}^{v_{u+1}}[r_{u-1}]$, for each $T_l$ we can find $\sim v_{u+1}$ sets $\{ \bar\ZT_{O_{u+1}}[r_{u}]:O_{u+1}<O_u,~O_{u+1}\in\bar\co_{u+1} \}$ and a tube $T_{l,j}$ in each set $\bar\ZT_{O_{u+1}}[r_{u}]$, such that $T_{l,j}<T_l$ for $j=1,\cdots,O(v_{u+1})$. Pick a point $z_{l,j}$ in $L_a\cap \mathring T_{l,j}$. Now the total number of points we picked is $\#\{z_{l,j}\}\sim mv_{u+1}$. We will prove
\begin{equation}
\label{pointestimate}
    \#\{z_{l,j}\}\lesssim d^3\frac{R^{O(\be)}}{(Rr_{u+1}^{-1/2})^2}\max_{O_{u+1}\in\bar{\cO}_{u+1}}\big|N_{R^{1+10\be}r_{u+1}^{-1/2}}L_{a,O_{u+1}}\big|. 
\end{equation}
This immediately implies
\begin{equation}
\nonumber
    \max_{O_u\in\bar{\cO}_u}\!\!\big|N_{10Rr_{u+1}^{-1/2}}L_{a,O_{u}}\big|\lesssim m(Rr_{u+1}^{-1/2})^2\lesssim R^{O(\be)}d^3v_{u+1}^{-1}\max_{O_{u+1}\in\bar{\cO}_{u+1}}\big|N_{10Rr_{u+1}^{-1/2}}L_{a,O_{u+1}}\big|
\end{equation}
as desired.

To prove \eqref{pointestimate}, we define for each cell $O_{u+1}\in \bar{\cO}_{u+1}(O_u)$ a set $Z_{O_{u+1}}$ consisting of points $z_{l,j}$ whose associated tube $T_{l,j}$ belongs to $\bar\ZT_{O_{u+1}}[r_{u}]$:
\begin{equation}
    Z_{O_{u+1}}=\{ z_{l,j}:T_{l,j}\in\bar\ZT_{O_{u+1}}[r_{u}] \}.
\end{equation}
A crucial fact is that the points in a single set $Z_{O_{u+1}}$ have different subscripts $l$, since they come from different $T_l$ by definition. We fix the point set $Z_{O_{u+1}}$. Via the assumption $T_{l,j}<T_l$ one has $\mathring T_{l,j}\cap \cp_{10R^{1+\be}}\subset (100R^{1+\be}r_{u+1}^{-1/2})\mathring T_{l}$. Also, since the sets $\{\mathring T_l\cap L_a\}$ are $Rr_{u+1}^{-1/2}$ separated, we get that the set $\{ \mathring T_{l,j}\cap L_a: z_{l,j}\in Z_{O_{u+1}} \}$ are $\sim R^{1-O(\be)}r_{u+1}^{-1/2}$ separated. Hence, 
\begin{equation}
    \#Z_{O_{u+1}}\sim \frac{R^{O(\be)}}{(Rr_{u+1}^{-1/2})^2}\Big|N_{Rr_{u+1}^{-1/2}}\Big(L_a\bigcap\Big\{\bigcup_{T_{l,j}: z_{l,j}\in Z_{O_{u+1}}}\mathring T_{l,j}\Big\}\Big)\Big|,
\end{equation}
which is bounded above by
\begin{equation}
\nonumber
    \frac{R^{O(\be)}}{(Rr_{u+1}^{-1/2})^2}|N_{Rr_{u+1}^{-1/2}}\Big(L_a\bigcap\Big\{\bigcup_{T\in\bar\ZT_{O_{u+1}}[r_{u}]}\mathring T\Big\}\Big)\Big|\leq\frac{R^{O(\be)}}{(Rr_{u+1}^{-1/2})^2}\big|N_{R^{1+10\be}r_{u+1}^{-1/2}}L_{a,O_{u+1}}\big| 
\end{equation}

Finally, note that there are $O(d^3)$ many cells in $\bar{\cO}_{u+1}(O_u)$. Since
\begin{equation}
    \#\{z_{l,j}\}\leq\sum_{O_{u+1}\in \bar{\cO}_{u+1}(O_u)}\# Z_{O_{u+1}},
\end{equation}
we can sum up $\# Z_{O_{u+1}}$ for every cell $O_{u+1}\in \bar{\cO}_{u+1}(O_u)$ using the above estimates to prove \eqref{pointestimate}. \qed

\medskip

Next, we consider the transverse state.
\subsection{Transverse state at step \texorpdfstring{$u+1$}{Lg}} \hfill

Suppose we are in the situation that \textup{STATE}$(u+1)=$\textup{trans}. So any $O_u\in\cO_{u}$ has $\sim \mu_{u+1}$ children $O_{u+1}\in\co_{u+1}$.

We define $\bar\co_u:=\co_u$, $\bar\ZT_{O_u}[r_{u-1}]:=\ZT_{O_u}[r_{u-1}]$, and hence defined the auxiliary vector $\vh_{O_u,\om_u}$ as
\begin{equation}
    \vh_{O_{u},\om_{u}}:=\vp_{O_u}\sum_{\om_{u-1}<\om_{u}}\sum_{T\in\bar\ZT_{O_{u},\om_{u-1}}[r_{u-1}]}\vf_{O_{u-1},\om_{u-1}}\Id_{T}^\ast,
\end{equation}
which is just $\vf_{O_u,\om_u}$.

Since the wave packets summed in $\vh_{O_{u+1},\om_{u+1}}$ is a subset of the wave packets of summed in $\vf_{O_{u+1},\om_{u+1}}$, we have
\begin{equation}
    \sum_{O\in \bar{\co}_{u+1}}\sum_{\om\in\Om_{u+1}}\|\Sq\vh_{O,\om}\|_2^2\le \sum_{O\in \bar{\co}_{u+1}}\sum_{\om\in\Om_{u+1}}\|\Sq\vf_{O,\om}\|_2^2.
\end{equation}
Combining it with \eqref{ittransverse2}, we get
\begin{equation}
    \sum_{O\in \bar{\co}_{u+1}}\sum_{\om\in\Om_{u+1}}\|\Sq\vh_{O,\om}\|_2^2\lesssim {\rm{Poly}}(d)R^{-\de}\sum_{O\in\bar{\cO}_{u}}\sum_{\om\in\Om_u}\|\Sq\vh_{O,\om}\|_2^2.
\end{equation}
We can also trivially obtain 
\begin{equation}
    \max_{O_u\in\bar{\cO}_u}\big|N_{R^{1+10\be}r_{u}^{-1/2}}L_{a,O_{u}}\big|\lesssim R^{O(\be)} \mu_u\max_{O_{u+1}\in\bar{\cO}_{u+1}}\big|N_{R^{1+10\be}r_{u+1}^{-1/2}}L_{a,O_{u+1}}\big|
\end{equation}
similarly as we did in the cell case. One just need to notes that each $O_u$ have $ O(\mu_u)$ many children $O_{u+1}\in\bar{\cO}_{u+1}$.

\subsection{Conclusion}\hfill

The backward algorithm stops when $u=1$, from which we obtain the refinements $\bar{\cO}_1\subset \cO_1$, $\{\bar\ZT_{O_1}[R]\}_{O_1\in\bar{\cO}_1}$ and $\{\vh_{O_1,\om_1}\}_{O_1\in\bar{\cO}_1}$. Set
\begin{equation}\label{v}
     v=\prod_uv_{u+1}\lesssim d^{s_c}, 
\end{equation}
where the product is taken over all $v_{u+1}$ in the cell cases.

Note that $(\log R)^{O(s)}R^{O(s_c\be )}\lesssim R^{O(\de)}$ and $|\co_u|=O(R^3)$ for any $1\leq u\leq s+1$. Iterate \eqref{cell1} and \eqref{trans1} and use the base estimate \eqref{l2-base-case} to have
\begin{align}
    \sum_{O\in \bar{\co}_{s+1}}\sum_{\om\in\Om_{s}}\|\Sq\vf_{O,\om}\|_2^2\lesssim&\, R^{O(\de)}d^{-s_c}R^{-s_t\de}v\sum_{O\in\bar{\cO}_{1}}\sum_{\om\in\Om_1}\|\Sq\vh_{O,\om}\|_2^2\\ \nonumber
    &+\rap(R)\|\Sq\vf_{\cp_R}\|_2^2.
\end{align}
Since there are $R^{O(\de)}$ many cells in $\bar\co_1$, by pigeonholing, there exists a cell $O_1\in\bar\co_1$ such that
\begin{align}
\label{ftoh}
    \sum_{O\in \bar{\co}_{s+1}}\sum_{\om\in\Om_{s}}\|\Sq(\vf_{O,\om})\|_2^2\lesssim&\,  R^{O(\de)}d^{-s_c}R^{-s_t\de}v\sum_{\om\in\Om_1}\|\Sq(\vh_{O_1,\om})\|_2^2\\ \label{ftohrapid}
    &+\rap(R)\|\Sq\vf_{\cp_R}\|_2^2.
\end{align}
Fix this cell $O_1$ and set $\bar\ZT=\bar\ZT_{O_1}[R]$. Via \eqref{formula-for-h-2}, we can write  $\vh_{O_1,\om_1}$ as

\begin{equation}
\label{formulaofh}
    \vh_{O_1,\om_1}:=\vp_{O_1}\sum_{\theta<\om_{1}}\sum_{T\in\bar\ZT_\theta}\vf_{\theta}\Id_{T}^\ast.
\end{equation}

\eqref{ftoh} is the $L^2$ estimate we obtain from the backward algorithm. For the support estimate, define
\begin{equation}
    L_{a,\cp_R}:=L_a\bigcap\Big\{ \bigcup_{T\in \bar\T}T \Big\}. 
\end{equation}
Iterate \eqref{cell2}, \eqref{trans2} and use \eqref{base-tangent} for the base case to have
\begin{equation}
\label{backward-supp-1}
    |N_{R^{10\be+1/2}}L_{a,\cp_R}|\lesssim R^{O(\de)}d^{3s_c}\prod_{u}\mu_uv^{-1}\max_{O\in\cO_{s+1}}\big|N_{R^{1+10\be}r_{s+1}^{-1/2}}L_{a,O}\big|.
\end{equation}
We need the following lemma to estimate $|N_{R^{1+10\be}r_{s+1}^{-1/2}}L_{a,O}\big|$. Its proof is postponed to the appendix.
\begin{lemma}
\label{tangenttubelemma}
Fix any cell $O_{s+1}\in \bar O_{s+1}$. Recall $L_{a,O_{s+1}}=L_a\bigcap \{\cup_{T\in\ZT_{O_{s+1}}[r_s]}\mathring T\}$, and set $r=r_s$. Then uniform for every $a\in[-R^{1+10\be},R^{1+10\be}]$, one has
\begin{equation}
    \big|N_{R^{1+10\be}r_{s+1}^{-1/2}}L_{a,O}\big|\lesssim R^{O(\de)}R^2r^{-1/2}.
\end{equation}
\end{lemma}

Recalling \eqref{scenario2.4}, we use Lemma \ref{tangenttubelemma} to bound \eqref{backward-supp-1} so that
\begin{align}
    |N_{R^{10\be+1/2}}L_{a,\cp_R}|\lesssim M^{-1}R^{O(\de)}R^2r^{-1/2}|\co_{s+1}|v^{-1}.
\end{align}
It yields that, if define
\begin{equation}
    X=\bigcup_{T\in \bar\ZT} R^{5\be}T,
\end{equation}
then we can obtain the support estimate from the backward algorithm
\begin{equation}
\label{ineqofX}
    |X|\lesssim R^{O(\de)}R^3\min(M^{-2},M^{-1}r^{-1/2}|\cO_{s+1}|v^{-1}).
\end{equation}

\section{Concluding the proof}

We conclude the proof of Theorem \ref{local-broad-thm} in this section. Let us first estimate the right hand side of \eqref{ftoh}. By \eqref{formulaofh} and the triangle inequality, we have
\begin{align}
    \sum_{\om\in\Om_1}\|\Sq\vh_{O_1,\om}\|_2^2&\lesssim R^{O(\de)}\sum_{\theta\in\Theta}\sum_{T\in\bar\ZT_{\theta}}\int |\Sq\vf_{\theta}|^2\Id_{T}^\ast\\ \nonumber
    &=R^{O(\de)}\sum_{\theta\in\Theta}\sum_{j=1}^R\int |m_{j,\theta}\ast f|^2\Big(\sum_{T\in\bar\ZT_{\theta}}\Id_{T}^\ast\Big),
\end{align}
where the kernel $m_{j,\theta}$ was defined in \eqref{m-j-theta}. Via the kernel estimate \eqref{kernelestimate} and Definition \ref{adapt-def}, we see that the kernels  $\{R^{-O(\be)}m_{j,\theta}\}_j$ is adapt to the rectangular tube $T_{0,\theta}$, where $T_{0,\theta}$ is the $R^{1/2+\be}\times R^{1/2+\be}\times R^{1+\be}$ rectangular tube centered at the origin with direction $c_\theta$. Now via Lemma \ref{weightedlemma} and recalling \eqref{lattice-sqfcn}, \eqref{q-tau}, we can argue similarly as in Lemma \ref{half-lem} to obtain
\begin{align}
\label{concluding-1}
    \sum_{\om\in\Om_1}\|\Sq(\vh_{O_1,\om})\|_2^2&\lesssim R^{O(\de)}\sum_{\theta\in\Theta}\sum_{q\in\bq(\theta)}\int|\De_qf|^2 \wt w_{T_{0,\theta},N}\ast\Big(\sum_{T\in\bar\ZT_{\theta}}\Id_{T}^\ast\Big)\\ \nonumber
    &\lesssim R^{O(\de)}\sum_{\theta\in\Theta}\sum_{q\in\bq(\theta)}\int|\De_qf|^2 \Big(\sum_{T\in\bar\ZT_{\theta}}w_{T,N}\Big).
\end{align}
We define $X_\theta$ as 
\begin{equation}
    X_\theta:=\bigcup_{T\in \bar\ZT_\theta} R^{5\be}T,
\end{equation}
so that when $x\in \ZR^3\setminus X_\theta$, 
\begin{equation}
    \sum_{T\in\bar\ZT_{\theta}}w_{T,N}\leq \rap(R)w_{\cp_R}.
\end{equation}
Plug this back to \eqref{concluding-1} and use Lemma \ref{finiely-overlap-lattice-cube} to get
\begin{align}
\label{concluding-2}
    \sum_{\om\in\Om_1}\|\Sq\vh_{O_1,\om}\|_2^2\lesssim&\, R^{O(\de)}\sum_{\theta\in\Theta}\sum_{q\in\bq(\theta)}\int|\De_qf|^2\Id_{X_\theta}\\  \label{concluding-3}
    &+\rap(R)\int\sum_{q\in\bq}|\De_qf |^2w_{\cp_R}.
\end{align}
To bound the right hand side of \eqref{concluding-2}, we trivially bound $\Id_{X_\theta}$ by $\Id_X$ and use H\"older's inequality so that
\begin{align}
\label{concluding-4}
    \sum_{\theta\in\Theta}\sum_{q\in\bq(\theta)}\int|\De_qf |^2\Id_{X_\theta}&\lesssim\int\Big(\sum_{q\in\bq}|\De_qf |^2\Big)\Id_X\\ \nonumber
    &\lesssim\Big(\int\Big(\sum_{q\in\bq}|\De_qf |^2\Big)^{p/2}w_{{\cp_R}}\Big)^{2/p}|X|^{\frac{p-2}{p}}.
\end{align}
Note that $|X|\geq R^2$. The second term \eqref{concluding-3}, as well as the rapidly decreasing term $\rap(R)\|\Sq\vf_{\cp_R}\|_2^2$, can be estimated trivially as (For $\|\Sq\vf_{\cp_R}\|_2^2$ we need to argue similarly as in \eqref{end-small-auxiliary-2})
\begin{equation}
\label{concluding-5}
    \eqref{concluding-3},\eqref{ftohrapid}\lesssim\rap(R)\Big(\int\Big(\sum_{q\in\bq}|\De_qf |^2\Big)^{p/2}w_{{\cp_R}}\Big)^{2/p}|X|^{\frac{p-2}{p}}.
\end{equation}

Now let us move back to \eqref{ftoh}. Combining \eqref{concluding-1}, \eqref{concluding-2}, \eqref{concluding-3}, \eqref{concluding-4} and \eqref{concluding-5}, one can use pigeonholing to find a cell $O_{s+1}\in\co_{s+1}$ such that
\begin{align}
\label{concluding-6}
    \sum_{\om\in\Om_{s}}\|\Sq\vf_{O_{s+1},\om}\|_2^2\lesssim&\, R^{O(\de)}d^{-s_c}R^{-s_t\de}v|\co_{s+1}|^{-1}|X|^{\frac{p-2}{p}}\\ \nonumber
    &\Big(\int\Big(\sum_{q\in\bq}|\De_qf |^2\Big)^{p/2}w_{{\cp_R}}\Big)^{2/p}.
\end{align}
Recall \eqref{scenario1.1}. To bound $\|\Id_{O_{s+1}}{\rm Br}_{A/2^{a(s+1)}} \Sq\vf_{O_{s+1}}\|_p^p$, we need the following lemma, whose proof is postponed to next section.
\begin{lemma}[Bilinear estimate]
\label{bilinear-lem}
Let $O\in\co_{s+1}$ be any step $s+1$ cell (a tangent cell) and let $\vf_{O}$ be the associated vector. Then
\begin{equation}
\label{bilinear-esti}
    \int_O \bil[\Sq\vf_{O}]^p\lesssim R^{O(\de)}r^{-\frac{5(p-2)}{4}}M^{\frac{p-2}{2}}\Big(\sum_{\om\in\Om_s}\|\Sq\vf_{O,\om}\|_2^2\Big)^{p/2}.
  \end{equation}
\end{lemma}
Since we know from \eqref{scenario1.2} that $\Id_{O_{s+1}}{\rm Br}_{A/2^{a(s+1)}} \Sq(\vf_{O_{s+1}})$ has the morally the same $L^p$ norm. Together with Lemma \ref{broad-blinear}, \eqref{scenario1.1}, \eqref{concluding-6} and \eqref{bilinear-esti}, one gets
\begin{align}
\label{concluding-7}
    \int_{\cp_R}|{\rm Br}_{A} \Sq f|^p\lesssim&\, R^{O(\de)}|\cO_{s+1}|^{\frac{2-p}{2}}M^{\frac{p-2}{2}}r^{\frac{5}{2}-\frac{5p}{4}}(d^{-s_c}R^{-s_t\de}v)^{\frac{p}{2}}|X|^{\frac{p-2}{2}}\\ \nonumber
    &\int\Big(\sum_{q\in\bq}|\De_qf |^2\Big)^{p/2}w_{{\cp_R}}.
\end{align}
Since $r=R d^{-s_c}R^{-s_t\de}$ and since \eqref{ineqofX}, it remains to prove the following lemma so that one can conclude Theorem \ref{local-broad-thm} from \eqref{concluding-7}.
\begin{lemma}

Assuming the notation as above, we have
\begin{align}
\label{counting-esti}
    &|\cO_{s+1}|^{\frac{2-p}{2}}(R/r)^{p-3}v^{p/2}(\min(M^{-1}r^{1/2},|\cO_{s+1}|v^{-1}))^{\frac{p-2}{2}} \\ \nonumber
    &\lesssim R^{O(\de)} M^{6-2p}R^{p-3}. 
\end{align}
\end{lemma}
 
\begin{proof}
We consider two cases.
\subsection{\texorpdfstring{$M^{-1}r^{1/2}\ge | \cO_{s+1}|v^{-1}$}{Lg}}
This implies
\begin{equation}
    r\ge M^2| \cO_{s+1}|^2v^{-2}.
\end{equation}
Recall \eqref{scenario2.4}, we know that $| \cO_{s+1}|\gtrsim R^{-O(\de)}d^{3s_c}$, which implies.
\begin{equation}
    r R^{O(\de)}\ge M^2d^{6s_c}v^{-2}
\end{equation}
Thus, combining the fact $v\lesssim d^{s_c}$ from \eqref{v}, one can estimate \eqref{counting-esti} as
\begin{align}
    \eqref{counting-esti}&=(R/r)^{p-3}v\lesssim R^{O(\de)}M^{6-2p}R^{p-3}(d^{s_c})^{-6(p-3)}v^{2(p-3)+1}\\ \nonumber
    &\lesssim R^{O(\de)}M^{6-2p}R^{p-3}(d^{s_c})^{-6(p-3)+2(p-3)+1}\\ \nonumber
    &=R^{O(\de)}M^{6-2p}R^{p-3}(d^{s_c})^{13-4p}.
\end{align}
It implies \eqref{counting-esti} when $p\geq3.25$.

\subsection{\texorpdfstring{$M^{-1}r^{1/2}\le | \cO_{s+1}|v^{-1}$}{Lg}} Recalling \eqref{scenario1.4}, we first set 
\begin{equation}
    \mu=\prod_{\textup{STATE}(u)=\textup{trans}}\mu_u
\end{equation}
In this case, we have
\begin{align}
    \eqref{counting-esti}&=M^{-\frac{p-2}{2}}| \cO_{s+1}|^{\frac{2-p}{2}}(d^{s_c}R^{s_t\de})^{p-3}v^{\frac{p}{2}}r^{\frac{p-2}{4}}\\ \label{finalineq2}
    &\lesssim R^{O(\de)} M^{-\frac{p-2}{2}}R^{\frac{p-2}{4}}(d^{s_c})^{\frac{2-3p}{4}}(R^{s_t\de})^{\frac{3p-10}{4}}\mu^{\frac{2-p}{2}}v^{\frac{p}{2}}
\end{align}
after using $| \cO_{s+1}|\gtrsim R^{-O(\de)} d^{3s_c}\mu$ and $r=Rd^{-s_c}R^{-s_t\de}$. Note that the condition $M^{-1}r^{1/2}\le | \cO_{s+1}|v^{-1}$ implies
\begin{equation}
    R\le M^2d^{7s_c}R^{s_t\de}\mu^2v^{-2}.
\end{equation}
Taking the $\frac{3p-10}{4}$ power on both side (note that $\frac{3p-10}{4}<0$), we get
\begin{equation}
    R^{\frac{3p-10}{4}}\ge M^{\frac{3p-10}{2}}(d^{s_c})^{\frac{21p-70}{4}}(R^{s_t\de})^{\frac{3p-10}{4}}\mu^{\frac{3p-10}{2}}v^{\frac{-3p+10}{2}}.
\end{equation}
Plug this back to \eqref{finalineq2} so that
\begin{align}
    \eqref{finalineq2}&=R^{O(\de)}R^{p-3}M^{6-2p}(d^{s_c})^{18-6p}\mu^{6-2p}v^{2p-5}\\ \nonumber
    &\le R^{O(\de)}R^{p-3}M^{6-2p}(d^{s_c})^{18-6p}v^{2p-5}
\end{align}
when $p\geq3$. Finally, we plug in the fact $v\lesssim d^{s_c}$ into the above inequality and obtain
\begin{align}
    \eqref{finalineq2}\lesssim R^{O(\de)}R^{p-3}M^{6-2p}(d^{s_c})^{18-6p+2p-5}\le R^{O(\de)}R^{p-3}M^{6-2p}
\end{align}
when $p\geq3.25$.
\end{proof}

\section{A bilinear estimate}
To deal with the tangent case in Lemma \ref{bilinear-lem}, we need a slightly stronger bilinear estimate than the one in \cite{Guth-R3} Lemma 3.10. Specifically, let $\Pi_1$ and $\Pi_2$ be two smooth compact curves with positive curvature in $\ZR^2$, such that any of their normal vectors $n_1$, $n_2$ satisfy $\textup{Angle}(n_1,n_2)\sim (MK)^{-1}$. For $k=1,2$, let $W_k=\{w_k\}$ be a collection of disjoint rectangles of dimensions $\rho^{-1/2}\times M^{-1}\rho^{-1/2}$ contained in $N_{M^{-1}\rho^{-1/2}}(\Pi_j)$ with $\rho>100$ (See Figure \ref{bilinear-figure}). Then we have the following geometric lemma.
\begin{lemma}
\label{bilinear-geometric-lemma}
Each $w_{1}+w_{2}$ is contained in a $\rho^{-1/2}\times M^{-1}\rho^{-1/2}$-rectangle. Also, each $w_{1}+w_{2}$ can intersect only $O(K)$ many other sets. That is,
\begin{equation}
\label{L4-orthogonality}
    \sum_{(w_1,w_2)\in W_1\times W_2}\Id_{w_{1}+w_{2}}\lesssim K.
\end{equation}
\end{lemma}

In \cite{Guth-R3} Lemma 3.10, Guth proved Lemma \ref{bilinear-geometric-lemma} when $\Pi_1$ and $\Pi_2$ are both portions of a single smooth curve, and when $M=1$. The linear version of Lemma \ref{bilinear-geometric-lemma} was first observed in \cite{Carbery-MBR}.

\begin{proof}
After rotation and translation, without loss of generality, we can assume that all the normal vectors of $\Pi_1$ point to the left hand side of the vertical axis while all the normal vectors of $\Pi_2$ point to the right hand side. Also, each normal vector of either $\Pi_1$ or $\Pi_2$ makes an angle $\sim (MK)^{-1}$ with respect to the vertical axis. As a consequence, the directional vectors of $\Pi_1$ make an angle $\sim (MK)^{-1}$ with respect to the horizontal axis, pointing upward, while the directional vectors of $\Pi_2$ also make an angle $\sim (MK)^{-1}$ with respect to the horizontal axis, but pointing downward. At this point, each $w_j$ is indeed a horizontal rectangle of dimensions $\rho^{-1/2}\times M^{-1}\rho^{-1/2}$.

By the triangle inequality, it suffices to prove \eqref{L4-orthogonality} when its right hand side is replaced by an absolute constant, for any collection of $K$ separated rectangles in $W_k$. We let $c_{w_1}=([c_{w_1}]_1,[c_{w_1}]_2)$ and $c_{w_2}=([c_{w_2}]_1,[c_{w_2}]_2)$ be the center of $w_{1}$ and $w_{2}$ respectively. Then points in  $\{[c_{w_k}]_1\}$ are $K\rho^{-1/2}$ separated and points in $\{[c_{w_k}]_2\}$ are $M^{-1}\rho^{-1/2}$ separated, for any $k=1,2$. Also, for any $w_1,w_1'\in W_1$, $[c_{w_1}]_1>[c_{w_1'}]_1$ implies $[c_{w_1}']_1>[c_{w_1'}]_2$. Conversely,  for any $w_2,w_2'\in W_2$, $[c_{w_2}]_1>[c_{w_2'}]_1$ implies $[c_{w_2}']_1<[c_{w_2'}]_2$.


Clearly $w_1+w_2$ is contained in a $\rho^{-1/2}\times M^{-1}\rho^{-1/2}$ horizontal rectangle. Note that $(w_1+w_2)\cap (w_1'+w_2')\not=\varnothing$ only if 
\begin{equation}
    [c_{w_1}]_1+[c_{w_2}]_1=[c_{w_1'}]_1+[c_{w_2'}]_1+O(\rho^{-1/2})
\end{equation}
and
\begin{equation}
    [c_{w_1}]_2+[c_{w_2}]_2=[c_{w_1'}]_2+[c_{w_2'}]_2+O(M^{-1}\rho^{-1/2}).
\end{equation}
By the geometric observation on the sets $\{[c_{w_j}]_1\}$ and $\{[c_{w_j}]_2\}$ above, there are finitely many $(w_1',w_2')$ satisfying $(w_1+w_2)\cap (w_1'+w_2')\not=\varnothing$. This concludes the proof of the lemma.
\end{proof}

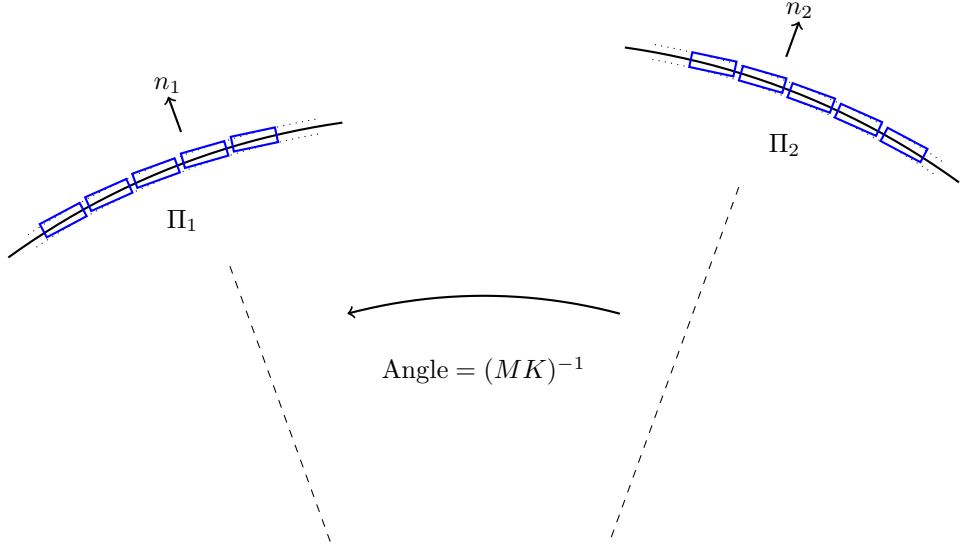
\begin{figure}
\begin{tikzpicture}

\draw [black,dotted,domain=100:124] plot ({-0.5+10*cos(\x)}, {-0.5+10*sin(\x)});
\draw [black,dotted,domain=100:124] plot ({-0.5+9.8*cos(\x)}, {-0.5+9.8*sin(\x)});
\draw [black,thick,domain=98:126] plot ({-0.5+9.9*cos(\x)}, {-0.5+9.9*sin(\x)});

\draw[color=blue,thick,rotate around={118:(-0.5,-0.5)}] (9.3,-0.3) rectangle (9.5,0.3);
\draw[color=blue,thick,rotate around={114:(-0.5,-0.5)}] (9.3,-0.3) rectangle (9.5,0.3);
\draw[color=blue,thick,rotate around={110:(-0.5,-0.5)}] (9.3,-0.3) rectangle (9.5,0.3);
\draw[color=blue,thick,rotate around={106:(-0.5,-0.5)}] (9.3,-0.3) rectangle (9.5,0.3);
\draw[color=blue,thick,rotate around={102:(-0.5,-0.5)}] (9.3,-0.3) rectangle (9.5,0.3);
\draw [->]  [color=black,thick,rotate around={110:(-0.5,-0.5)}] (9.8,-0.5) -- (10.3,-0.5);

\node at (-4,8) {$\Pi_1$};
\node at (-4.2,9.8) {$n_1$};

\draw [black,dotted,domain=56:80] plot ({0.5+10*cos(\x)}, {0.5+10*sin(\x)});
\draw [black,dotted,domain=56:80] plot ({0.5+9.8*cos(\x)}, {0.5+9.8*sin(\x)});
\draw [black,thick,domain=54:82] plot ({0.5+9.9*cos(\x)}, {0.5+9.9*sin(\x)});

\draw[color=blue,thick,rotate around={62:(0.5,0.5)}] (10.3,-0.3) rectangle (10.5,0.3);
\draw[color=blue,thick,rotate around={66:(0.5,0.5)}] (10.3,-0.3) rectangle (10.5,0.3);
\draw[color=blue,thick,rotate around={70:(0.5,0.5)}] (10.3,-0.3) rectangle (10.5,0.3);
\draw[color=blue,thick,rotate around={74:(0.5,0.5)}] (10.3,-0.3) rectangle (10.5,0.3);
\draw[color=blue,thick,rotate around={78:(0.5,0.5)}] (10.3,-0.3) rectangle (10.5,0.3);
\draw [->]  [color=black,thick,rotate around={70:(0.5,0.5)}] (10.8,0.5) -- (11.3,0.5);

\node at (4,9) {$\Pi_2$};
\node at (4.2,10.8) {$n_2$};

\draw[color=black,dashed,rotate around={70:(0.5,0.5)}] (4,0.5) -- (9,0.5);
\draw[color=black,dashed,rotate around={110:(-0.5,-0.5)}] (4,-0.5) -- (8,-0.5);
\draw [->] [black,thick,domain=75:105] plot ({7*cos(\x)}, {7*sin(\x)});
\node at (0,6) {$\textup{Angle}= (MK)^{-1}$};

\end{tikzpicture}
\caption{Rectangles in $(MK)^{-1}$ bilinear}
\label{bilinear-figure}
\end{figure}

Recall that for any step $s+1$ cell $O_{s+1}\in\co_{s+1}$, the associated function $\vf_{O_{s+1}}$ is defined in \eqref{inductiveformula1}. Also recall the definition for $\vf_{O_{s+1},\om_s}$ in \eqref{f_s+1}. Using Lemma \ref{bilinear-geometric-lemma}, we can prove Lemma \ref{bilinear-lem}. Let us recall this lemma below.

\begin{lemma}[Bilinear estimate]
\label{bilinear-lem-sec10}
Let $O\in\co_{s+1}$ be any step $s+1$ cell and let $\vf_{O}$ be the associated vector. Then
\begin{equation}
\label{bilinear-tangent}
    \int_O \bil[\Sq(\vf_{O})]^p\lesssim R^{O(\de)}r^{-\frac{5(p-2)}{4}}M^{\frac{p-2}{2}}\Big(\sum_{\om\in\Om_s}\big\|\Sq(\vf_{O,\om})\big\|_2^2\Big)^{p/2}.
  \end{equation}
\end{lemma}
{\bf Some notations:} Recall that the radius $r$ always stands for $r_s$. Until the end of this section, and letters $\om,\om_1,\om_2$ always represent $r^{-1/2}$-caps in $\ZS^2$. We remind readers that the definition of the bilinear operator $\bil$ is in \eqref{bilinear-def} and collection of $(MK)^{-1}$ caps $\cT=\cT_\si$ is defined above \eqref{sqfcn-3}, where $\si$ is a fixed $M^{-1}$ cap in $\ZS^2$. The set $\ZT_{O}[r_s]$ indeed is the set of tangent tubes $\ZT_{O,tang}$, which was defined in \eqref{tangent-definition}. For each cap $\tau\in\cT$, we use $\ZT_{O,\tau}$ to denote the collection of tubes in $\ZT_{O,\om}[r]$, with $\om\subset2\tau$.

\begin{remark}

\rm

It looks possible to the authors that one can prove Lemma \ref{bilinear-lem-sec10} via a variant of the two dimensional polynomial partitioning iteration, based on the idea in \cite{Guth-II}. The proof given here is a more direct one.

\end{remark}

\begin{proof}
Since \eqref{bilinear-tangent} is true for $p=2$, by H\"older's inequality, we only need to prove \eqref{bilinear-tangent} for the endpoint $p=4$. Since there are $O(K^2)$ many pairs $(\tau_1,\tau_2)$ in $\cT\times\cT$, it suffices to prove for a fixed pair $(\tau_1,\tau_2)$ that
\begin{equation}
\label{bilinear-tangent-2}
    \sum_{j_1,j_2}\int_O \Big|\sum_{T_1\in\ZT_{O,\tau_1}}f_{j_1,O,T_1}\Big|^2\Big|\sum_{T_2\in\ZT_{O,\tau_2}}f_{j_2,O,T_2}\Big|^2\lesssim R^{O(\de)}r^{-\frac{5}{2}}M\Big(\sum_{\om\in\Om_s}\big\|\Sq\vf_{O,\om}\big\|_2^2\Big)^2.
\end{equation}
Recall that $f_{j,O,T}=f_{j,O,\om}\Id^\ast_T$ (See \eqref{wave-packet-2}), where $\om\subset\ZS^2$ is the $r^{-1/2}$ cap dual to $T$. To save notations, we define $\ZT_1:=\ZT_{O,\tau_1}$, $\ZT_2:=\ZT_{O,\tau_2}$, $f_{j,\om}=f_{j,O,\om}$ and $f_{j,T}:=f_{j,O,T}$.

Let $\cq=\{Q\}$ be a collection of finitely overlapping $M$-rescaled balls of radius $r^{1/2+\be}$ that forms a cover of $O$ and $Q\cap \not=\varnothing$. For each $Q\in\cq$, pick a smooth function $\eta_Q$ such that the Fourier transform of $\eta_Q$ is supported in $\wh{Q}$, the dual slab of $Q$ centered at the origin. Since $\eta_Q$ decays rapidly outside $2Q$, we have 
\begin{equation}
\nonumber
    \sum_{j_1,j_2}\int_O \Big|\sum_{T_1\in\ZT_1}f_{j_1,T_1}\Big|^2\Big|\sum_{T_2\in\ZT_2}f_{j_2,T_2}\Big|^2\lesssim\sum_{j_1,j_2}\sum_{Q\in\cq}\int\eta_Q \Big|\sum_{T_1\in\ZT_1}f_{j_1,T_1}\Big|^2\Big|\sum_{T_2\in\ZT_2}f_{j_2,T_2}\Big|^2.
\end{equation}
Now for each $Q\in\cq'$ and each $k=1,2$, define $\ZT_{k,Q}$ as
\begin{equation}
    \ZT_{k,Q}:=\{T\in\ZT_{k}, T\cap 2Q\not=\varnothing\}
\end{equation}
and define $f_{j_k,Q}$ as the sum of wave packets $f_{j,T}$ that $T\in\ZT_{k,Q}$. Hence 
\begin{equation}
\label{restricted-to-Q}
    \int\eta_Q \Big|\sum_{T_1\in\ZT_1}f_{j_1,T_1}\Big|^2\Big|\sum_{T_2\in\ZT_2}f_{j_2,T_2}\Big|^2\lesssim\int\eta_Q \big|f_{j_1,Q}\big|^2\big|f_{j_2,Q}\big|^2.
\end{equation}

Tubes in $\ZT_{k,Q}$ are nearly coplanar. Indeed, by Lemma \ref{tangent-lemma}, there is a plane $V$ that tubes in $\ZT_{k,Q}$ are all contained in $N_{r^{1/2}R^{O(\de)}}(V)$. After rotation, we assume that $V$ is parallel to the plane $V_2:=\{x_2=0\}$. Recall that $\Ga_j$ is the surface defined in \eqref{surface}. If we let $\Si_j(\tau)\subset \Gamma_j$ be the pullback $\Si_j(\tau):=G_j^{-1}(V_2\cap\tau)$ where $G_j$ is the Gauss map defined in \eqref{Gauss-map}, then points on the curve $\Si_j(\tau)$ indeed satisfies the  equation
\begin{equation}
    \partial_{\xi_2}\Phi(\xi_1,\xi_2,t_j)=0.
\end{equation}
Employing the implicit function theorem, there is a smooth map $h:\ZR\to\ZR$ that $\partial_{\xi_2}\Phi(\xi,h(\xi),t_j)=0$, and because of \eqref{Phi-conditions}, 
\begin{equation}
    \partial h(\xi)=-\frac{\partial^2_{\xi_1\xi_2}\Phi(\xi,h(\xi);t_j)}{\partial^2_{\xi_2\xi_2}\Phi(\xi,h(\xi);t_j)}=O(1).
\end{equation}
Consider the projection map $P_2:(x_1,x_2,x_3)\to(x_1,x_3)$. The projected curve $P_2(\Si_j(\tau))$ can be parameterized as $\big(\xi,\Phi(\xi,h(\xi);t_j)\big)$. Hence, any of its normal vector can be written as
\begin{equation}
    \big(-\partial_{\xi_1}\Phi(\xi,h(\xi),t_j)-\partial h\cdot\partial_{\xi_2}\Phi(\xi,h(\xi),t_j),1\big),
\end{equation}
which equals to $\big(-\partial_{\xi_1}\Phi(\xi,h(\xi),t_j),1\big)$ since $\partial_{\xi_2}\Phi(\xi,h(\xi),t_j)=0$. This yields that any normal vectors $n_1,n_2$ of the projected curved $P_2(\Ga_{j,\tau_1})$ and $P_2(\Ga_{j,\tau_2})$ respectively are $(MK)^{-1}$ separated, because we already assume caps in $\cT$ are $8(MK)^{-1}$ separated below \eqref{sigma-sqfcn-2}. We in fact can check that both curves $P_2(\Ga_{j,\tau_1})$ and $P_2(\Ga_{j,\tau_2})$ have positive curvature.

Now for each $r^{-1/2}$-cap $\om$, we define $\ZT_{k,Q,\om}:=\{T\in\ZT_{k,Q}, T~\textup{dual~to}~\om\}$, so the right hand side of \eqref{restricted-to-Q} can be rewritten as
\begin{equation}
\label{before-L4-orthogonality}
    \int\eta_Q \Big|\sum_{\om_1\subset2\tau_1}\sum_{T\in\ZT_{1,Q,\om_1}}f_{j_1,T}\Big|^2\Big|\sum_{\om_2\subset2\tau_2}\sum_{T\in\ZT_{2,Q,\om_2}}f_{j_2,T}\Big|^2. 
\end{equation}
The Fourier transform of any $\sum_{T\in\ZT_{1,Q,\om_k}}f_{j_k,T}$ is contained in the $r^{-1/2}\times r^{-1/2}\times r^{-1}$-slab $S_{j_k}(\om_k)$, where the slab $S_j(\om)$ was introduced in Lemma \ref{Fourier-support-small-scale}. Since the caps $\om_k$ is contained in $N_{r^{-1/2}R^{O(\de)}}(V_k)$ if $\ZT_{k,Q,\om_k}\not=\varnothing$, the pullback $S_{j_k}(\om_k)$ is contained in the $r^{-1/2}R^{O(\de)}$ neighbourhood of the curve $\Ga_{j_k}(\tau_k)$. We can also check directly that the Fourier supports of both functions $\big(\sum_{T\in\ZT_{1,Q,\om_1}}f_{j_1,T}\big)\big(\sum_{T\in\ZT_{2,Q,\om_2}}f_{j_2,T}\big)$ and $\eta_Q\big(\sum_{T\in\ZT_{1,Q,\om_1}}f_{j_1,T}\big)\big(\sum_{T\in\ZT_{2,Q,\om_2}}f_{j_2,T}\big)$ are contained in $3(S_{j_1}(\om_1)+S_{j_2}(\om_2))$.

Define $w_k:=P_2(S_{j_k}(\om_k))$ for $k=1,2$, so that $w_k$ is contained in a $r^{-1/2}\times M^{-1}r^{-1/2}R^{O(\de)}$ rectangle that is further contained in the $M^{-1}r^{-1/2}R^{O(\de)}$ neighborhood of the two-dimensional curve $P_2(\Ga_{j_k,\tau})$. Rectangles in the set $W_k:=\{w_k\}$ does not overlap too much. Indeed, we first note that the slabs $\{S_{j_k}(\om)\}_\om$ are finitely overlapped. Then, since the curve $\Si_j(\tau)$ is parametrized as $\big(\xi,h(\xi),\Phi(\xi,h(\xi),t_j)\big)$ and since $\partial h=O(1)$, the pullback $P_2^{-1}(x)$ has measure $r^{-1/2}R^{O(\de)}$ for any $x\in P_2(N_{r^{-1/2}R^{O(\de)}}(\Ga_{j_k}(\om_k)))$. The fact that $S_{j_k}(\om_k)$ is contained in the thin neighborhood $N_{r^{-1/2}R^{O(\de)}}(\Ga_{j_k}(\om_k))$ thus implies
\begin{equation}
    \sum_{w\in W_k}\Id_w\lesssim R^{O(\de)}.
\end{equation}
Thus, we can apply Lemma \ref{bilinear-geometric-lemma} to conclude 
\begin{equation}
\label{after-L4-orthogonality}
    \eqref{before-L4-orthogonality}\lesssim R^{O(\de)}\int\eta_Q \sum_{\om_1,\om_2}\Big|\sum_{T\in\ZT_{1,Q,\om_1}}f_{j_1,T}\Big|^2\Big|\sum_{T\in\ZT_{2,Q,\om_2}}f_{j_2,T}\Big|^2,
\end{equation}
which, by the $L^2$ orthogonality \eqref{L2-orthogonality-small-1}, is bounded above by
\begin{equation}
     R^{O(\de)}\sum_{\om_1,\om_2}\sum_{T_1\in\ZT_{1,Q,\om_1}}\sum_{T_2\in\ZT_{2,Q,\om_2}}\int\eta_Q|f_{j_1,T_1}|^2|f_{j_2,T_2}|^2.
\end{equation}

What follows is a standard application of the essentially constant property of wave packets. For each pair $(f_{j_1,T_1}, f_{j_2,T_2})$, one has 
\begin{equation}
    \int\eta_Q|f_{j_1,T_1}|^2|f_{j_2,T_2}|^2\lesssim R^{O(\be)}r^{-\frac{5}{2}}M\|f_{j_1,T_1}\|_2^2\|f_{j_2,T_2}\|_2^2.
\end{equation}
Recall \eqref{restricted-to-Q}, \eqref{before-L4-orthogonality} and \eqref{after-L4-orthogonality}. We first sum up all $T_k\in\ZT_{k,Q,\om}$ in the above estimate, then we sum up all $\om_k\in\tau_k$ to have
\begin{equation}
\nonumber
    \int\eta_Q \Big|\sum_{T_1\in\ZT_1}f_{j_1,T_1}\Big|^2\Big|\sum_{T_2\in\ZT_2}f_{j_2,T_2}\Big|^2\lesssim R^{O(\de)}r^{-\frac{5}{2}}M\sum_{T_1\in\ZT_{1,Q}}\sum_{T_2\in\ZT_{2,Q}}\|f_{j_1,T_1}\|_2^2\|f_{j_2,T_2}\|_2^2
\end{equation}
Since each pair of tubes $(T_1,T_2)$ belongs to $O(1)$ many cross product $\ZT_{1,Q}\times\ZT_{2,Q}$, summing up all $Q\in\cq$ yields 
\begin{equation}
\nonumber
    \int\Big|\sum_{T_1\in\ZT_1}f_{j_1,T_1}\Big|^2\Big|\sum_{T_2\in\ZT_2}f_{j_2,T_2}\Big|^2\lesssim R^{O(\de)}r^{-\frac{5}{2}}M\sum_{T_1\in\ZT_{1}}\sum_{T_2\in\ZT_{2}}\|f_{j_1,T_1}\|_2^2\|f_{j_2,T_2}\|_2^2
\end{equation}
Finally, we sum up all pairs $(j_1,j_2)$ to conclude \eqref{bilinear-tangent-2}.
\end{proof}

\section{Appendix: intersection of tubes and plane \texorpdfstring{$L_a$}{Lg}}

We prove Lemma \ref{tangenttubelemma} here in the appendix. The proof is similar to the one in \cite{Wu}. First, recall that we set $r=r_s$, which is the scale of $O_s$. So the scale of tangent cells  $O_{s+1}$ is $r_{s+1}=r_sR^{-\de}$. Since $r_s$ and $r_{s+1}$ only differ by a factor $R^\de$ which is acceptable, it's safe to treat $r_s$ as $r_{s+1}$ or $r_{s+1}$ as $r_s$. Also recall that each $O_{s+1}$ lies in a rescaled ball $\cp_{r_{s+1}}$ of dimensions $M^{-1}r_{s+1}\times M^{-1}r_{s+1}\times r_{s+1}$. For each cell $O_{s}\in\cO_s$, we have done polynomial partitioning on $O_{s}$, and obtain a polynomial $P$ (depending on $O_s$) of degree $O(d)$. For any $O_{s+1}<O_s$, the collection $\T_{O_{s+1}}[r_s]$ was defined as the tubes that are tangent to $Z(P)$ in $O_{s+1}$.   

To obtain Lemma \ref{tangenttubelemma}, it suffices to prove the following lemma.
\begin{lemma}
Fix a $O_{s+1}\in  \cO_{s+1}$, so $O_{s+1}\subset \cp_{r_{s+1}}$. For each $T\in \T_{O_{s+1}}[r]$, let $\mathring T$ be the tube with infinite length which is obtained by prolonging $T$.  Let $\mathring \T_{O_{s+1}}[r]$ be these prolonged tubes. Let $\{y_l\}_{l=1}^m$ be a $R^{1+10\be}r^{-1/2}$-separated subset of 
$$ L_a\bigcap\{\cup_{ T\in\T_{O_{s+1}}[r]}\mathring T\} $$
Then $m\lesssim R^{O(\de)}r^{1/2}M^{-1}$ uniformly for all $|a|\le 2R^{1+10\beta}$.
\end{lemma}

Since the tiny factor $\be$ is harmless in our proof, let us assume $\be=0$ in the rest of this section.

\begin{proof}
For each point $y_l$, we pick a tube $\mathring T_l\in\mathring \T_{O_{s+1}}[r]$ satisfying $y_l\in\mathring T_l$. Let $ \mathring \T_{a}:=\{\mathring T_l\}_{l=1}^m$, and let $\T_a:=\{T_l\}_{l=1}^m$. We point out that if two distinct tubes $T_1, T_2\in \T_{a}$ intersect, then they make an angle $\gtrsim r^{-1/2}R^{-\de}$. Let $P$ be the polynomial that comes from the polynomial partitioning of $O_s$. We see that for any $T\in \T_{O_{s+1}}[r]$,  $T\cap \cp_{r_{s+1}}\subset N_{r^{1/2}R^{O(\de)}}Z(P)\cap \cp_{r_{s+1}}$. If we let $\cQ=\{Q\}$ to be a collection of finitely overlapping $r^{1/2}$-cube that cover $N_{r^{1/2}R^{O(\de)}}Z(P)\cap \cp_{r_{s+1}}$. By Wongkew's theorem \cite{Wongkew}, $\#\{Q\}\lesssim R^{O(\de)}rM^{-1}.$

For any $Q\in \cQ$, $T_1,T_2\in \T_{a}$, we define an incidence function $\chi(Q,T_1,T_2)$ which $=1$ if $Q\cap T_1\cap T_2\neq \varnothing$ and $=0$ otherwise. We let $n_Q$ be the number of tubes $T\in\T_{a}$ that intersect $Q$. By Cauchy-Schwartz,
$$ \sum_{Q,T_1,T_2}\chi(Q,T_1,T_2)\ge\sum_Q n_Q^2\ge \#\{Q\}^{-1}(\sum_Q n_Q)^2\ge R^{-O(\de)}M|\mathring \T_{a}|^2. $$
(The last inequality is because each $T$ contains $\sim r^{1/2}$ many $Q$'s.)

If we already have 
$$ |\mathring \T_{a}|\lesssim R^{O(\de)}r^{1/2}M^{-1}, $$
then we are done.
Otherwise we have $|\mathring \T_{a}|>C R^{O(\de)}r^{1/2}M^{-1}$, which implies a bound on the diagonal term:
$$ \sum_{Q,T}\chi(Q,T,T)\le r^{1/2}|\mathring \T_{a}|\le \frac{1}{2}R^{-O(\de)}M|\mathring \T_{a}|^2. $$
The last inequality holds since it is equivalent to
$$ |\mathring \T_{a}|\ge R^{O(\de)}2r^{1/2}M^{-1}. $$
So, we have
$$ \sum_Q\sum_{T_1\neq T_2}\chi(Q,T_1,T_2)\gtrsim R^{-O(\de)}M|\mathring \T_{a}|^2. $$
Notice that when $T_1\neq T_2$ and $\T_1\cap T_2\neq\emptyset$, we have $\angle(T_1,T_2)\ge cr^{-1/2}R^{-\de}$. Thus, there exists a dyadic value $\nu\in [cr^{-1/2}R^{-\de},M^{-1}]$ such that
$$
    \sum_Q\sum_{\angle(T_1,T_2)\sim \nu}\chi(Q,T_1,T_2)\gtrsim R^{-O(\de)}M|\mathring \T_{a}|^2.
$$
By pigeonholing, there exists a $T_1\in\T_{a}$ such that for this fixed $T_1$,
$$ \sum_Q\sum_{T_2:\ \angle(T_1,T_2)\sim \nu}\chi(Q,T_1,T_2)\gtrsim R^{-O(\de)}M|\mathring \T_{a}|. $$

From now on we fix this $T_1$ and define $\T_\nu:=\{ T\in\T_{a}:\angle(T_1,T)\sim\nu \}$. First we note that $\sum_Q\chi(Q,T_1,T_2)\lesssim \nu^{-1}$ for $T_2\in \T_{\nu}$, so we have
\begin{equation}\label{beforeclaim}
    |\T_\nu|\gtrsim\nu M R^{-O(\de)}|\mathring \T_{a}|.
\end{equation}

Next, we define $H:=\bigcup_{T\in \T_\nu} T\cap\cp_{r_{s+1}}$ which is a hairbrush rooted at $T_1$ (note that each $T\cap\cp_{r_{s+1}}$ is morally a $r^{1/2}\times r^{1/2}\times r_{s+1}$-tube). We claim that
\begin{equation}\label{claim}
    |\T_\nu|r^2\lesssim R^{O(\de)}(\log r)|H|.
\end{equation}

Let us quickly see how these two inequality combines to give the result. Since $H$ is contained in a fat tube of dimensions $C\nu r_{s+1}\times C\nu r_{s+1}\times Cr_{s+1}$, and since $H$ is contained in the set $N_{R^{O(\de)}r^{1/2}} Z(P)$. By Wongkew's theorem, we have
$$ |H|\lesssim R^{O(\de)}\nu r^{5/2}. $$
Together with \eqref{beforeclaim} and \eqref{claim}, we proved
$$ |\mathring \T_a|\lesssim R^{O(\de)}r^{1/2}M^{-1}. $$

To prove the claim \eqref{claim}, we use the idea of two-ends reduction from Wolff. We will decompose the set $\T_\nu$ into $\sim\nu r^{1/2}$ subsets. To do this, we choose $\sim \nu r^{1/2}$ many planes each of which contains the central line of $T_1$ and their normal vectors are $\nu^{-1}r^{-1/2}$ separated. Let $\{P_k\}_{k}$ be the $r^{1/2}$-neighborhood of these planes. We see that all the tubes in $\T_\nu$ lie in $\cup_k P_k$, since all the tubes in $\T_\nu$ lie in a cylinder of dimension $\nu r\times\nu r\times r$ and $\cup_k P_k$ covers this cylinder. For each tube $T\in \T_\nu$, we associate it to a $P_k$ if $T\in P_k$ (if there are many choice of $P_k$, we just choose one). Denoting by $\T_{\nu,k}$ the tubes that are associated to $P_k$, we get $\T_\nu=\sqcup_k \T_{\nu,k}$.

For each tube $T\in \T_{\nu,k}$, we let $\tilde{T}$ be the portion $(T\cap\cp_{r_{s+1}})\setminus N_{cr_{s+1}\nu}(T_1)$ Roughly speaking, $\tilde T$ contains two parts each of which is morally a tube of dimensions $r^{1/2}\times r^{1/2}\times r_{s+1}$ (the same dimensions as $T\cap \cp_{r_{s+1}}$). Let $\tilde{\T}_{\nu,k}$ be the collection of these $\tilde{T}$, and let $H_k$ be the union of tubes in $\tilde{T}_{\nu,k}$. By the separation of $P_k$, we see $\{H_k\}$ are at most $R^{O(\de)}$-overlapped. So it suffices to show 
\begin{equation}\label{lastone}
    |\T_{\nu,k}|r^2\lesssim R^{O(\de)}(\log r)|H_k|. 
\end{equation} 

To save notations, we let $\T=\T_{\nu,k}$,  $H=H_k$ and $P=P_k$ in the rest of the proof. For each $T\in\T$, the intersection of its stretch $\mathring T$ with $L_a$ contains a point $y_l$. By the geometric condition, the $\{y_l\}$ obtained from $\T$ lie in a $Rr^{-1/2}\times R$ rectangle in $L_a$. Since $y_l$'s are $Rr^{-1/2}$-separated, we can morally think about these $y_l$ are arranged in a line and two nearby $y_l$ are at least $Rr^{-1/2}$ separated. Now the following argument is quite standard as in the proof of $2$-dimensional Kakeya conjecture.

We choose $\T'$ to be a subset of $\T$ such that $\{\mathring T\cap L_a: T\in\T' \}$ are $C(\log r)R^{\de}Rr^{-1/2}$-separated and $|\T'|\gtrsim \frac{1}{(\log r)R^{\de}}|\T|$. For each $T\in\T'$, We have
$$ \sum_{T'\in \T',\ T'\neq T } |T \cap T'|\le r^2 \sum_{j=0}^{r}\frac{1}{Cj(\log r)R^{\de}}\le R^{-\de/2} r^{1/2} r^{1/2} r_{s+1}\le  \frac{1}{2}|\tilde{T}|. $$
This shows that $|\tilde{T}\setminus \cup_{T'\in \T',\ T'\neq T} T'|\ge \frac{1}{2}|\tilde{T}|$ for $T\in \T'$. As a result,
$$|H|\ge |\cup_{T\in \T'}\tilde{T}|\ge \sum_{T\in \T'}|\tilde{T}\setminus \cup_{T'\in \T',\ T'\neq T} T'|\gtrsim R^{-\de}r^2|\T'|\gtrsim \frac{R^{-\de}r^2}{\log r}|\T|, $$
which finishes the proof of \eqref{lastone}.
\end{proof}

\bibliographystyle{alpha}
\bibliography{bibli}

\end{document}